\documentclass[10pt]{article}
\usepackage[dvips]{graphicx} 
\usepackage[margin=2cm]{geometry}
\usepackage{amsmath,amsthm,comment}
\usepackage[colorlinks=true, pdfstartview=FitV,
linkcolor=blue, citecolor=blue, urlcolor=blue]{hyperref}
\usepackage{algorithm, algorithmic}
\usepackage{caption}
\usepackage{subcaption}
\usepackage{bbm}
\usepackage{natbib}
\usepackage{color} 

\usepackage[dvipsnames]{xcolor}
\usepackage{amssymb,fancyhdr}
\usepackage{url}
\usepackage{authblk}

\DeclareMathOperator*{\argmin}{arg\,min}

\newcommand{\Norm}[1]{\left\Vert#1\right\Vert}
\newcommand{\norm}[1]{\Vert#1\Vert}
\newcommand{\Abs}[1]{\left|#1\right|}
\newcommand{\abs}[1]{|#1|}
\newcommand{\set}[1]{\left\{#1\right\}}
\newcommand{\Real}{\mathbb R}

\def\ev{\mathbf e}

\def\Av{\mathbf A}
\def\Bv{\mathbf B}

\newcommand{\Ac}{\mathcal{A}}
\newcommand{\Bc}{\mathcal{B}}
\newcommand{\Cc}{\mathcal{C}}
\newcommand{\Dc}{\mathcal{D}}
\newcommand{\Ec}{\mathcal{E}}

\newcommand{\Lc}{\mathcal{L}}
\newcommand{\Mc}{\mathcal{M}}
\newcommand{\Nc}{\mathcal{N}}

\newcommand{\Qc}{\mathcal{Q}}
\newcommand{\Rc}{\mathcal{R}}
\newcommand{\Sc}{\mathcal{S}}

\newcommand{\Uc}{\mathcal{U}}

\newcommand{\Eb}{\mathbb{E}}

\newcommand{\Nb}{\mathbb{N}}

\newcommand{\Pb}{\mathbb{P}}

\newcommand{\Sb}{\mathbb{S}}

\usepackage{authblk}
\newcommand{\n}{\frac{1}{n}}
\renewcommand{\hat}{\widehat}

\newtheorem{theorem}{Theorem}
\newtheorem{assumption}{Assumption}
\newtheorem{intassumption}{Assumption}
\numberwithin{intassumption}{assumption}

\newtheorem{lemma}[theorem]{Lemma}
\newtheorem{remark}{Remark}
\newtheorem{proposition}[theorem]{Proposition}

\newtheorem{corollary}{Corollary}

\title{Inference for Projection Parameters in Linear Regression:\\ beyond $d = o(n^{1/2})$}
\author[1]{Woonyoung Chang}
\author[1]{Arun Kumar Kuchibhotla}
\author[2]{Alessandro Rinaldo}
\affil[1]{Department of Statistics \& Data Science, Carnegie Mellon University}
\affil[2]{Department of Statistics \& Data Science, The University of Texas at Austin}
\date{\today}
\begin{document}

\maketitle
\begin{abstract}
    We consider the problem of inference for projection parameters in linear regression with increasing dimensions. This problem has been studied under a variety of assumptions in the literature. The classical asymptotic normality result for the least squares estimator of the projection parameter only holds when the dimension $d$ of the covariates is of a smaller order than $n^{1/2}$, where $n$ is the sample size. Traditional sandwich estimator-based Wald intervals are asymptotically valid in this regime. In this work, we propose a bias correction for the least squares estimator and prove the asymptotic normality of the resulting debiased estimator. Precisely, we provide an explicit finite sample Berry Esseen bound on the Normal approximation to the law of the linear contrasts of the proposed estimator normalized by the sandwich standard error estimate. Our bound, under only finite moment conditions on covariates and errors, tends to 0 as long as $d = o(n^{2/3})$ up to the polylogarithmic factors. Furthermore, we leverage recent methods of statistical inference that do not require an estimator of the variance to perform asymptotically valid statistical inference and that leads to a sharper miscoverage control compared to Wald's. We provide a discussion of how our techniques can be generalized to increase the allowable range of $d$ even further.
\end{abstract}
\section{Introduction}
Linear regression is a fundamental statistical tool that has been widely used in various fields of research. The classical literature on linear regression studies the ordinary least square (OLS) estimation has focused primarily on the well-specified case, where the underlying truth postulates the linear relation between the response variable and covariates. As elucidated in works in an assumption-lean framework \citep{berk2019assumption,vansteelandt2022assumption}, although the model assumption sometimes takes account of the pre-knowledge, the usage of model assumptions is dishonest when used because of the mathematical convenience. As real data often possess a highly nonlinear structure, relying on model assumptions and considering them as representing ground truth in inference may be problematic. Compared to the popularity of parametric regression in both practical studies in statistics and econometrics, the theoretically established properties of regression coefficients in an assumption-lean framework have just started to gain attraction. To this effort, some approaches start with a traditional estimator of a parameter indexing a parametric regression model and then characterize what estimand the estimator converges to, without assuming that the model is true \citep{kuchibhotla2020berry, berk2019assumption}. In particular, we focus on misspecified linear regression models comprised of a $d$-dimensional random vector of covariates $X$ and a scalar random variable $Y$. If $\Pb_{X,Y}$ admits the second moment, then the conditional expectation $\Eb[Y|X]$, which is not necessarily linear, is the best $L^2$ approximation to $Y$ among functions of $X$. It is well known that the best linear $L_2$ approximation to $Y$ is the linear function $\beta^\top X$ is well-defined where the coefficients $\beta=\beta(\Pb_{X,Y})$ is given by
\begin{eqnarray*}
    \beta = \argmin_{\theta\in\Real^d}\Eb[(Y-\theta^\top X)^2].
\end{eqnarray*} Provided that the population Gram matrix $\Sigma = \Eb_X[XX^\top]$ is invertible, the solution is unique and given by the vector of \emph{projection} parameters \citep{buja2019models1,buja2019models2},
\begin{equation*}
    \beta = \Sigma^{-1}\Gamma,
\end{equation*} where $\Gamma=\Eb[XY]$. The projection parameter is traditionally estimated using the ordinary least square estimator (OLSE). Suppose that we observe a sample of $n$ i.i.d observations $(X_1,Y_1),\ldots,(X_n,Y_n)$ from $\Pb_{X,Y}$. Then, the OLSE is defined as
\begin{eqnarray*}
    \hat\beta = \argmin_{\theta\in\Real^d}\frac{1}{n}\sum_{i=1}^n(Y_i-\theta^\top X_i)^2.
\end{eqnarray*} Provided that the sample Gram matrix $\hat\Sigma:=n^{-1}\sum_{i=1}^nX_iX_i^\top$ is invertible with probability 1, then the OLSE is well-defined and can be expressed as 
\begin{eqnarray*}
    \hat\beta = \hat\Sigma^{-1}\hat\Gamma,
\end{eqnarray*} where $\hat\Gamma = n^{-1}\sum_{i=1}^n X_iY_i$.

\subsection{Related Works}
\paragraph{Asymptotic Normality and Berry Esseen bound for the Least Square Estimator}
In fixed-dimensional settings, the OLSE has been the subject of large sample theory and conventional Berry-Esseen bounds, demonstrated by \cite{van2000asymptotic} and \cite{pfanzagl1973accuracy}, respectively. With increasing dimensions, \cite{bickel1983bootstrapping} established the asymptotic normality of the least squares and the consistency of bootstrap, requiring $d=o(\sqrt{n})$ under a well-specified linear model with fixed covariates and homogeneous errors. Additionally, a series of papers by \cite{portnoy1984asymptotic,portnoy1985asymptotic,portnoy1986asymptotic,portnoy1988asymptotic} showcased various consistency and asymptotic normality properties for estimators, including M-estimators and maximum likelihood estimators. Translating these findings into our context mandates $d=o(\sqrt{n})$ to ensure the $\sqrt{n}$-consistency of the OLSE within a well-specified model under a set of arguably strong assumptions on the data-generating process. In a more general setting, albeit postulating at least a partially linear model, the asymptotic normality of (non-)studentized OLSE has been investigated with additional regularity conditions on covariates and errors \citep{DONALD199430,cattaneo2015treatment,cattaneo2018alternative,cattaneo2018inference}. This reference list is far from complete. Some of these works pertain to the regime $d=o(n)$ and occasionally allow the boundary case of $d=O(n)$.

An initial step toward the misspecified linear model with minimal assumptions on covariates and errors was made in \cite{mammen1993bootstrap}. They were concerned with linear contrasts of the regression coefficients and asserted the consistency of both the resampling bootstrap and the wild bootstrap in estimating the regression coefficient with increasing dimensions. They claimed that $d=o(n^{3/4})$ is sufficient for consistent resampling bootstrap and that $d=o(\sqrt{n})$ for the wild bootstrap. While their work provides useful intuition regarding the asymptotic behavior of the OLSE under model misspecification, their findings rely on heuristics and do not explicitly present the convergence rate. \cite{rinaldo2019bootstrapping} formulated the Berry Esseen bound for general nonlinear statistics under the misspecified setting. Their analysis, specifically applied to projection parameters, established a uniform Berry Esseen bound for entry-wise asymptotic normality of the Ordinary Least Square Estimator (OLSE), converging to $0$ under $d=o(n^{1/5})$, disregarding polylogarithmic factors (refer to Theorem~2 therein). More recently, \cite{kuchibhotla2020berry} derived a novel finite sample bound for the studentized OLSE under finite moment assumptions on covariates and errors. They obtained a uniform Berry Esseen bound for entry-wise asymptotic normality of the OLSE, with their bound scaling as $(d/n)^{1/3}+d/\sqrt{n}$, disregarding polylogarithmic factors and under sufficient moment conditions on covariates and errors; See Theorem~10 of \cite{kuchibhotla2020berry}. Consequently, they required $d=o(\sqrt{n})$ for the bound to vanish, facilitating the construction of simultaneous confidence intervals for the projection parameter coefficients under such dimension range. Notably, these intervals can exhibit a width of the order $1/\sqrt{n}$. It should be emphasized that the parameters of interest in this article are linear contrasts of the projection parameters which are different from those addressed in \cite{rinaldo2019bootstrapping} and \cite{kuchibhotla2020berry}.

\paragraph{De-biasing of the Least Square Estimator} 

Under misspecified linear models, extensive attention has focused on scenarios where $d=o(\sqrt{n})$, with limited exploration beyond this scaling despite the OLSE remaining well-defined for $d\leq n$. In cases where $d\gg \sqrt{n}$, the OLSE demonstrates a bias of order $d/\sqrt{n}$ \citep{mammen1993bootstrap,cattaneo2019two}, impeding valid inference on regression parameters without de-biasing. To tackle this, \cite{cattaneo2019two} proposed the bias-corrected $M$-estimator via the Jackknife method, ensuring consistency and asymptotic normality of the two-step point estimate under $d=O(\sqrt{n})$, presenting bootstrap-based inferential methods as well.

The term ``de-biasing" typically refers to a correction applied to the original regularized estimator, often necessary to resist the curse of dimensionality or to enhance prior knowledge about the geometric/intrinsic structure of data. Common examples in linear regression include the LASSO, ridge, and SLOP estimators. It is worth emphasizing that ``de-biasing" here aims to address the ``bias" mainly induced by model misspecification at the population level rather than a specific regularization technique. 

\paragraph{Variance Estimation} In the context of linear regression, the inflating variance of regression coefficient estimates along with increasing dimension has prompted efforts to propose robust covariance estimates capable of accommodating dimensionality and potential heteroscedasticity. The degree-of-freedom-corrected covariance estimator, introduced by \cite{mackinnon1985some}, has served as a foundation for subsequent modifications within well-specified linear models, known as HC-class variance estimators. For comprehensive reviews of these estimators, we refer our readers to \cite{long2000using} and \cite{mackinnon2012thirty}. Recent additions to the HC class include the HCK estimator proposed by \cite{cattaneo2018inference} and the HCA estimator proposed by \cite{jochmans2022heteroscedasticity}. However, in both cases, consistency of variance estimate requires a dimensionality constraint of $d=O(n^{1/2-\epsilon})$ in misspecified linear models, albeit under different settings. Under a very similar setting to ours,  \cite{kuchibhotla2020berry} noted that a dimensionality requirement of $d=o(n^{1/2})$ appears to be unavoidable for the consistency of the sandwich estimator for the OLSE variance (see Lemma~8 of \cite{kuchibhotla2020berry}).

\subsection{Contributions}
In this paper, we employ an assumption-lean framework to offer a finite-sample distribution approximation for the projection parameter, imposing minimal assumptions. Specifically, our focus centers on estimating and inferring a scalar parameter $\theta:=c^\top\beta$ in a predetermined direction $c\in\Real^d$. The main contributions of this study are summarized as follows.
\begin{itemize}
    \item We propose a bias-corrected least square estimator for the projection parameter. Theorem~\ref{thm:3} presents the Berry Esseen bound for the approximation to the adjusted Normal distribution of the unnormalized linear contrast of the proposed estimator, and the bound is sharp and roughly scales like
    \begin{eqnarray}\label{eq:sec1:1}
        \frac{1}{\sqrt{n}}+\frac{d^{3/2}}{n},
    \end{eqnarray} provided that $q_x\geq 12$ and $q\geq 4$, where $q_x$ and $q$ represent the finite moments of the covariate $X$ and the error $Y-X^\top\beta$, respectively; See also Corollary~\ref{cor:1}. Consequently, achieving the vanishing bound necessitates the scaling $d = o(n^{2/3})$, strictly embracing the traditional requirement of $d=o(n^{1/2})$. Our result, to the best of our knowledge, presents the sharpest bound under the (possibly) misspecified linear model while relying on the weakest assumptions regarding the true data distribution.
    
    \item While the Berry Esseen bound presented in Theorem~\ref{thm:3} involves the unknown parameter for the asymptotic variance of the proposed estimator, Section~\ref{sec:4} outlines at least two inference methods for constructing (non-)asymptotically valid confidence intervals. These methods do not necessitate estimating the asymptotic variance, thereby enabling us to utilize the sharp Berry Esseen bound from Theorem~\ref{thm:3}. Our approach involves inferential techniques based on resampling and sample splitting, including HulC \citep{kuchibhotla2021hulc}, $t-$statistic based inference \citep{ibragimov2010t, lam2022cheap}, and wild bootstrap \citep{wu1986jackknife, mammen1993bootstrap}. We also illustrate these proposed methods through empirical examples in Section~\ref{sec:5}.
    
    \item Our main result, Theorem~\ref{thm:5}, establishes a finite-sample Berry-Esseen bound for a studentized linear contrast of the bias-corrected estimator under the assumptions of finite moments for both the covariates and the errors. Our bound, disregarding the polylogarithmic factors, roughly scales as follows:
    \begin{eqnarray}\label{eq:sec1:2}
    \frac{1}{n^{1/3}} +\sqrt{\frac{d}{n}}+ \frac{d^{3/2}}{n},
    \end{eqnarray} provided that $q_x\geq 12$ and $q\geq 4$; See also Corollary~\ref{cor:2}. Our bound immediately leads to the validity of the Wald confidence interval for a linear contrast of the projection parameters as outlined in Corollary~\ref{cor:3}.

    The slower convergence rate outlined in \eqref{eq:sec1:2} compared to that in \eqref{eq:sec1:1} is solely due to the additional convergence rate introduced by the sandwich variance estimator. However, it is noteworthy that the consistency of the sandwich variance estimator does not demand any additional dimension scaling beyond $d=o(n^{2/3})$. Moreover, as per Theorem~\ref{thm:4}, the sandwich variance estimator attains consistency if $d=o(\min\{n^{1-1/q_x-1/q},n^{1-2/q_x-2/(3q)}\})$. To the best of our knowledge, our result on the consistency of the sandwich variance estimator allows the broadest range of dimension scaling in assumption lean linear regression literature. 

    \item One of the key quantities in our analysis is $\Dc_\Sigma = \norm{\Sigma^{-1/2}\hat\Sigma\Sigma^{-1/2}-I_d}_{\rm op}$, representing the difference between the population and the sample gram matrix in the operator norm. We introduce a new concentration inequality for $\Dc_\Sigma$ under finite moment assumptions on covariates ($q_x\geq 4$), detailed in Proposition~\ref{prop:6}. This result is inspired by recent findings on the universal spectral properties of random matrices, relating them to Gaussian random matrices with identical mean and covariance in their entries \citep{brailovskaya2023universality}.  Our bound for $\Dc_\Sigma$ matches the sharpness (up to polylogarithmic factors) of the one provided in \cite{tikhomirov2018sample}, regarded as the sharpest to date. Notably, our bound avoids the necessity for a `large' $d$ to ensure a `high' probability concentration of the sample Gram matrix. Section~\ref{sec:conc_ineq} outlines the historical context of $\Dc_\Sigma$ and offers a comparison of concentration inequalities in this regard.

\end{itemize}

\subsection{Organization}
The remainder of this paper is organized as follows: Section~\ref{sec:2.1} introduces the problem setup and notations, while Section~\ref{sec:2.2} outlines the distributional assumptions on the data generating process. Our primary result on the Berry Esseen bound for the distribution approximation of the law of the studentized linear contrast of the proposed estimator is presented in Section~\ref{sec:3}. Section~\ref{sec:4} details several inferential methods for the projection parameter based on the bootstrap and sample splitting. Numerical results are provided in Section~\ref{sec:5}. The Appendix includes all proofs, technical lemmas, and additional numerical results.
\section{Problem setup and Assumptions}\label{sec:2}
\subsection{Projection parameters}\label{sec:2.1}
Let $(X_1, Y_1),\ldots, (X_n, Y_n)\sim(X,Y)\in\Real^d\times\Real$ be an independent sample of $n$ observations. The \emph{projection parameter} $\beta$ is defined as
\begin{equation}\label{eq:projection_parameters}
    \beta=\beta_n:=\argmin_{\theta\in\Real^d}\frac{1}{n}\sum_{i=1}^n\Eb[(Y_i-X_i^\top\theta)^2].
\end{equation} The minimizer $\beta$ in \eqref{eq:projection_parameters} solves the linear equation that
\begin{equation}\label{eq:projection_parameters_2}
    \frac{1}{n}\sum_{i=1}^n\Eb[X_i(Y_i-X_i^\top\beta)]=0.
\end{equation} If the population gram matrix $\Sigma = \Sigma_n:= n^{-1}\sum_{i=1}^n\Eb[X_iX_i^\top]$ is invertible, the projection parameter is well-defined and can be uniquely expressed as $\beta = \Sigma^{-1}\Gamma$ where $\Gamma=\Gamma_n:=n^{-1}\sum_{i=1}^n\Eb[X_iY_i]$. In the case where the observations follow the linear model $Y= X^\top\beta^*+\epsilon$ with $\Eb[\epsilon|X]=0$, the projection parameter corresponds to the model parameter, i.e., $\beta=\beta^*$. On the contrary, if the underlying truth exhibits a possibly non-linear structure, then the projection parameter gives the best linear approximation $X^\top\beta$ to $Y$ with respect to the joint distribution of $(X,Y)$. For a detailed discussion on the projection parameter and its interpretation, see \cite{buja2019models1,buja2019models2}.

The projection parameter is estimated using the ordinary least square estimator (OLSE) which follows the traditional definition that 
\begin{equation*}
    \hat\beta = \hat\beta_n:=\argmin_{\theta\in\Real^d}\frac{1}{n}\sum_{i=1}^n(Y_i-X_i^\top\theta)^2.
\end{equation*} If the sample population matrix $\hat\Sigma=\hat\Sigma_n:=n^{-1}\sum_{i=1}^n X_iX_i^\top$ is invertible with probability 1, then the OLSE is well defined and can be written as $\hat\beta = \hat\Sigma^{-1}\hat\Gamma$ where $\hat\Gamma=\hat\Gamma_n:=n^{-1}\sum_{i=1}^n X_iY_i$.

\paragraph{Additional Notations} For two real numbers $a$ and $b$, let $a\vee b=\max\{a,b\}$, $a\wedge b=\min\{a,b\}$ and $a_+=\max\{a,0\}$. Let $I_p$ be the $p\times p$ identity matrix.  We let $1_p$ be the $p$-dimensional vector of ones and $0_p$ be of zeros. For any $x\in\Real^d$, we write $\norm{x}_2=\sqrt{x^\top x}$. In addition with a positive definite matrix $A\in\Real^{d\times d}$, we denote the scaled Euclidean norm as $\norm{x}_A = \sqrt{x^\top Ax}$. We let $\norm{\cdot}_{\rm op}$ be the operator norm of the matrices. The unit sphere in $\Real^d$ is $\Sb^{d-1}=\set{\theta\in\Real^d:\norm{\theta}_2=1}$. The $j$:th canonical basis of $\Real^d$ is written as $\ev_j$ for all $j=1,\ldots,d$. Convergence in distribution is denoted by $\xrightarrow{D}$ and convergence in probability by $\xrightarrow{P}$. We write the smallest eigenvalue and the largest eigenvalue of the square matrices as $\lambda_{\min}(\cdot)$ and $\lambda_{\max}(\cdot)$, respectively. The cumulative distribution function of the standard Normal distribution is denoted as $\Phi(\cdot)$ and its $k$:th derivative as $\Phi^{(k)}(\cdot)$ for $k\in\Nb$. Finally, we adopt the conventional asymptotic notations. Given two non-negative sequences of reals $\{a_n\}$ and $\{b_n\}$, we write $a_n = O(b_n)$ if there exists a constant $C\in(0,\infty)$ and $N\in\Nb$ such that for any $n\geq N$, $a_n\leq C b_n$, and  $a_n=o(b_n)$ if $a_n/b_n\to0$ as $n\to\infty$. We sometimes write $a_n\lesssim b_n$ for $a_n=O(b_n)$, and $a_n\ll b_n$ for $a_n=o(b_n)$. Furthermore, for a sequence of random variables $\set{\gamma_n}$ and a non-zero sequence $\{c_n\}$, we write $\gamma_n = o_P(c_n)$ if $\abs{\gamma_n/c_n}\xrightarrow{P}0$ as $n\to\infty$, and $\gamma_n = O_P(c_n)$ if for any $\epsilon>0$, there exists $C_\epsilon\in(0,\infty)$ and $N_\epsilon\in\Nb$ such that $\Pb(\abs{\gamma_n/c_n}>C_\epsilon)<\epsilon$ for all $n\geq N_\epsilon$.

\subsection{Assumptions}\label{sec:2.2}
Our main results are built upon the following assumptions, primarily emphasizing minimal assumptions on the data generating process. These include moment conditions on both the response and covariates, thereby encompassing a wide range of distributions, such as heavy-tailed distributions and discrete distributions.


\begin{assumption}\label{asmp:2}
    There exists some $q\geq2$ and a constant $K_y>0$ satisfying that
    \begin{equation*}
        \left(\Eb[|Y_i-X_i^\top\beta|^q]\right)^{1/q}\leq K_y,
    \end{equation*} for $i=1,\ldots,n$.
\end{assumption}

Assumption~\ref{asmp:2} imposes a moment condition solely on the error, allowing for heteroscedastic errors that may arbitrarily depend on the covariates and even accommodate heavy-tailed errors. This assumption is prevalent in recent literature on heteroskedastic linear regression \citep{mourtada2022distribution,kuchibhotla2022least,han2019convergence,mendelson2016upper}. Additionally, it is often assumed that the conditional variance of the error is uniformly bounded almost everywhere, described by 
\begin{equation}\label{eq:cond:1}
    0<\inf_x v^2(x)\leq \sup_x v^2(x) <\infty,
\end{equation} where $v^2(x)=\Eb[(Y-X^\top\beta)^2|X=x]$. Here, $\inf_x$ and $\sup_x$ are essential infimum and essential supremum, respectively, with respect to the marginal distribution of $X$.  Given \eqref{eq:cond:1}, our results in the following section remain valid even under the relaxed moment assumptions on the covariates.

\stepcounter{assumption}
\begin{intassumption}\label{asmp:3.a} 
    There exists some $q_x\ge2$ and a constant $K_x\geq1$ satisfying that
    \begin{equation*}
        \left(\Eb[|u^\top\Sigma^{-1/2}X_i|^{q_x}]\right)^{1/q_x}\leq K_x,
    \end{equation*} for $i=1,\ldots,n$ and $u\in\Sb^{d-1}$.
\end{intassumption}

\begin{intassumption}\label{asmp:3.b} There exists a constant $K_x>0$ such that
\begin{equation*}
    \Eb\left[\exp\left(\frac{|u^\top\Sigma^{-1}X_i|^2}{K_x^2}\right)\right]\leq2,
\end{equation*} for all $i=1,\ldots,n$ and $u\in\Sb^{d-1}$.
\end{intassumption}

\begin{intassumption}\label{asmp:3.c}
    There exists some $q_x\ge2$ and a constant $K_x\geq1$ satisfying that
    \begin{equation*}
\left(\Eb[|u^\top\Sigma^{-1/2}X_i|^{q_x}]\right)^{1/q_x}\leq K_x,
    \end{equation*} 
    for $i=1,\ldots,n$ and $u\in\Sb^{d-1}$. Furthermore, $\Sigma^{-1/2}X_i\in\Real^d$, $i=1,\ldots,n$, have $d$ independent entries.
\end{intassumption}
Assumption~\ref{asmp:3.a} implies $L^{q_x}-L^2$ equivalence for one-dimensional marginals of covariates. This is a significant weakening of sub-Gaussianity in Assumption~\ref{asmp:3.b}. It is common to impose $q_x\geq 4$ when studying the linear regression with increasing dimensions \citep{oliveira2016lower,catoni2016pac, mourtada2022distribution, mourtada2022exact,vavskevivcius2023suboptimality} even though we sometimes require a higher moment condition to analyze a higher order expansion of the least square estimate. 

Furthermore, Assumption~\ref{asmp:3.a} implies the bounded-ness of $q_x$:th moment of the $\Sigma^{-1}$-norm of the random vector $X_i$. To illustrate, note that $\norm{X_i}_{\Sigma^{-1}}^2=\sum_{j=1}^d(\ev_j^\top\Sigma^{-1/2}X_i)^2$ for all $i=1,\ldots,n$. By applying Jensen's inequality, we have:
\begin{equation*}
    \left(\frac{\norm{X_i}_{\Sigma^{-1}}^2}{d}\right)^{q_x/2}\leq\frac{1}{d}\sum_{j=1}^d\left(\ev_j^\top\Sigma^{-1/2}X_i\right)^{q_x},
\end{equation*} which holds deterministically as long as $q_x\geq2$. Consequently, it leads to
\begin{equation*}
    \Eb\left[\left(\frac{\norm{X_i}_{\Sigma^{-1}}}{\sqrt{d}}\right)^{q_x}\right]\leq \frac{1}{d}\sum_{j=1}^d\Eb\left[(\ev_j^\top\Sigma^{-1/2}X_i)^{q_x}\right]\leq \max_{1\leq j \leq d}\Eb\left[(\ev_j^\top\Sigma^{-1/2}X_i)^{q_x}\right]\leq K_x^{q_x},
\end{equation*} where the last inequality follows from Assumption~\ref{asmp:3.a}. This also emphasizes that the appropriate scaling of $\norm{X_i}_{\Sigma^{-1}}$ is proportional to $\sqrt{d}$ as the dimension $d$ grows. 

Assumption~\ref{asmp:3.c} assumes that $\Sigma^{-1/2}$-normalized covariates have independent entries, in addition to the moment conditions. Under Assumption~\ref{asmp:3.c}, we can achieve a more precise control over the $L^2$-norm of covariates \citep{rio2017constants} and the sample Gram matrix \citep{srivastava2013covariance,mendelson2014singular}.


\begin{assumption}\label{asmp:4} Let
\begin{equation*}
    V:={\rm Var}\left[\frac{1}{\sqrt{n}}\sum_{i=1}^n X_i(Y_i-X_i^\top\beta)\right]={\rm Var}\left[X_1(Y_1-X_1^\top\beta)\right].
\end{equation*}There exist constants $0< \underline{\lambda} \leq \overline{\lambda} <\infty$ satisfying that
    \begin{equation*}
        \underline{\lambda}\leq \lambda_{\rm min}(\Sigma^{-1/2}V\Sigma^{-1/2}) \leq \lambda_{\rm max}(\Sigma^{-1/2}V\Sigma^{-1/2}) \leq \overline{\lambda}.
    \end{equation*}
\end{assumption}

Assumption~\ref{asmp:4} ensures that both $V$ and $\Sigma$ scale in a similar order, serving as the only technical assumption in our paper. This condition appears inevitable to prevent the asymptotic variance of the least square estimator from inflating or vanishing. Specifically, noting that $V=\Eb[XX^\top(Y-X^\top\beta)^2] = \Eb[XX^\top\Eb[(Y-X^\top\beta)^2|X]]$ the condition~\eqref{eq:cond:1} is sufficient for Assumption~\ref{asmp:4} to be held.

\section{Methods and Results}\label{sec:3}
\subsection{Approximation of the Sample Gram Matrix and Bias Characterization}\label{sec:3.1}
In this section, we present an approximation of the distribution of $\sqrt{n}c^\top(\hat\beta-\beta)$ for any $c\in\Real^d\setminus\{0_d\}$. To be presented results are mostly invariant to the scaling of $c\in\Real^d$. However, we sometimes consider those with $\norm{c}_{\Sigma^{-1}}=1$, which gives a proper normalization, $\sqrt{n}c^\top(\hat\beta-\beta)=\sqrt{n}\tilde c^\top\Sigma^{1/2}(\hat\beta-\beta)$ for some $\tilde c\in\Sb^{d-1}$.

Our analysis starts with the least square estimate $\hat\beta$, particularly by approximating the sample Gram matrix $\hat\Sigma$ using Taylor series expansion. We note
\begin{eqnarray*}
    \hat\Sigma^{-1}&=&\Sigma^{-1} +\hat\Sigma^{-1}(\Sigma-\hat\Sigma)\Sigma^{-1}\\
    &\approx&\Sigma^{-1} +\Sigma^{-1}(\Sigma-\hat\Sigma)\Sigma^{-1}.
\end{eqnarray*} 
(See Lemma 33 of~\cite{kuchibhotla2020berry}.)
Consequently, we get the following approximation of $\hat\beta-\beta$:
\begin{eqnarray}
    \hat\beta - \beta &=& \hat\Sigma^{-1}\frac{1}{n}\sum_{i=1}^n X_i(Y_i-X_i^\top\beta) \nonumber\\
    &\approx& \Sigma^{-1}\frac{1}{n}\sum_{i=1}^n X_i(Y_i-X_i^\top\beta)\label{eq:2.2.1}\\
    &&+\,\Sigma^{-1}(\Sigma-\hat\Sigma)\Sigma^{-1}\frac{1}{n}\sum_{i=1}^n X_i(Y_i-X_i^\top\beta)\label{eq:2.2.2}.
\end{eqnarray}
The term \eqref{eq:2.2.1} is the first-order approximation and is the average of influence function---with respect to the square loss---evaluated at independent data points. Using the conventional large sample approximation for the case when $d = o(\sqrt{n})$, the variability of $\hat\beta$ can be well-represented only through the term \eqref{eq:2.2.1}, and the second order term \eqref{eq:2.2.2} or higher order terms are negligible. In particular, \eqref{eq:2.2.1} is responsible for the asymptotic normality of $\sqrt{n}(\hat\beta-\beta)$ when $d=o(\sqrt{n})$. Here, however, the term \eqref{eq:2.2.2} may create a non-vanishing bias when $d\gg \sqrt{n}$. To see this, we note that
\begin{eqnarray}
    \Sigma^{-1}(\Sigma-\hat\Sigma)\Sigma^{-1}\frac{1}{n}\sum_{i=1}^n X_i(Y_i-X_i^\top\beta) &=& \left(\frac{1}{n}\sum_{i=1}^n \Sigma^{-1}(\Sigma-X_iX_i^\top)\Sigma^{-1}\right)\left(\frac{1}{n}\sum_{i=1}^n X_i(Y_i-X_i^\top\beta)\right)\label{eq:2.2.3.0}\\
    &=&\frac{1}{n^2}\sum_{i=1}^n \Sigma^{-1}X_i(Y_i-X_i^\top\beta)\label{eq:2.2.3}\\
    &&-\frac{1}{n^2}\sum_{i=1}^n \Sigma^{-1}X_i(Y_i-X_i^\top\beta)\norm{X_i}_{\Sigma^{-1}}^2\label{eq:2.2.4}\\
    &&+\frac{1}{n^2}\sum_{1\leq i\neq j\leq n} \Sigma^{-1}(\Sigma-X_iX_i^\top)\Sigma^{-1}X_j(Y_j-X_j^\top\beta)\label{eq:2.2.5}.
\end{eqnarray} The quantities in \eqref{eq:2.2.3} and \eqref{eq:2.2.4} are resulted from multiplying the individual terms in \eqref{eq:2.2.3.0} with the same indices, while the quantity \eqref{eq:2.2.5} corresponds to the cross-index product. It is noteworthy that \eqref{eq:2.2.3} can be absorbed into the average of influence functions in \eqref{eq:2.2.1} as it has an additional scaling factor $1/n$. The quantity \eqref{eq:2.2.5} is a mean-zero degenerate $U$-statistic of order $2$ up to a scaling factor that converges to 1 as $n\to\infty$. Therefore, the only quantity that may have a non-zero mean is \eqref{eq:2.2.4} and we denote this by
\begin{equation}\label{eq:bias_true}
    \Bc :=-\frac{1}{n^2}\sum_{i=1}^n\Sigma^{-1}X_i(Y_i-X_i^\top\beta)\norm{X_i}_{\Sigma^{-1}}^2.
\end{equation} Under the control of $c^\top\Sigma^{-1}X_i$ and $Y_i-X_i^\top\beta$ using Assumption~\ref{asmp:3.a} and \ref{asmp:2}, respectively, the quantity $c^\top\Bc$ scales like $O(d/n)$ due the $\norm{X_i}_{\Sigma^{-1}}^2$ factor. Consequently, $\sqrt{n}c^\top\Bc$ yields a non-degenerate bias when $d\gg\sqrt{n}$. By manually removing the bias, a linear contrast $c^\top(\hat\beta-\Bc-\beta)$ can be approximated by the sum of the average of influence functions and the degenerate second-order $U$-statistics as
\begin{equation*}
    c^\top(\hat\beta-\Bc-\beta)~\approx~\frac{1}{n}\sum_{i=1}^nc^\top\psi(X_i,Y_i)+\frac{1}{n(n-1)}\sum_{1\leq i\neq j\leq n}c^\top\phi(X_i,Y_i,X_j,Y_j),
\end{equation*} where 
\begin{equation}\label{eq:psi_and_phi}
    \psi(x,y) = \left(1+\frac{1}{n}\right)\Sigma^{-1}x(y-x^\top\beta)\quad\mbox{and}\quad \phi(x,y,x',y') = \left(1-\frac{1}{n}\right) \Sigma^{-1}(\Sigma-xx^\top)\Sigma^{-1}x'(y'-{x'}^\top\beta).
\end{equation}  
It is well known that $U$-statistics of order $k\geq 2$ are asymptotically normally distributed, and there has been a vast literature related to Normal approximations and the rates of convergence for $k$-order $U$-statistics. We refer the readers to \cite{bentkus2009normal}. In order to state our first theorem, we define
\begin{eqnarray}
    \sigma_c^2 &=& {\rm Var}\left[c^\top\Sigma^{-1}X(Y-X^\top\beta)\right]=c^\top\Sigma^{-1}V\Sigma^{-1} c,\label{eq:2.2.6}\\
    \kappa_c &=& n^{-5/2}\binom{n}{2}\Eb\left[\{c^\top\psi(X_1,Y_1)\}\{c^\top\psi(X_2,Y_2)\}\{c^\top\phi(X_1,Y_1,X_2,Y_2)\}\right]/\sigma_c^3,\label{eq:2.2.7}
\end{eqnarray} for any $c\in\Real^d\setminus\{0_d\}$. It should be emphasized that the parameter $\kappa_c$ is invariant to the scaling of $c\in\Real^d\setminus\{0_d\}$. That is, $\kappa_c = \kappa_{\alpha c}$ for $\alpha\neq 0$. In addition, recall that $\Phi(x)$ is the cumulative distribution function of the standard Normal distribution and $\Phi^{(3)}(x)$ is the third order derivative of $\Phi(x)$.

Theorem~\ref{thm:1} states a Berry-Essen bound for the Normal approximation of a linear contrast $c^\top(\hat\beta-\Bc-\beta)$ with known bias $\Bc$.

\begin{theorem}\label{thm:1}
Suppose that Assumption~\ref{asmp:2} holds for $q>3$ and that Assumption \ref{asmp:4} holds.
\begin{itemize}
    \item[(i)] Suppose further that Assumption~\ref{asmp:3.a} holds for $q_x\geq 4$ and $s:=(1/q+1/q_x)^{-1}\geq3$, and assume $d+2\log(2n)\leq n/(18K_x)^2$. Then, there exists a constant $C=C(q,q_x,K_x,K_y,\overline{\lambda},\underline{\lambda})$ such that
    \begin{eqnarray*}
        &&\sup_{c\in\Real^d\setminus\{0_d\}}\sup_{t\in\Real}\Abs{\Pb\left[\sqrt{n}\left\{c^\top(\hat \beta-\Bc-\beta)\right\}\leq t\right]-\left\{\Phi\left(\frac{t}{\sigma_c}\right)+\kappa_c\Phi^{(3)}\left(\frac{t}{\sigma_c}\right)\right\}}\\
        &\leq&  \frac{C}{\sqrt{n}}+C\left\{\left(\frac{d\log^3(2n)}{n^{4/5-8/(5q_x)}}\right)^{5q_x/(2q_x+8)}+\frac{d^{3/2}\log(2n)}{n}\right\} \left(1\vee\sqrt{\frac{\log (2n)}{d}}\right). 
    \end{eqnarray*}

    \item[(ii)] Suppose further that Assumption~\ref{asmp:3.b} holds. Then, there exists a constant $C=C(q,K_x,K_y,\overline{\lambda},\underline{\lambda})$ such that
    \begin{eqnarray*}
        &&\sup_{c\in\Real^d\setminus\{0_d\}}\sup_{t\in\Real}\Abs{\Pb\left[\sqrt{n}\left\{c^\top(\hat \beta-\Bc-\beta)\right\}\leq t\right]-\left\{\Phi\left(\frac{t}{\sigma_c}\right)+\kappa_c\Phi^{(3)}\left(\frac{t}{\sigma_c}\right)\right\}}\\
        &\leq& C\left(\frac{1}{\sqrt{n}}+\frac{d^{3/2}}{n}+\frac{d\log (2n)}{n}\right).
    \end{eqnarray*}
    \item[(iii)] Suppose further that Assumption~\ref{asmp:3.c} holds for $q_x> 4$ and $s\geq3$, and assume $d+2\log(2n)\leq n/(18K_x)^2$. Then, there exists a constant $C=C(q,q_x,K_x,K_y,\overline{\lambda},\underline{\lambda})$ such that
    \begin{eqnarray*}
        &&\sup_{c\in\Real^d\setminus\{0_d\}}\sup_{t\in\Real}\Abs{\Pb\left[\sqrt{n}\left\{c^\top(\hat \beta-\Bc-\beta)\right\}\leq t\right]-\left\{\Phi\left(\frac{t}{\sigma_c}\right)+\kappa_c\Phi^{(3)}\left(\frac{t}{\sigma_c}\right)\right\}}\\
        &\leq&  \frac{C}{\sqrt{n}}+C\left\{\left(\frac{d^{3/2}\log^3(2n)}{n}\right)^{q_x/(q_x+4)} +\frac{d^{3/2}\log(2n)}{n}\right\}\left(1\vee\sqrt{\frac{\log (2n)}{d}}\right). 
    \end{eqnarray*}

    Furthermore, if $(3/q_x+1/q)^{-1}\geq 2$, then there exists a constant $C=C(q,q_x,K_x,K_y,\overline{\lambda},\underline{\lambda})$ such that
    \begin{eqnarray*}
        &&\sup_{c\in\Real^d\setminus\{0_d\}}\sup_{t\in\Real}\Abs{\Pb\left[\sqrt{n}\left\{c^\top(\hat \beta-\beta)\right\}\leq t\right]-\left\{\Phi\left(\frac{t}{\sigma_c}\right)+\kappa_c\Phi^{(3)}\left(\frac{t}{\sigma_c}\right)\right\}}\\
        &\leq&  C\sqrt{\frac{d}{n}}+C\left\{\left(\frac{d^{3/2}\log^3(2n)}{n}\right)^{q_x/(q_x+4)} +\frac{d^{3/2}\log(2n)}{n}\right\}\left(1\vee\sqrt{\frac{\log (2n)}{d}}\right).
    \end{eqnarray*}
\end{itemize}

\end{theorem}

\begin{remark}\label{rmk:1}
    The conventional Berry Esseen type inequality implies the proximity between the distribution of the law of the estimator and the Normal distribution. In contrast, we show a distribution approximation to the ``adjusted'' Normal distribution, allowing a more precise representation of the approximation. The degree of adjustment is determined by the parameter $\kappa_c$, which turns out to scale as $O(\sqrt{d/n})$ (see Lemma~\ref{lem:11}). 
    Hence, we can enhance the Berry-Esseen bound for the Normal approximation by incorporating an additional term of $O(\sqrt{d/n})$ to the existing bounds in Theorem~\ref{thm:1}, but yielding a slower convergence rate. More importantly, under the well-specified linear model, i.e., $\Eb[Y|X]=X^\top\beta$, then it follows that $\kappa_c=0$ from its definition in \eqref{eq:2.2.7}. Consequently, in this case, the presented Berry-Esseen bound in Theorem~\ref{thm:1} can be employed for the Normal approximation as well.
\end{remark}

\begin{remark}\label{rmk:2}
    The last part of Theorem~\ref{thm:1} says that the OLSE does not necessarily require de-biasing under Assumption~\ref{asmp:3.c} with a sufficiently large finite moment of covariates and errors. To provide an intuition for this, let us consider the bias term $c^\top\Bc$ in \eqref{eq:bias_true}, which scales with its expected value $\Eb[c^\top\Bc]$. By examining the definition of the projection parameter, we can express $\Eb[c^\top\Bc]$ as the covariance between two random variables:
    \begin{eqnarray*}
    \Eb[c^\top\Bc] &=& -\frac{1}{n}\Eb [c^\top\Sigma^{-1}X(Y-X^\top\beta)\norm{X}_{\Sigma^{-1}}^2]\\
    &=&-\frac{1}{n}{\rm Cov}\left(c^\top\Sigma^{-1}X(Y-X^\top\beta), \norm{X}_{\Sigma^{-1}}^2\right).
    \end{eqnarray*}
    An application of Cauchy Schwarz's inequality yields that
    \begin{equation}\label{eq:rmk2.1}
    \Abs{\Eb[c^\top\Bc]}\leq\frac{1}{n}\left({\rm Var}\left[c^\top\Sigma^{-1}X(Y-X^\top\beta)\right]\right)^{1/2} \left({\rm Var}\left[ \norm{X}_{\Sigma^{-1}}^2\right]\right)^{1/2}.
    \end{equation}
    The leading variance term on the right-hand side of \eqref{eq:rmk2.1}, under the finite moment assumptions of covariates and errors, is $O(1)$. Meanwhile, the quantity ${\rm Var}[ \norm{X}_{\Sigma^{-1}}^2]$ generally scales as $\Eb[\norm{X}_{\Sigma^{-1}}^4]=O(d^2)$, resulting in $|\Eb[c^\top\Bc]|=O(d/n)$. However, when an additional assumption of independent entries of covariates, i.e., Assumption~\ref{asmp:3.c}, is satisfied, $\norm{X}_{\Sigma^{-1}}^2$ roughly follows $\chi^2_d$ distribution, thereby leading to ${\rm Var}[ \norm{X}_{\Sigma^{-1}}^2]=O(d)$. This significantly reduces the magnitude of the right-hand side of \eqref{eq:rmk2.1} to $O(\sqrt{d}/n)$. Consequently, the expected bias (when scaled by $\sqrt{n}$) is negligible if $d=o(n)$.
\end{remark}

\begin{remark}
    Regarding the results under Assumptions \ref{asmp:3.a} and \ref{asmp:3.c}, we conjecture that the dependence on logarithmic factors of our bound is suboptimal. The presence of seemingly spurious logarithmic factors mainly stems from the concentration inequality we employ for the sample Gram matrix, described in Proposition~\ref{prop:6} in Section~\ref{sec:conc_ineq}. Although Theorem 1.1 of \cite{tikhomirov2018sample} partially addresses our query when $q_x>4$, its application to our context (see Proposition~\ref{prop:tik}) requires further supposition: $d\to\infty$ as $n\to\infty$.
\end{remark}

Theorem~\ref{thm:1} gives the Berry Essen bound for the linear contrast of the \emph{properly-debiased} projection parameters under three different distributional assumptions on covariates. The first part of Theorem~\ref{thm:1} regards the most general case which relies only on finite-moment assumptions. Ignoring the polylogarithmic factors, the Berry-Esseen bound tends to zero as long as $$d=o\left(n^{\min\{2/3,4/5-8/(5q_x)\}}\right),\quad n\to\infty.$$ If $q_x\geq12$, this reduces to $d=o(n^{2/3})$ and meets the dimension requirement for the vanishing Berry Esseen bound under Assumption~\ref{asmp:3.b} of sub-Gaussianity. Moreover, under Assumption~\ref{asmp:3.c}, the bound tends to 0 as $d=o(n^{2/3})$ for all $q_x>4$, disregarding the polylogarithmic factors.

\subsection{Bias Estimation}\label{sec:3.2}

The result given in Theorem~\ref{thm:1} is mostly of theoretical importance as the inequality yet involves unknown parameters through bias $\Bc$ and asymptotic variance $\sigma_c^2$. Therefore, to do inference on $\beta$, for e.g. building confidence region, it is necessary to incorporate the estimates of bias and variance in the Berry Esseen bound given in Theorem~\ref{thm:1}. To this effect, we consider a method of moment estimator $\hat\Bc$ for $\Bc$ defined as
\begin{equation}\label{eq:bias_est}
    \hat\Bc = -\frac{1}{n^2}\sum_{i=1}^n \hat\Sigma^{-1}X_i(Y_i-X_i^\top\hat\beta)\norm{X_i}_{\hat\Sigma^{-1}}^2.
\end{equation} The next result provides the consistency rate for the bias estimate in high dimensions and under mild moment conditions. As elucidated in Remark~\ref{rmk:2}, the OLSE does not necessitate a debiasing procedure under Assumption~\ref{asmp:3.c} with a sufficient number of moments of covariates and errors. Therefore, the following results will consider the cases under Assumption~\ref{asmp:3.a} and \ref{asmp:3.b} only.

\begin{theorem}\label{thm:2}
    Suppose that Assumption~\ref{asmp:2} holds for $q>3$ and that Assumption~\ref{asmp:4} holds. Fix $c\in\Real^d$ with $\norm{c}_{\Sigma^{-1}}=1$, and assume $d+2\log(2n)\leq n/(18K_x)^2$.
    \begin{itemize}
        \item[(i)] Suppose further that Assumption~\ref{asmp:3.a} holds for $q_x\geq8$, $(1/q_x+1/q)^{-1}\geq3$, and $(3/q_x+1/q)^{-1}\geq 2$. Then there exists a constant $C=C(q,q_x,\overline{\lambda},K_x,K_y)$ such that for all $\eta\in(0,1]$,
\begin{eqnarray*}
\Pb\left(\sqrt{n}\abs{c^\top(\hat\Bc-\Bc)}\geq\delta_{n,d}(\eta)\right)\leq \frac{1}{\sqrt{n}}+\delta_{n,d}(\eta)+\eta,
\end{eqnarray*} where
\begin{equation*}
    \delta_{n,d}(\eta):= C\left(\frac{d^2\log^{3/2}(2d/\eta)}{\eta^{2/q_x}n^{3/2-2/q_x}}+\frac{d^3\log^3(2d/\eta)}{\eta^{4/q_x}n^{5/2-4/q_x}}+\frac{d^{3/2}\log(2n/\eta)}{n}\right).
\end{equation*}
\item[(ii)] Suppose further that Assumption~\ref{asmp:3.b} holds. Then, there exists a constant $C=C(q,\overline{\lambda},K_x,K_y)$ such that
\begin{eqnarray*}
\Pb\left(\sqrt{n}\abs{c^\top(\hat\Bc-\Bc)}\geq\delta_{n,d}\right)\leq \frac{1}{\sqrt{n}}+\delta_{n,d},
\end{eqnarray*} where
\begin{equation*}
    \delta_{n,d}:= C\log(2d)\left[\frac{d^{3/2}}{n}+\frac{d\log(2n)}{n}\right].
\end{equation*}
\end{itemize}
\end{theorem}

Theorem~\ref{thm:2} provides the finite sample concentration inequalities for the bias estimate. Inspecting the result under Assumption~\ref{asmp:3.a}, with the choice of
\begin{equation*}
    \eta=\left(\frac{d}{n^{3/4-1/q_x}}\right)^{(2q_x)/(2+q_x)}\vee\left(\frac{d}{n^{5/6-4/(3q_x)}}\right)^{(3q_x)/(4+q_x)},
\end{equation*} which was intended to minimize the tail probability, the bias estimate is $\sqrt{n}-$consistent if $d=o(n^{\min\{2/3,3/4-1/q_x\}})$, ignoring the polylogarithmic factors. If $q_x\geq12$, then the dimension requirement reduces to $d=o(n^{2/3})$. If Assumption~\ref{asmp:3.a} is replaced with Assumption~\ref{asmp:3.b}, then the dimension requirement for the consistency becomes $d=o(n^{2/3})$ again ignoring the polylogarithmic factors.

\subsection{Berry Esseen Bound for the Unnormalized Bias-corrected Estimator}\label{sec:3.3}
Combining the error bound for the bias estimate from Theorem~\ref{thm:2} with the Berry-Esseen bound in Theorem~\ref{thm:1}, we get to the next result, a uniform Berry-Esseen bound for the unnormalized bias-corrected OLS estimator. We denote the bias-corrected estimator for the projection parameters as 
\begin{eqnarray*}
    \hat \beta_{\rm bc}:=\hat\beta-\hat\Bc.
\end{eqnarray*}

\begin{theorem}\label{thm:3}
\begin{itemize}
    \item[(i)] Suppose that assumptions made in Theorem~\ref{thm:2}(i) hold. Then, there exists a constant $C=C(q,q_x,K_x,K_y,\overline{\lambda},\underline{\lambda})$ such that
\begin{eqnarray*}
\sup_{c\in\Real^d\setminus\{0_d\}}\sup_{t\in\Real}\Abs{\Pb\left[\sqrt{n}c^\top\left(\hat \beta_{\rm bc}-\beta\right)\leq t\right]-\left\{\Phi\left(\frac{t}{\sigma_c}\right)+\kappa_c\Phi^{(3)}\left(\frac{t}{\sigma_c}\right)\right\}}\leq \epsilon_{n,d},
\end{eqnarray*} where
\begin{eqnarray*}
    \epsilon_{n,d}&=&C\left\{\left(\frac{d\log^3(2n)}{n^{4/5-8/(5q_x)}}\right)^{5q_x/(2q_x+8)}+\left(\frac{d\log^{3/4}(2n)}{n^{3/4-1/q_x}}\right)^{(2q_x)/(2+q_x)}+\left(\frac{d\log(2n)}{n^{5/6-4/(3q_x)}}\right)^{(3q_x)/(4+q_x)}\right\}\\
&&+C\left(\frac{1}{\sqrt{n}}+\frac{d^{3/2}\log(2n)}{n}\right).
\end{eqnarray*}
\item[(ii)] Suppose that the assumptions made in Theorem~\ref{thm:2}(ii) hold. Then, there exists a constant $C=C(q,K_x,K_y,\overline{\lambda},\underline{\lambda})$ such that
\begin{eqnarray*}
    &&\sup_{c\in\Real^d\setminus\{0_d\}}\sup_{t\in\Real}\Abs{\Pb\left[\sqrt{n}c^\top\left(\hat \beta_{\rm bc}-\beta\right)\leq t\right]-\left\{\Phi\left(\frac{t}{\sigma_c}\right)+\kappa_c\Phi^{(3)}\left(\frac{t}{\sigma_c}\right)\right\}}\\
    &\leq& C \left(\frac{1}{\sqrt{n}}+\frac{d^{3/2}\log (2n)}{n}\right).
\end{eqnarray*}
\end{itemize}
\end{theorem}

Theorem~\ref{thm:3} presents the Berry Esseen bound for approximating the law of the unnormalized linear contrast of the bias-corrected LSE --- with respect to any query vector $c\in\Real^d\setminus\{0_d\}$--- to the adjusted Normal distribution. To the best of our knowledge, our bound stands as the sharpest among Berry Esseen bounds applicable to existing estimators in the assumption-lean linear regression framework. Moreover, it accommodates the widest range of dimensions for vanishing bound.

Under the moment condition of Assumption~\ref{asmp:3.a}, our bound converges to zero if $d=o(n^{\min\{2/3,4/5-8/q_x\}})$, disregarding the polylogarithmic factors. This reduces to $d=o(n^{2/3})$ when $q_x\geq 12$. Furthermore, if the covariates are sub-Gaussian, the Berry Esseen bound tends to zero if $d=o(n^{2/3})$, again ignoring the polylogarithmic factors. Comparing the Berry Essen bound of the bias-corrected estimator to the bound given in Theorem~\ref{thm:1} when the bias is known, we observe that using the bias estimate does not impose a stricter dimension requirement. This is due to the fact that the dimension requirement for the $\sqrt{n}$-consistency of the bias estimate in Theorem~\ref{thm:2} is less strict than that given in Theorem~\ref{thm:1}, as shown by the inequality: ${\rm min}\{2/3,4/5-8/(5q_x)\}\leq{\rm min}\{2/3,3/4-1/q_x\}$ for $q_x\geq4$. If $q_x\geq 12$, then the quantities on both sides coincide with 2/3, and thus, the bound converges to zero if $d=o(n^{2/3})$.

The following corollary presents a simpler Berry Essen bound for the bias-corrected OLS estimator. We use $A\lesssim B$ as a shorthand for the inequality $A\leq C_{n,d} B$, where $C_{n,d}$ involves only the constants defined in the assumptions and log-polynomial factors of $n$ and $d$.

\begin{corollary}\label{cor:1} Suppose that the assumptions made in Theorem~\ref{thm:3}(i) hold. Furthermore, if $q_x\geq 12$ and $d = O(n^{2/3})$, then
\begin{equation}\label{eq:cor1:1}
\sup_{c\in\Real^d\setminus\{0_d\}}\sup_{t\in\Real}\Abs{\Pb\left[\sqrt{n}c^\top\left(\hat \beta_{\rm bc}-\beta\right)\leq t\right]-\left\{\Phi\left(\frac{t}{\sigma_c}\right)+\kappa_c\Phi^{(3)}\left(\frac{t}{\sigma_c}\right)\right\}}\lesssim\frac{1}{\sqrt{n}} + \frac{d^{3/2}}{n}.
\end{equation} In addition, the approximation to the Normal distribution follows:
\begin{equation}\label{eq:cor1:2}
\sup_{c\in\Real^d\setminus\{0_d\}}\sup_{t\in\Real}\Abs{\Pb\left[\sqrt{n}c^\top\left(\hat \beta_{\rm bc}-\beta\right)\leq t\right]-\Phi\left(\frac{t}{\sigma_c}\right)}\lesssim\sqrt{\frac{d}{n}} + \frac{d^{3/2}}{n}.
\end{equation}
\end{corollary}

A slower rate in \eqref{eq:cor1:2} for the Normal approximation in comparison to that in \eqref{eq:cor1:1} stems from the incorporation of the adjustment term $\sup_{c}\sup_t\kappa_c\Phi^{(3)}(t/\sigma_c)$, which approximately scales at $O(\sqrt{d/n})$ (See Lemma~\ref{lem:11}), into the bound. As discussed previously in Remark~\ref{rmk:1}, however, if the linear model is correctly specified so that $\kappa_c=0$, then the bound provided in \eqref{eq:cor1:1} can also be utilized for the Normal approximation.

The Berry Esseen bound in Theorem~\ref{thm:3}, despite its reliance on the unknown distribution parameter $\sigma_c^2$, does not hinder our ability to conduct inference on the projection parameter $\beta$. We leverage recent statistical methodologies based on re-sampling or sample splitting to enable this inference, which remains unaffected by the uncertainty of $\sigma_c^2$. In Section~\ref{sec:4}, a series of distinct methods are introduced, illustrating and exemplifying this approach.

\subsection{Consistency of the Sandwich Variance Estimator}\label{sec:3.4}

Berry Esseen bound for the studentized bias-corrected LSE remains pertinent, notably in the construction of Wald-type confidence intervals. As previously discussed, accomplishing this entails devising an estimator for the asymptotic variance that ideally maintains consistency across a wide range of dimensions argued in Theorem~\ref{thm:3}. For a specified query vector $c\in\Real^d\setminus\{0_d\}$, the asymptotic variance of $\sqrt{n}c^\top(\hat\beta_{\rm bc}-\beta)$ can be approximated through $\sigma_c^2=c^\top\Sigma^{-1}V\Sigma^{-1}c$, as asserted in Theorem~\ref{thm:3}. Notably, the $d\times d$ matrix $\Sigma^{-1}V\Sigma^{-1}$ is commonly referred to as the `sandwich variance' for OLS in a misspecified linear regression setting \citep{white1980heteroskedasticity,buja2019models1}. A natural estimator for the sandwich variance is its plug-in counterpart $\hat\Sigma^{-1}\hat V\hat\Sigma^{-1}$, termed the `sandwich variance estimator', where $\hat\Sigma$ is the sample Gram matrix and $\hat V$ is expressed as:
\begin{equation*}
    \hat V:=\frac{1}{n}\sum_{i=1}^nX_iX_i^\top(Y_i-X_i^\top\hat\beta)^2.
\end{equation*} To this end, we define $\hat\sigma_c^2:=c^\top\hat\Sigma^{-1}\hat V\hat\Sigma^{-1}c$ for $c\in\Real^d\setminus\{0_d\}$.

The consistency of the sandwich variance estimator has been studied under various model- and distributional- assumptions and has often been a stand-alone line of research. A recent significant finding in a closely related context is detailed in Lemma~7 of \cite{kuchibhotla2020berry}. In particular, they prove that $\norm{\hat\Sigma^{-1}\hat V\hat\Sigma^{-1}-\Sigma^{-1}V\Sigma^{-1}}_{\rm op}=o_p(1)$ under conditions where $d=o(\min\{n^{1/2},n^{2(1-2/q_x)/3}\})$, up to polylogarithmic factors, provided that $(1/q_x+1/q)^{-1}\geq4$. Hence, directly applying their results to our scenario may be challenging, as we are considering situations beyond $d=o(\sqrt{n})$. Nevertheless, we have found that along a specific query vector $c\in\Real^d$, the estimator $\hat\sigma_c^2$ maintains consistency across a broader range of dimensions, as elucidated in the following theorem.

\begin{theorem}\label{thm:4} Suppose that Assumption~\ref{asmp:2},\ref{asmp:3.a}, and \ref{asmp:4} holds. Fix $c\in\Real^d\setminus\{0_d\}$, and assume $d+4\log(2n)\leq n/(18K_x)^2$. If $(1/q_x+1/q)^{-1}\geq4$ and $(3/q_x+1/q)^{-1}\geq2$, then there exists a constant $C=C(q,q_x,K_x,K_y,\overline{\lambda},\underline{\lambda})$ such that
\begin{equation*}
    \Pb\left(\Abs{\frac{\hat\sigma_c^2}{\sigma_c^2}-1}\geq \eta_{n,d}\right)\leq\eta_{n,d},
\end{equation*} where
\begin{eqnarray*}
    \eta_{n,d}&=&C\left(\frac{1}{n^{1/3}}+\sqrt{\frac{d+\log(2n)}{n}}+\frac{d\log(2n)}{n^{1-1/q_x}}\right)\\
    &&+C\left[\left(\frac{d\log^{3/2}(2n)}{n^{1-2/q_x}}\right)^{\frac{1}{1+2/q_x}}+\left(\frac{d\log^{5/4}(2n)}{n^{1-1/q_x-1/q}}\right)^{\frac{2}{1+2/q_x+2/q}}+\left(\frac{d\log^{3/2}(2n)}{n^{1-2/q_x-2/(3q)}}\right)^{\frac{3}{1+3/q_x+1/q}}\right]\\
    &&+C\left[\left(\frac{d\log^{5/4}(2n)}{n^{1-2/q_x}}\right)^{\frac{2}{1+4/q_x}}+\left(\frac{d\log^{3/4}(2n)}{n^{1-2/q_x}}\right)^{\frac{2+8/q_x+4/q}{1+6/q_x+2/q}}+\left(\frac{d\log^{1/2}(2n)}{n^{1-2/q_x-2/(3q)}}\right)^{\frac{3+12/q_x+6/q}{1+6/q_x+2/q}}\right].
\end{eqnarray*}
\end{theorem}

\begin{remark}
    The proof of Theorem~\ref{thm:4} is extensive due to the complex nature of the sandwich variance estimator, encompassed in Section~\ref{sec:A.Consistency of the Sandwich Variance Estimator}. Theorem~\ref{thm:4} primarily focuses on the concentration inequality for $\abs{\hat\sigma_c^2/\sigma_c^2-1}$, with the pivotal deterministic inequality established in Lemma~\ref{lem:det_bound_sand_var}. This lemma argues that the difference between the plug-in sandwich variance estimate and its population counterparts can be deterministically bounded using several factors, including $\norm{\Sigma^{-1/2}\hat\Sigma\Sigma^{-1/2}-I_d}_{\rm op}$, $\norm{\hat\beta-\beta}_\Sigma$, and $\max_{1\leq i\leq n}|X_i^\top(\hat\beta-\beta)|$, among others mentioned. Consequently, the proof relies on the concentration inequality of these quantities, which in themselves hold potential as areas of independent interest. The rate presented in Theorem~\ref{thm:4} is a result of leveraging these concentration inequalities, aiming to achieve the widest possible scaling of $d$ in terms of $n$. However, alternative choices remain viable.
\end{remark}

Theorem~\ref{thm:4} shows that for a given $c\in\Real^d$, the variance estimate $\hat\sigma_c^2$ is consistent for the asymptotic variance $\sigma_c^2$ as long as
\begin{equation}\label{eq:thm5.dim_requirement}
    d=o\left(n^{\min\{1-2/q_x,1-1/q_x-1/q,1-2/q_x-2/(3q)\}}\right),
\end{equation}ignoring the polylogarithmic factors. Under the assumptions that $(1/q_x+1/q)^{-1}\geq4$ and $(3/q_x+1/q)^{-1}\geq2$, it reduces to $d=o(n^{2/3})$. This strictly embraces the one obtained in Theorem~\ref{thm:3}, but it should be emphasized that Theorem~\ref{thm:3} relies on the moment conditions $q_x\geq8$ and $(1/q_x + 1/q)^{-1}\geq 3$.

\subsection{Berry Esseen Bound for the Studentized Bias-corrected Estimator}\label{sec:3.5}
The error bounds established for the sandwich variance estimator in Theorem~\ref{thm:4} and the Berry Esseen bound derived in Theorem~\ref{thm:3} help to derive the Berry Esseen bound for the studentized debiased LSE, leading to our main result Theorem~\ref{thm:5}. For easier comprehension, we present a simpler bound in Corollary~\ref{cor:2}.

\begin{theorem}\label{thm:5}
    Suppose that Assumption~\ref{asmp:2},\ref{asmp:3.a}, and \ref{asmp:4} holds with $q_x\geq8$, $(1/q_x+1/q)^{-1}\geq4$, and $(3/q_x+1/q)^{-1}\geq2$. Assume $d+4\log(2n)\leq n/(18K_x)^2$, and let $\epsilon_{n,d}$ and $\eta_{n,d}$ be those defined in Theorem~\ref{thm:3} and Theorem~\ref{thm:4}, respectively, with possibly different choices of constants. Then, 
    \begin{eqnarray*}
    \sup_{c\in\Real^d\setminus\{0_d\}}\sup_{t\in\Real}\Abs{\Pb\left[\sqrt{n}\frac{c^\top(\hat \beta_{\rm bc}-\beta)}{\hat\sigma_c}\leq t\right]-\left\{\Phi(t)+\kappa_c\Phi^{(3)}(t)\right\}}\leq\epsilon_{n,d}+\eta_{n,d}.
    \end{eqnarray*}
\end{theorem}

\begin{corollary}\label{cor:2}
    Suppose that the assumptions made in Theorem~\ref{thm:5} hold. Furthermore, if $q_x\geq12$, $q\geq 6$, and $d=O(n^{2/3})$, then
    \begin{equation}\label{eq:cor2:1}
        \sup_{c\in\Real^d\setminus\{0_d\}}\sup_{t\in\Real}\Abs{\Pb\left[\sqrt{n}\frac{c^\top(\hat \beta_{\rm bc}-\beta)}{\hat\sigma_c}\leq t\right]-\Phi(t)}\lesssim\frac{1}{n^{1/3}}+\sqrt{\frac{d}{n}}+\frac{d^{3/2}}{n}.
    \end{equation} Here, $\lesssim$ is a shorthand for the inequality up to constants and log-polynomial factors of $n$ and $d$.
\end{corollary}

\begin{remark}
    The bound provided in \eqref{eq:cor2:1} consists of three terms, none of which entirely dominates the others. Specifically, when the dimension is relatively `small', that is, if $d\lesssim n^{1/3}$, the dominant factor in the bound becomes $n^{-1/3}$. In a moderately dimensional scenario, i.e., when $n^{1/3}\lesssim d\lesssim n^{1/2}$, $\sqrt{d/n}$ becomes the leading factor. However, in a `large' dimension scenario where $d\gtrsim n^{1/2}$, the prevailing factor in the bound is $d^{3/2}/n$.
\end{remark}

One of the immediate consequences of Theorem~\ref{thm:5} is the asymptotic validity of the Wald confidence interval for the linear contrast of the projection parameters. Define a level-$\alpha$ Wald confidence interval with respect to a predetermined $c\in\Real^d\setminus\{0_d\}$ as
\begin{equation*}
    {\rm CI}^{\rm Wald}_{c,\alpha} = \left[c^\top\hat\beta_{\rm bc}-\Phi^{-1}\left(1-\frac{\alpha}{2}\right)\frac{\hat\sigma_c}{\sqrt{n}},c^\top\hat\beta_{\rm bc}-\Phi^{-1}\left(\frac{\alpha}{2}\right)\frac{\hat\sigma_c}{\sqrt{n}}\right].
\end{equation*}

\begin{corollary}\label{cor:3}
    Suppose that assumptions made in Theorem~\ref{thm:6} hold, and let $\epsilon_{n,d}$ and $\eta_{n,d}$ be as described therein. Then, the false coverage rate of Wald CI can be controlled as
    \begin{equation*}
        \sup_{c\in\Real^d\setminus\{0_d\}}\Abs{\Pb\left(c^\top\beta\notin {\rm CI}^{\rm Wald}_{c,\alpha}\right)-\alpha}\leq \epsilon_{n,d}+\eta_{n,d}+C\sqrt{\frac{d}{n}},
    \end{equation*} for some $C=C(K_x,K_y,\underline{\lambda})$. If $q_x\geq12$, $q\geq6$, and $d=O(n^{2/3})$, then 
    \begin{equation*}
        \sup_{c\in\Real^d\setminus\{0_d\}}\Abs{\Pb\left(c^\top\beta\notin {\rm CI}^{\rm Wald}_{c,\alpha}\right)-\alpha}\lesssim\frac{1}{n^{1/3}}+\sqrt{\frac{d}{n}}+\frac{d^{3/2}}{n}.
    \end{equation*}
\end{corollary}

\section{Confidence Region for the Projection Parameter}\label{sec:4}
In the previous section, we have shown that the bias-corrected least square estimator is asymptotically Normal distributed. In this section, we present three methods that do not require a consistent estimate of the asymptotic variance but still enable us to get valid confidence intervals for the linear contrasts of the projection parameters.

\subsection{HulC}\label{sec:4.1}
The HulC \citep{kuchibhotla2021hulc} is a statistical method that computes confidence regions for parameters by constructing convex hulls around a set of estimates.
Suppose that we split a sample of $n$ i.i.d observations $(X_1,Y_1),\ldots,(X_n,Y_n)$ into $B$ batches where each batch contains at least $\lfloor n/B\rfloor$ observations\footnote{\cite{kuchibhotla2021hulc} noted that the batches need
not be equal-sized, but having approximately equal sizes yields good width properties.}. Denote the bias-corrected least square estimators obtained from $b$:th batch as $\hat\beta_{\rm bc}^{(b)}$ for $b=1,\ldots,B$. For any $c\in\Real^d$, define the maximum median bias \citep{kuchibhotla2021hulc} for $\beta$ as
\begin{equation*}
    \Delta_c := \max_{1\leq b\leq B}\left(\frac{1}{2}-\min\left\{\Pb(c^\top(\hat\beta_{\rm bc}^{(b)}-\beta)>0),\Pb(c^\top(\hat\beta_{\rm bc}^{(b)}-\beta)<0)\right\}\right).
\end{equation*} A key idea of the HulC is that the event 
\begin{equation*}
    \min_{1\leq b\leq B}c^\top\hat\beta_{\rm bc}^{(b)}\leq c^\top\beta\leq\max_{1\leq b\leq B}c^\top\hat\beta_{\rm bc}^{(b)}
\end{equation*} occurs unless either of the following happens: (1) $c^\top(\hat\beta_{\rm bc}^{(b)}-\beta)>0$ for all $1\leq b\leq B$ or (2) $c^\top(\hat\beta_{\rm bc}^{(b)}-\beta)<0$ for all $1\leq b\leq B$. Consequently, we get
\begin{equation*}
    \Pb\left(\min_{1\leq b\leq B}c^\top\hat\beta_{\rm bc}^{(b)}\leq c^\top\beta\leq\max_{1\leq b\leq B}c^\top\hat\beta_{\rm bc}^{(b)}\right)\geq 1-\left(\frac{1}{2}-\Delta_c\right)^B-\left(\frac{1}{2}+\Delta_c\right)^B.
\end{equation*} Since the bias-corrected OLS estimator is asymptotically Normal as shown in Corollory~\ref{cor:1}, $\Delta_c$ converges to the asymptotic median, which is $0$. The detailed procedure for constructing confidence interval at level $\alpha$ is described in Algorithm~\ref{alg:1}.

\begin{theorem}\label{thm:6}
    Suppose that the assumptions made in Theorem~\ref{thm:3}(i) hold, and denote the Berry Esseen bound given in Theorem~\ref{thm:3}(i) as $\epsilon_{n,d}$. For any $\alpha\in(0,1)$ and $c\in\Real^d$ with $\norm{c}_{\Sigma^{-1}}=1$, let $B=\lceil \log (2/\alpha)\rceil$ and assume $d\leq n/B$. Let ${\rm CI}^{\rm HulC}_{c,\alpha}$ be the confidence interval returned by Algorithm~\ref{alg:1}. Then, the false coverage rate of HulC CI can be controlled as
    \begin{equation*}
        \sup_{c\in\Real^d\setminus\{0_d\}}\left\{\Pb\left(c^\top\beta\notin {\rm CI}^{\rm HulC}_{c,\alpha}\right)-\alpha\right\}\leq \alpha\left(B^2\epsilon_{\frac{n}{B},d}^2+C\frac{B^3d}{n}\right),
    \end{equation*} for some constant $C=C(K_x,K_y,\underline{\lambda})$. If $q_x\geq12$ and $d=O(n^{2/3})$, then
    \begin{equation*}
        \sup_{c\in\Real^d\setminus\{0_d\}}\left\{\Pb\left(c^\top\beta\notin {\rm CI}^{\rm HulC}_{c,\alpha}\right)-\alpha\right\}\lesssim \alpha\left(\frac{B^3d}{n}+\frac{B^4d^3}{n^2}\right).
    \end{equation*}
\end{theorem}

\begin{remark}[Comparison with Wald Confidence Interval] The Wald CI and HulC CI are both (asymptotically) valid as demonstrated in Corollary~\ref{cor:3} and Theorem~\ref{thm:6}, respectively. A notable distinction between them lies in the fact that the HulC CI does not require an estimate for the asymptotic variance. Furthermore, the HulC CI shows a sharper miscoverage rate due to its reliance on the `squared' Berry Esseen bound in Theorem~\ref{thm:3}. Despite this, the HulC CI tends to be wider due to its robust nature, which can be empirically demonstrated in Section~\ref{sec:5}. Importantly, the expected width ratio of the HulC CI to the Wald CI approximately equals $\sqrt{\log_2(\log_2(2/\alpha))}$ \citep{kuchibhotla2021hulc}, slowly diverging as $\alpha\to0+$, and reaching 1.55 for the conventional choice of $\alpha=0.05$. A more comprehensive comparison can be found in Section 2.3 of \cite{kuchibhotla2021hulc}.
    
\end{remark}

\begin{algorithm}[t]
  \caption{ \small HulC for the projection parameter}\label{alg:1}
  \algorithmicrequire \ { data $(X_1,Y_1),\ldots,(X_n,Y_n)$, target coverage $1-\alpha$, and a query vector $c\in\Real^d\setminus\{0_d\}$.}\\ 
  \algorithmicensure \ { A confidence interval ${\rm CI}^{\rm HulC}_{c,\alpha}$ for $c^\top\beta$ } 
    \begin{algorithmic}[1]
    \STATE Let $B=\lceil\log_2(2/\alpha)\rceil$ and $\tau=2^{B-1}\alpha-1$. Note that $\tau\in [0,1)$.
    \STATE Draw a random variable $U$ from the uniform distribution on the interval $[0,1]$. Let $B^* = B-1$ if $U\leq\tau$ and let $B^*=B$ otherwise.
    \STATE Split the data $(X_1,Y_1),\ldots,(X_n,Y_n)$ into $B^*$ disjoint batches where each batch contains at least $\lfloor n/B^*\rfloor$ observations. 
    \STATE Compute $B^*$ bias-corrected LS estimators $\hat\beta_{\rm bc}^{(1)},\ldots,\hat\beta_{\rm bc}^{(B^*)}$.
     \STATE \textbf{Return} the level-$\alpha$ confidence interval
     \begin{equation*}
         {\rm CI}^{\rm HulC}_{c,\alpha} = \left[\min_{1\leq b\leq B^*}c^\top\hat\beta_{\rm bc}^{(b)},\max_{1\leq b\leq B^*}c^\top\hat\beta_{\rm bc}^{(b)}\right].
     \end{equation*}
  \end{algorithmic}
\end{algorithm}

\subsection{\texorpdfstring{$t$}{}-statistic based Inference}\label{sec:4.2}
As in the previous section, consider $B$ partitions of $n$ observations. Here, the number of batches $B$ does not necessarily depend on the target coverage $1-\alpha$, and suppose that each batch has at least $\lfloor n/B\rfloor$ observations. Denote by $\hat\beta_{\rm bc}^{(b)}$ the bias-corrected estimator of $\beta$ using observations in batch $b$, $1\leq b\leq B$. Since estimators are obtained independently, this amounts to the convergence of the joint distribution;
\begin{equation*}
    \sqrt{\frac{n}{B}}\left(\frac{c^\top(\hat\beta_{\rm bc}^{(1)}-\beta)}{\sigma_c},\ldots,\frac{c^\top(\hat\beta_{\rm bc}^{(B)}-\beta)}{\sigma_c}\right)^\top\xrightarrow{D}\Nc(0_B,I_B),
\end{equation*} as $n\to\infty$ if the dimension requirement in Theorem~\ref{thm:3} are met. Hence, we can form an asymptotically valid confidence region based on the $t$-statistic,
\begin{equation*}
    T = \sqrt{B}\frac{c^\top(\overline{\hat\beta}_{\rm bc}-\beta)}{s_{\hat\beta}},
\end{equation*} where $\overline{\hat\beta}_{\rm bc}=B^{-1}\sum_{b=1}^B \hat\beta_{\rm bc}^{(b)}$ and $s_{\hat\beta}^2 = (B-1)^{-1}\sum_{b=1}^B\{c^\top(\hat\beta_{\rm bc}^{(b)}-\overline{\hat\beta}_{\rm bc})\}^2$. By the continuous mapping theorem \citep[Theorem 2.7]{billingsley2013convergence}, the given statistic $T$ is asymptotically $t$-distributed with $(B-1)$-degrees of freedom. By implication, the confidence interval for $c^\top\beta$ is given by
\begin{equation*}
    {\rm CI}^{\rm T}_{c,\alpha}:=\left[c^\top\overline{\hat\beta}_{\rm bc} - t_{\alpha/2,B-1}\frac{s_{\hat\beta}}{\sqrt{B}},c^\top\overline{\hat\beta}_{\rm bc} + t_{\alpha/2,B-1}\frac{s_{\hat\beta}}{\sqrt{B}}\right],
\end{equation*} where $t_{q,\nu}$ denotes the upper $q$th quantile of $t$-distribution with $\nu$ degrees of freedom.

\subsection{Bootstrap Inference}\label{sec:4.3}
This section presents a method for carrying out post-bias-correction inference in our scenario, based on a bootstrap approach. Specifically, we utilize the wild bootstrap technique \citep{wu1986jackknife} to derive a computationally efficient approximation of the distribution of the bias-corrected OLS estimator. The idea is, as the residual bootstrap, to leave the covariates at their sample value, but to resample the response variable based on the residual values.
First, we compute from the original dataset that the OLS estimator $\hat\beta$ and the residuals $\hat\epsilon_i:=Y_i-X_i^\top\hat\beta$ for $i=1,\ldots,n$. For $i=1,\ldots,n$, let $\xi_i^*$ be i.i.d. bootstrap weights such that $\Eb\xi_i^* =0$, $\Eb{\xi_i^*}^2=1$, and $\Eb{\xi_i^*}^3=1$. Then, we generate a bootstrap sample of $n$ observations $(X_1,Y_1^*),\ldots,(X_n,Y_n^*)$ by letting
\begin{equation*}
    Y_i^* = X_i^\top\hat\beta +\hat\epsilon_i\xi_i^*,
\end{equation*} for $i=1,\ldots,n$. Although Bootstrap weights are often drawn from the standard Normal distribution, we made use of the asymmetric weights in Section~\ref{sec:5}, which was proposed in \cite{mammen1993bootstrap}; $\Pb(\xi^*=-(\sqrt{5}-1)/2)=(\sqrt{5}+1)/(2\sqrt{5})$ and $\Pb(\xi^*=(\sqrt{5}+1)/2)=(\sqrt{5}-1)/(2\sqrt{5})$. Let estimator $\hat\beta^*_{\rm bc}$ be the bias-corrected least square estimator based on bootstrap observations,
\begin{eqnarray*}
    \hat\beta^*_{\rm bc} = \hat\beta^* - \hat\Bc^*,
\end{eqnarray*} where $\hat\beta^*$ and $\hat\Bc^*$ are bootstrap counterparts of the OLS estimator and estimated bias, respectively;
\begin{eqnarray*}
     \hat\beta^*=\hat\Sigma^{-1}\frac{1}{n}\sum_{i=1}^nX_iY_i^*\quad\mbox{and}\quad\hat\Bc^* = -\frac{1}{n^2}\sum_{i=1}^n \hat\Sigma^{-1}X_i(Y_i^*-X_i^\top\hat\beta^*)\norm{X_i}_{\hat\Sigma^{-1}}^2.
\end{eqnarray*} To construct the confidence interval, let $T_b^* = \hat\beta_{\rm bc}^{*,(b)}-\hat\beta$ where $\hat\beta_{\rm bc}^{*,(b)}$ is the bootstrap bias-corrected estimator in $b$:th simulation. The wild bootstrap level $\alpha$ bias-corrected confidence interval for $c^\top\beta$ is given by

\begin{eqnarray*}
    {\rm CI}_{\alpha}^{\rm WBS} = \left[\hat\beta_{\rm bc} -\hat q^*_{1-\alpha/2}, \hat\beta_{\rm bc} -\hat q^*_{\alpha/2}\right],
\end{eqnarray*} where $\hat q^*_{\alpha}=\inf\{t\in\Real:\hat F^*(t)\geq\alpha\}$ is the empirical upper $\alpha$:th quantile of $\{T_b^*:1\leq b \leq B\}$ with $\hat F^*(t) = \frac{1}{B}\sum_{b=1}^B \mathbbm{1}(T_b^*\leq t)$.
\begin{remark}\label{rmk:3}
    The validity of the above wild bootstrap inference may require additional dimension requirements. Theorem~\ref{thm:1} of \cite{mammen1993bootstrap} shows that wild bootstrap provides a consistent estimation of the unconditional distribution of $\sqrt{n}(\hat\beta-\beta)$ in a well-specified linear regression model, without any bias correction, under the condition that $d = o(\sqrt{n})$. More recently, Theorem~\ref{thm:1}1 of \cite{kuchibhotla2020berry} establishes a finite sample approximation of the conditional bootstrap distribution of the law of OLSE in a similar setting to ours. In particular, the proof utilizes the Gaussian comparison bound \citep{lopes2020central} to compare the conditional distribution of the bootstrap estimate with the Normal distribution. However, the estimation of asymptotic variance appears to be inevitable in the process, and this adds an additional dimension requirement of $d = o(\sqrt{n})$, particularly when $q_x\geq 8$ and $s\geq 4$ (see Lemma 10 of \cite{kuchibhotla2020berry}). Furthermore, although not discussed in a linear regression context, Lemma 4.1 of \cite{zhilova2020nonclassical} gives a sequence of random variables that justifies the necessity of the condition $d= o(\sqrt{n})$ for the consistency of the wild bootstrap when estimating sample mean.
\end{remark}

\section{Numerical Studies}\label{sec:5}
In this section, we compare the empirical coverages and widths of the 95\% confidence intervals obtained from three inferential methods discussed in Section~\ref{sec:4}. Furthermore, we also compare the bootstrap confidence interval based on the jackknife-debiased OLS proposed in \cite{cattaneo2019two}. In our specific simulation settings, which will be described shortly, we discovered that when leveraged with the bootstrap inferential methods, the jackknife-based debiased OLSE slightly outperforms the OLSE and the proposed bias-corrected OLSE. A large set of implementations, including both resampling bootstrap- and wild bootstrap-based confidence intervals using the OLS and the proposed debiased OLS, are presented in Appendix~\ref{sec:D}.

\subsection{Well-specified Linear Model}\label{sec:5.1}
This section concerns the well-specified linear model. The simulation setting is as follows: for $n\in\set{1000, 2000}$ and $d\in\set{20k:1\leq k\leq 24}$, independent observations $(Y_i,X_i)\in\Real^d\times\Real$, $1\leq i \leq n$, are generated as
\begin{eqnarray*}
    X_i\sim\Nc(0_d, I_d), \epsilon_i\sim \Nc(0, 1),\mbox{ and }Y_i = 2X_i(1)+\epsilon_i,
\end{eqnarray*}where $X_i(1)$ is the first coordinate of $X_i$, that is, $X_i(1)=\ev_1^\top X_i$. Thus, the response $Y_i$ relies only on the first dimension of $X_i$ and independent error $\epsilon_i$. As the given model is well-specified, the projection parameter is given by $\beta=2\ev_1\in\Real^d$.

Figure~\ref{fig:1} compares the empirical coverage and the width of the 95\% confidence intervals of the first coordinate of $\beta$, which is 2, obtained from various methods. It is noteworthy that our bias-corrected estimator when incorporated with HULC achieves the target coverage of 0.95. On the contrary, the $t$-statistic-based confidence interval seems to be fairly conservative while the wild bootstrap inference yields a seemingly less conservative confidence interval.

\begin{figure}[t]
     \centering
     \includegraphics[width=\textwidth]{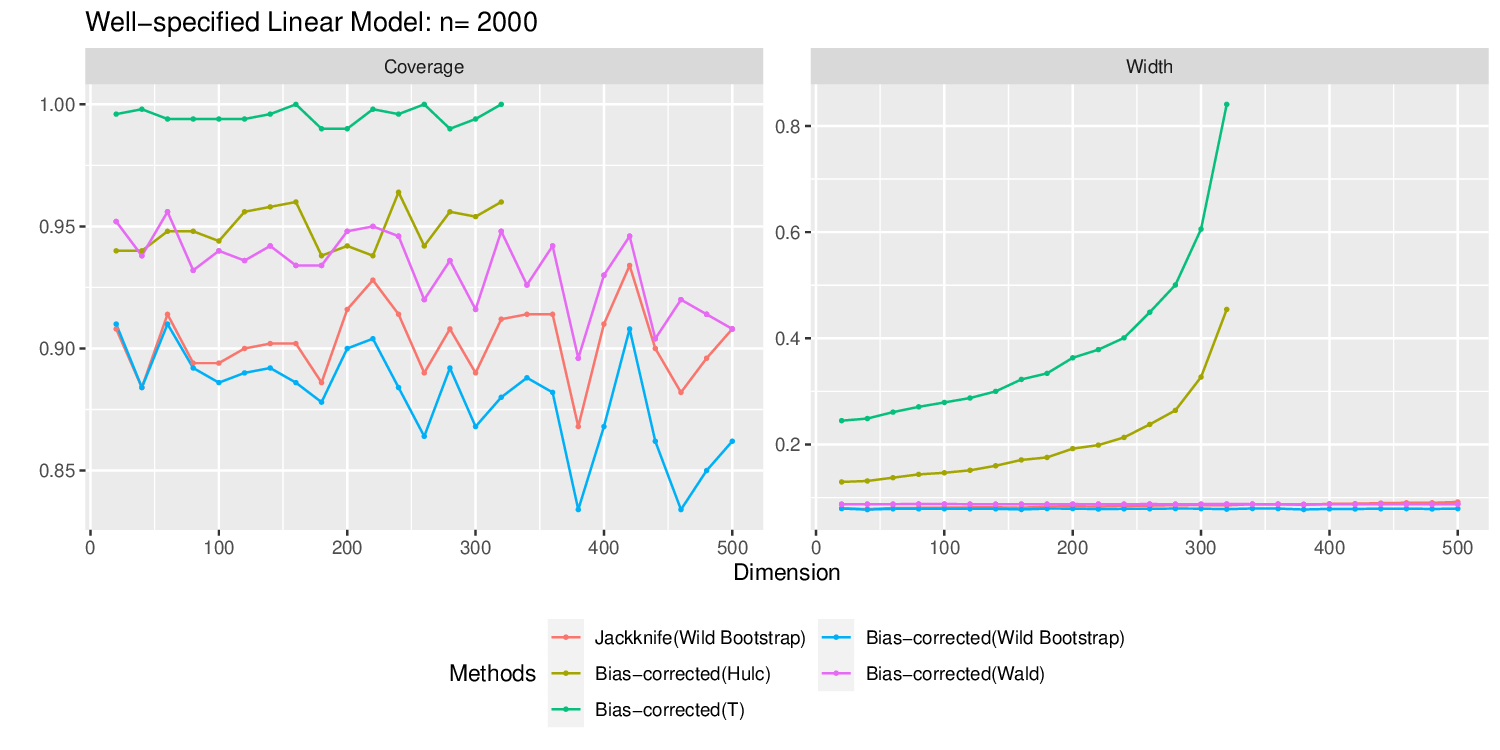}
     \caption{ Comparison of coverages and widths of confidence intervals; wild bootstrap based on Jackknife estimators, wild bootstrap based on proposed bias-corrected estimators, HulC using proposed bias-corrected estimators, and $t$-statistic based inference using proposed bias-corrected estimators, and Wald confidence interval based on proposed estimators. The empirical coverages and widths of CI are computed based on 1000 replications. To attain a 0.95-level HULC confidence interval, Algortihm~\ref{alg:1} requires data to be split into at most six subsets. This only allows the dimension to increase by $\approx$320 in order to ensure that the least square estimator is well-defined within a subset of data. Furthermore, $t$-statistics-based inference also requires data splitting, and we used the same data split as employed in HulC CI for the sake of simplicity. 
     }
     \label{fig:1}
\end{figure}

\subsection{Misspecified Model}\label{sec:5.2}
In this section, we concern about a misspecified non-linear model. The data-generating process for individual observation is as follows. We first independently generate two $d$-dimensional Gaussian random vector 
\begin{eqnarray*}
    Z\sim\Nc(0_d,I_d) \mbox{ and }W\sim\Nc(0_d,\Sigma_\rho).
\end{eqnarray*} Here, the covariance matrix of $W$, $\Sigma_\rho$, is a compound symmetry where the diagonal entries are all $1$ and the off-diagonal elements are all $\rho$ for $\rho\in[0,1)$. That is, $\Sigma_\rho = (1-\rho)I_d + \rho 1_d1_d^\top$. We define the covariate vector as $X = Z\odot W$ where $\odot$ denotes the entry-wise product. Then, the entries of $X$ are uncorrelated, i.e., $\Eb[XX^\top]=I_d$, but not necessarily independent except the case when $\rho=0$. For a $d$-dimensional parameter $\theta\in\Real^d$, we let 
\begin{equation}\label{eq:mis_dgp}
    Y=(X^\top\theta)^3+\epsilon,
\end{equation} where $\epsilon\sim\Nc(0,1)$ and $\epsilon$ is independent with $Z$ and $W$. 

Under the aforementioned data-generating process, we conducted comprehensive experiments with the following sample sizes, dimensions, $\theta$, and $\rho$;
\begin{eqnarray*}
    n&\in&\set{1000, 2000, 5000, 10000, 20000},\\
    d&\in&\set{50k: 1\leq k \leq 20},\\
    \theta &\in&\set{\ev_1, 1_d/\sqrt{d}}\subset\Real^d,\\
    \rho&\in&\set{0, 0.2, 0.5}.
\end{eqnarray*}
Under the model \eqref{eq:mis_dgp}, the projection parameter $\beta$ has the closed-form representation and turns out to be determined by $\theta$ and $\rho$ as,
\begin{eqnarray*}
    \beta = 3(1+2\rho^2)\norm{\theta}^2_2\theta + 6(1-\rho^2)\theta^{\odot 3},
\end{eqnarray*} where $\theta^{\odot 3} = \theta\odot\theta\odot\theta$ (see Lemma~\ref{lem:aux:1}). The linear contrast of interest was set to $\theta^\top\beta$ which becomes the first coordinate of $\beta$ when $\theta=\ev_1$ and becomes the scaled average of coefficients of $\beta$ when $\theta=1_d/\sqrt{d}$. Furthermore, we can show that
\begin{align*}
    \theta^\top\beta=
    \begin{cases}
    9,& \mbox{ if }\theta=\ev_1,\\
    3(1+2\rho^2)+6(1-\rho^2)/d, & \mbox{ if }\theta=1_d^\top/\sqrt{d};
    \end{cases}.
\end{align*} Figure~\ref{fig:3} compares the coverage and length of the 95\% confidence intervals for $\theta^\top\beta$ under the setting where $n=20000$, $\theta=\ev_1$, and $\rho=0$ or $0.5$. Results under different combinations of sample size, dimension, $\theta$, and $\rho$ are contained in Appendix \ref{sec:D}.

\begin{figure}
     \centering
     \begin{subfigure}[b]{\textwidth}
         \centering
         \includegraphics[width=\textwidth]{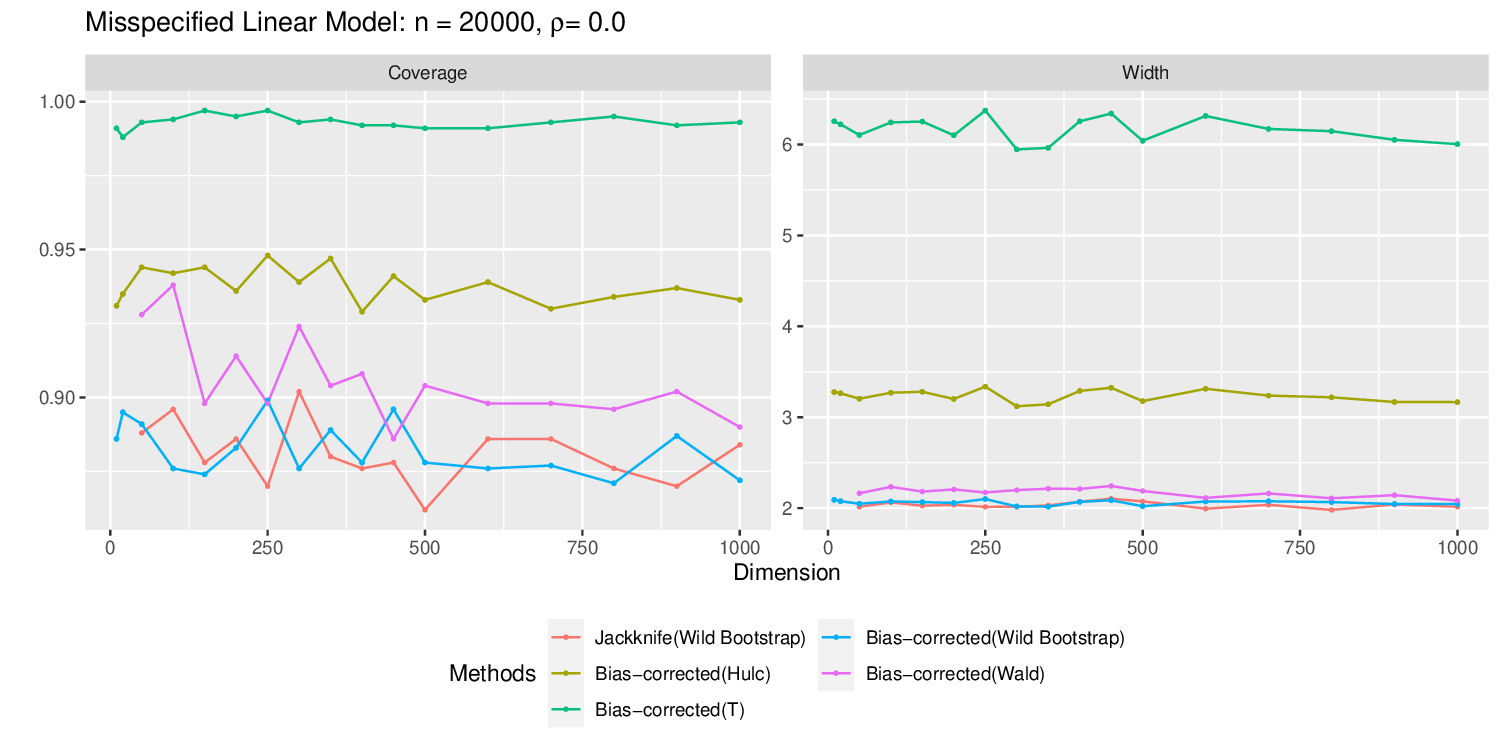}
     \end{subfigure}\\
     \begin{subfigure}[b]{\textwidth}
         \centering
         \includegraphics[width=\textwidth]{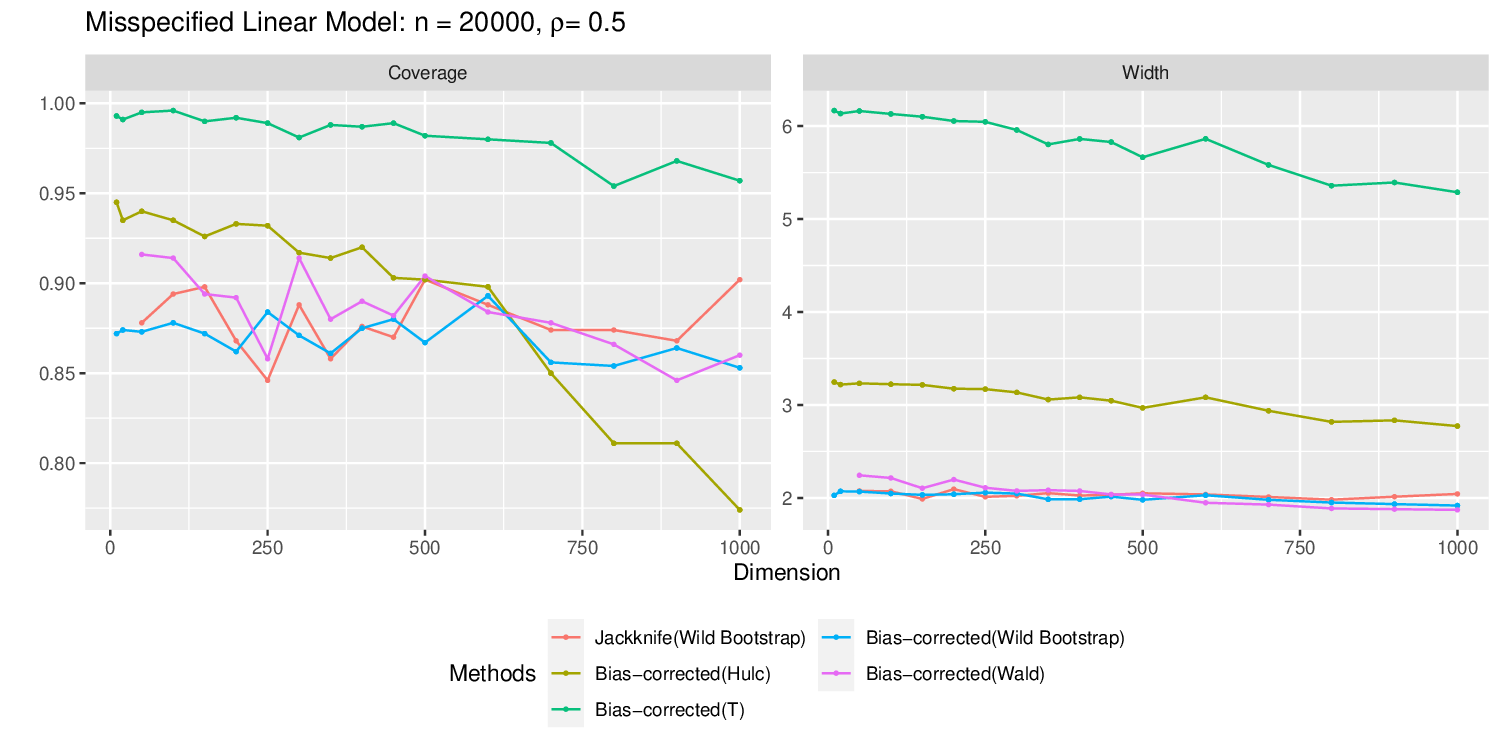}
     \end{subfigure}
     \caption{Comparison of coverages and widths of confidence intervals obtained from 5 different methods under the misspecified model with $n=20000$, $\theta=\ev_1$. The model parameter $\rho$ is set to $0.0$ (top) and $0.5$ (bottom). See Figure~\ref{fig:1} for the abbreviations of methods. The empirical coverages and widths of CI are computed based on 1000. It is noteworthy that $n^{1/2}\approx 141$ and $n^{2/3}\approx 737$.}
     \label{fig:3}
\end{figure}
\section{Discussion}
We provide an estimator for the bias of the ordinary least squares estimator that is consistent as long as the dimension $d$ grows slower than $n^{2/3}$, where $n$ is the sample size. The resulting debiased least squares estimator, after proper normalization, is asymptotically Normal at an $n^{1/2}$-rate as long as $d = o(n^{2/3})$. We also provide valid inference along any arbitrary direction for the projection parameter without having to estimate the variance. We achieve this inferential goal by leveraging the methods such as HulC. We believe the results of this paper can be extended to further expand the allowed growth rate of dimension. This extension would involve further expansion of the least squares estimator into higher order $U$-statistics and removing the bias.
\paragraph{Acknowledgements.}  A. K. Kuchibhotla and A. Rinaldo were partially supported by NSF DMS-2113611.
\bibliography{bib}

\begin{thebibliography}{59}
\providecommand{\natexlab}[1]{#1}
\providecommand{\url}[1]{\texttt{#1}}
\expandafter\ifx\csname urlstyle\endcsname\relax
  \providecommand{\doi}[1]{doi: #1}\else
  \providecommand{\doi}{doi: \begingroup \urlstyle{rm}\Url}\fi

\bibitem[Abdalla(2023)]{abdalla2023covariance}
Pedro Abdalla.
\newblock Covariance estimation under missing observations and $ l\_4-l\_2 $
  moment equivalence.
\newblock \emph{arXiv preprint arXiv:2305.12981}, 2023.

\bibitem[Adamczak et~al.(2010)Adamczak, Litvak, Pajor, and
  Tomczak-Jaegermann]{adamczak2010quantitative}
Rados{\l}aw Adamczak, Alexander Litvak, Alain Pajor, and Nicole
  Tomczak-Jaegermann.
\newblock Quantitative estimates of the convergence of the empirical covariance
  matrix in log-concave ensembles.
\newblock \emph{Journal of the American Mathematical Society}, 23\penalty0
  (2):\penalty0 535--561, 2010.

\bibitem[Bentkus et~al.(2009)Bentkus, Jing, and Zhou]{bentkus2009normal}
Vidmantas Bentkus, Bing-Yi Jing, and Wang Zhou.
\newblock {On normal approximations to U-statistics}.
\newblock \emph{The Annals of Probability}, 37\penalty0 (6):\penalty0 2174 --
  2199, 2009.

\bibitem[Berk et~al.(2019)Berk, Buja, Brown, George, Kuchibhotla, Su, and
  Zhao]{berk2019assumption}
Richard Berk, Andreas Buja, Lawrence Brown, Edward George, Arun~Kumar
  Kuchibhotla, Weijie Su, and Linda Zhao.
\newblock Assumption lean regression.
\newblock \emph{The American Statistician}, 2019.

\bibitem[Bickel and Freedman(1983)]{bickel1983bootstrapping}
Peter~J Bickel and David~A Freedman.
\newblock Bootstrapping regression models with many parameters.
\newblock \emph{Festschrift for Erich L. Lehmann}, pages 28--48, 1983.

\bibitem[Billingsley(2013)]{billingsley2013convergence}
Patrick Billingsley.
\newblock \emph{Convergence of probability measures}.
\newblock John Wiley \& Sons, 2013.

\bibitem[Brailovskaya and van Handel(2023)]{brailovskaya2023universality}
Tatiana Brailovskaya and Ramon van Handel.
\newblock Universality and sharp matrix concentration inequalities, 2023.

\bibitem[Buja et~al.(2019{\natexlab{a}})Buja, Brown, Berk, George, Pitkin,
  Traskin, Zhang, and Zhao]{buja2019models1}
Andreas Buja, Lawrence Brown, Richard Berk, Edward George, Emil Pitkin, Mikhail
  Traskin, Kai Zhang, and Linda Zhao.
\newblock Models as approximations {I}: Consequences illustrated with linear
  regression.
\newblock \emph{Statistical Science}, 34\penalty0 (4):\penalty0 523--544,
  2019{\natexlab{a}}.

\bibitem[Buja et~al.(2019{\natexlab{b}})Buja, Brown, Kuchibhotla, Berk, George,
  and Zhao]{buja2019models2}
Andreas Buja, Lawrence Brown, Arun~Kumar Kuchibhotla, Richard Berk, Edward
  George, and Linda Zhao.
\newblock Models as approximations {II}: A model-free theory of parametric
  regression.
\newblock \emph{Statistical Science}, 34\penalty0 (4):\penalty0 545--565,
  2019{\natexlab{b}}.

\bibitem[Catoni(2016)]{catoni2016pac}
Olivier Catoni.
\newblock Pac-bayesian bounds for the gram matrix and least squares regression
  with a random design.
\newblock \emph{arXiv preprint arXiv:1603.05229}, 2016.

\bibitem[Cattaneo et~al.(2015)Cattaneo, Jansson, and
  Newey]{cattaneo2015treatment}
Matias~D Cattaneo, Michael Jansson, and Whitney~K Newey.
\newblock Treatment effects with many covariates and heteroskedasticity.
\newblock Technical report, cemmap working paper, 2015.

\bibitem[Cattaneo et~al.(2018{\natexlab{a}})Cattaneo, Jansson, and
  Newey]{cattaneo2018alternative}
Matias~D Cattaneo, Michael Jansson, and Whitney~K Newey.
\newblock Alternative asymptotics and the partially linear model with many
  regressors.
\newblock \emph{Econometric Theory}, 34\penalty0 (2):\penalty0 277--301,
  2018{\natexlab{a}}.

\bibitem[Cattaneo et~al.(2018{\natexlab{b}})Cattaneo, Jansson, and
  Newey]{cattaneo2018inference}
Matias~D Cattaneo, Michael Jansson, and Whitney~K Newey.
\newblock Inference in linear regression models with many covariates and
  heteroscedasticity.
\newblock \emph{Journal of the American Statistical Association}, 113\penalty0
  (523):\penalty0 1350--1361, 2018{\natexlab{b}}.

\bibitem[Cattaneo et~al.(2019)Cattaneo, Jansson, and Ma]{cattaneo2019two}
Matias~D Cattaneo, Michael Jansson, and Xinwei Ma.
\newblock Two-step estimation and inference with possibly many included
  covariates.
\newblock \emph{The Review of Economic Studies}, 86\penalty0 (3):\penalty0
  1095--1122, 2019.

\bibitem[Donald and Newey(1994)]{DONALD199430}
S.G. Donald and W.K. Newey.
\newblock Series estimation of semilinear models.
\newblock \emph{Journal of Multivariate Analysis}, 50\penalty0 (1):\penalty0
  30--40, 1994.
\newblock ISSN 0047-259X.
\newblock \doi{https://doi.org/10.1006/jmva.1994.1032}.
\newblock URL
  \url{https://www.sciencedirect.com/science/article/pii/S0047259X84710323}.

\bibitem[Einmahl and Li(2008)]{einmahl2008characterization}
Uwe Einmahl and Deli Li.
\newblock Characterization of {LIL} behavior in {B}anach space.
\newblock \emph{Transactions of the American Mathematical Society},
  360\penalty0 (12):\penalty0 6677--6693, 2008.

\bibitem[Gu{\'e}don and Rudelson(2007)]{guedon2007lp}
Olivier Gu{\'e}don and Mark Rudelson.
\newblock Lp-moments of random vectors via majorizing measures.
\newblock \emph{Advances in Mathematics}, 208\penalty0 (2):\penalty0 798--823,
  2007.

\bibitem[Gu{\'e}don et~al.(2017)Gu{\'e}don, Litvak, Pajor, and
  Tomczak-Jaegermann]{guedon2017interval}
Olivier Gu{\'e}don, Alexander~E Litvak, Alain Pajor, and Nicole
  Tomczak-Jaegermann.
\newblock On the interval of fluctuation of the singular values of random
  matrices.
\newblock \emph{Journal of the European Mathematical Society (EMS Publishing)},
  19\penalty0 (5), 2017.

\bibitem[Han and Wellner(2019)]{han2019convergence}
Qiyang Han and Jon~A Wellner.
\newblock Convergence rates of least squares regression estimators with
  heavy-tailed errors.
\newblock \emph{The Annals of Statistics}, 47\penalty0 (4):\penalty0
  2286--2319, 2019.

\bibitem[Ibragimov and M{\"u}ller(2010)]{ibragimov2010t}
Rustam Ibragimov and Ulrich~K M{\"u}ller.
\newblock t-statistic based correlation and heterogeneity robust inference.
\newblock \emph{Journal of Business \& Economic Statistics}, 28\penalty0
  (4):\penalty0 453--468, 2010.

\bibitem[Jochmans(2022)]{jochmans2022heteroscedasticity}
Koen Jochmans.
\newblock Heteroscedasticity-robust inference in linear regression models with
  many covariates.
\newblock \emph{Journal of the American Statistical Association}, 117\penalty0
  (538):\penalty0 887--896, 2022.

\bibitem[Koltchinskii and Lounici(2017)]{koltchinskii2017concentration}
Vladimir Koltchinskii and Karim Lounici.
\newblock Concentration inequalities and moment bounds for sample covariance
  operators.
\newblock \emph{Bernoulli}, 23\penalty0 (1):\penalty0 110--133, 2017.

\bibitem[Kuchibhotla and Patra(2022)]{kuchibhotla2022least}
Arun~K Kuchibhotla and Rohit~K Patra.
\newblock On least squares estimation under heteroscedastic and heavy-tailed
  errors.
\newblock \emph{The Annals of Statistics}, 50\penalty0 (1):\penalty0 277--302,
  2022.

\bibitem[Kuchibhotla et~al.(2020)Kuchibhotla, Rinaldo, and
  Wasserman]{kuchibhotla2020berry}
Arun~Kumar Kuchibhotla, Alessandro Rinaldo, and Larry Wasserman.
\newblock {Berry-Esseen} bounds for projection parameters and partial
  correlations with increasing dimension.
\newblock \emph{arXiv preprint arXiv:2007.09751}, 2020.

\bibitem[Kuchibhotla et~al.(2021)Kuchibhotla, Balakrishnan, and
  Wasserman]{kuchibhotla2021hulc}
Arun~Kumar Kuchibhotla, Sivaraman Balakrishnan, and Larry Wasserman.
\newblock The {HulC}: Confidence regions from convex hulls.
\newblock \emph{arXiv preprint arXiv:2105.14577}, 2021.

\bibitem[Lam(2022)]{lam2022cheap}
Henry Lam.
\newblock A cheap bootstrap method for fast inference.
\newblock \emph{arXiv preprint arXiv:2202.00090}, 2022.

\bibitem[Long and Ervin(2000)]{long2000using}
J~Scott Long and Laurie~H Ervin.
\newblock Using heteroscedasticity consistent standard errors in the linear
  regression model.
\newblock \emph{The American Statistician}, 54\penalty0 (3):\penalty0 217--224,
  2000.

\bibitem[Lopes(2020)]{lopes2020central}
Miles~E Lopes.
\newblock Central limit theorem and bootstrap approximation in high dimensions
  with near $1/\sqrt{n}$ rates.
\newblock \emph{arXiv preprint arXiv:2009.06004}, 2020.

\bibitem[MacKinnon(2012)]{mackinnon2012thirty}
James~G MacKinnon.
\newblock Thirty years of heteroskedasticity-robust inference.
\newblock In \emph{Recent advances and future directions in causality,
  prediction, and specification analysis: Essays in honor of Halbert L. White
  Jr}, pages 437--461. Springer, 2012.

\bibitem[MacKinnon and White(1985)]{mackinnon1985some}
James~G MacKinnon and Halbert White.
\newblock Some heteroskedasticity-consistent covariance matrix estimators with
  improved finite sample properties.
\newblock \emph{Journal of econometrics}, 29\penalty0 (3):\penalty0 305--325,
  1985.

\bibitem[Mammen(1993)]{mammen1993bootstrap}
Enno Mammen.
\newblock Bootstrap and wild bootstrap for high dimensional linear models.
\newblock \emph{The annals of statistics}, 21\penalty0 (1):\penalty0 255--285,
  1993.

\bibitem[Mendelson(2016)]{mendelson2016upper}
Shahar Mendelson.
\newblock Upper bounds on product and multiplier empirical processes.
\newblock \emph{Stochastic Processes and their Applications}, 126\penalty0
  (12):\penalty0 3652--3680, 2016.

\bibitem[Mendelson and Paouris(2012)]{mendelson2012generic}
Shahar Mendelson and Grigoris Paouris.
\newblock On generic chaining and the smallest singular value of random
  matrices with heavy tails.
\newblock \emph{Journal of Functional Analysis}, 262\penalty0 (9):\penalty0
  3775--3811, 2012.

\bibitem[Mendelson and Paouris(2014)]{mendelson2014singular}
Shahar Mendelson and Grigoris Paouris.
\newblock On the singular values of random matrices.
\newblock \emph{Journal of the European Mathematical Society (EMS Publishing)},
  16\penalty0 (4), 2014.

\bibitem[Mourtada(2022)]{mourtada2022exact}
Jaouad Mourtada.
\newblock Exact minimax risk for linear least squares, and the lower tail of
  sample covariance matrices.
\newblock \emph{The Annals of Statistics}, 50\penalty0 (4):\penalty0
  2157--2178, 2022.

\bibitem[Mourtada et~al.(2022)Mourtada, Va{\v{s}}kevi{\v{c}}ius, and
  Zhivotovskiy]{mourtada2022distribution}
Jaouad Mourtada, Tomas Va{\v{s}}kevi{\v{c}}ius, and Nikita Zhivotovskiy.
\newblock Distribution-free robust linear regression.
\newblock \emph{Mathematical Statistics and Learning}, 4\penalty0 (3):\penalty0
  253--292, 2022.

\bibitem[Oliveira(2016)]{oliveira2016lower}
Roberto~Imbuzeiro Oliveira.
\newblock The lower tail of random quadratic forms with applications to
  ordinary least squares.
\newblock \emph{Probability Theory and Related Fields}, 166:\penalty0
  1175--1194, 2016.

\bibitem[Pfanzagl(1973)]{pfanzagl1973accuracy}
Johann Pfanzagl.
\newblock The accuracy of the normal approximation for estimates of vector
  parameters.
\newblock \emph{Zeitschrift f{\"u}r Wahrscheinlichkeitstheorie und verwandte
  Gebiete}, 25\penalty0 (3):\penalty0 171--198, 1973.

\bibitem[Portnoy(1984)]{portnoy1984asymptotic}
Stephen Portnoy.
\newblock Asymptotic behavior of $ m $-estimators of $ p $ regression
  parameters when $ p^2/n $ is large. i. consistency.
\newblock \emph{The Annals of Statistics}, 12\penalty0 (4):\penalty0
  1298--1309, 1984.

\bibitem[Portnoy(1985)]{portnoy1985asymptotic}
Stephen Portnoy.
\newblock Asymptotic behavior of $ m $ estimators of $ p $ regression
  parameters when $ p^2/n $ is large; ii. normal approximation.
\newblock \emph{The Annals of Statistics}, 13\penalty0 (4):\penalty0
  1403--1417, 1985.

\bibitem[Portnoy(1986)]{portnoy1986asymptotic}
Stephen Portnoy.
\newblock Asymptotic behavior of the empiric distribution of m-estimated
  residuals from a regression model with many parameters.
\newblock \emph{The Annals of Statistics}, 14\penalty0 (3):\penalty0
  1152--1170, 1986.

\bibitem[Portnoy(1988)]{portnoy1988asymptotic}
Stephen Portnoy.
\newblock Asymptotic behavior of likelihood methods for exponential families
  when the number of parameters tends to infinity.
\newblock \emph{The Annals of Statistics}, pages 356--366, 1988.

\bibitem[Puchkin et~al.(2023)Puchkin, Noskov, and Spokoiny]{puchkin2023sharper}
Nikita Puchkin, Fedor Noskov, and Vladimir Spokoiny.
\newblock Sharper dimension-free bounds on the frobenius distance between
  sample covariance and its expectation.
\newblock \emph{arXiv preprint arXiv:2308.14739}, 2023.

\bibitem[Rinaldo et~al.(2019)Rinaldo, Wasserman, and
  G’Sell]{rinaldo2019bootstrapping}
Alessandro Rinaldo, Larry Wasserman, and Max G’Sell.
\newblock Bootstrapping and sample splitting for high-dimensional,
  assumption-lean inference.
\newblock \emph{The Annals of Statistics}, 47\penalty0 (6):\penalty0
  3438--3469, 2019.

\bibitem[Rio(2017)]{rio2017constants}
Emmanuel Rio.
\newblock About the constants in the fuk-nagaev inequalities.
\newblock \emph{HAL}, 2017, 2017.

\bibitem[Rudelson(1999)]{rudelson1999random}
Mark Rudelson.
\newblock Random vectors in the isotropic position.
\newblock \emph{Journal of Functional Analysis}, 164\penalty0 (1):\penalty0
  60--72, 1999.

\bibitem[Srivastava and Vershynin(2013)]{srivastava2013covariance}
Nikhil Srivastava and Roman Vershynin.
\newblock Covariance estimation for distributions with 2+ $\varepsilon$
  moments.
\newblock \emph{Annals of probability: An official journal of the Institute of
  Mathematical Statistics}, 41\penalty0 (5):\penalty0 3081--3111, 2013.

\bibitem[Tikhomirov(2018)]{tikhomirov2018sample}
Konstantin Tikhomirov.
\newblock Sample covariance matrices of heavy-tailed distributions.
\newblock \emph{International Mathematics Research Notices}, 2018\penalty0
  (20):\penalty0 6254--6289, 2018.

\bibitem[Tropp(2016)]{tropp2016expected}
Joel~A Tropp.
\newblock The expected norm of a sum of independent random matrices: An
  elementary approach.
\newblock In \emph{High Dimensional Probability VII: The Carg{\`e}se Volume},
  pages 173--202. Springer, 2016.

\bibitem[Tropp et~al.(2015)]{tropp2015introduction}
Joel~A Tropp et~al.
\newblock An introduction to matrix concentration inequalities.
\newblock \emph{Foundations and Trends{\textregistered} in Machine Learning},
  8\penalty0 (1-2):\penalty0 1--230, 2015.

\bibitem[Van~der Vaart(2000)]{van2000asymptotic}
Aad~W Van~der Vaart.
\newblock \emph{Asymptotic statistics}, volume~3.
\newblock Cambridge university press, 2000.

\bibitem[Vansteelandt and Dukes(2022)]{vansteelandt2022assumption}
Stijn Vansteelandt and Oliver Dukes.
\newblock Assumption-lean inference for generalised linear model parameters.
\newblock \emph{Journal of the Royal Statistical Society Series B: Statistical
  Methodology}, 84\penalty0 (3):\penalty0 657--685, 2022.

\bibitem[Va{\v{s}}kevi{\v{c}}ius and
  Zhivotovskiy(2023)]{vavskevivcius2023suboptimality}
Tomas Va{\v{s}}kevi{\v{c}}ius and Nikita Zhivotovskiy.
\newblock Suboptimality of constrained least squares and improvements via
  non-linear predictors.
\newblock \emph{Bernoulli}, 29\penalty0 (1):\penalty0 473--495, 2023.

\bibitem[Vershynin(2012)]{vershynin2012close}
Roman Vershynin.
\newblock How close is the sample covariance matrix to the actual covariance
  matrix?
\newblock \emph{Journal of Theoretical Probability}, 25\penalty0 (3):\penalty0
  655--686, 2012.

\bibitem[Vershynin(2018)]{vershynin2018high}
Roman Vershynin.
\newblock \emph{High-dimensional probability: An introduction with applications
  in data science}, volume~47.
\newblock Cambridge university press, 2018.

\bibitem[White(1980)]{white1980heteroskedasticity}
Halbert White.
\newblock A heteroskedasticity-consistent covariance matrix estimator and a
  direct test for heteroskedasticity.
\newblock \emph{Econometrica: journal of the Econometric Society}, pages
  817--838, 1980.

\bibitem[Wu(1986)]{wu1986jackknife}
Chien-Fu~Jeff Wu.
\newblock Jackknife, bootstrap and other resampling methods in regression
  analysis.
\newblock \emph{the Annals of Statistics}, 14\penalty0 (4):\penalty0
  1261--1295, 1986.

\bibitem[Yang and Kuchibhotla(2021)]{yang2021finite}
Yachong Yang and Arun~Kumar Kuchibhotla.
\newblock Finite-sample efficient conformal prediction.
\newblock \emph{arXiv preprint arXiv:2104.13871}, 2021.

\bibitem[Zhilova(2020)]{zhilova2020nonclassical}
M~Zhilova.
\newblock Nonclassical {Berry--Esseen} inequalities and accuracy of the
  bootstrap.
\newblock \emph{Annals of Statistics}, 48\penalty0 (4):\penalty0 1922--1939,
  2020.

\end{thebibliography}
\bibliographystyle{plainnat}

\appendix
\section{Proofs of Theorems and Corollaries}

\subsection{Proof of Theorem~\ref{thm:1}}
Recall that the OLS estimator can be expressed as $\hat\beta = \beta + \hat\Sigma^{-1}\frac{1}{n}\sum_{i=1}^nX_i(Y_i-X_i^\top\beta)$. We write Taylor series expansion (up to the second order) of the sample Gram matrix $\hat\Sigma$ at $\Sigma$ as
\begin{eqnarray*}
    \hat\Sigma^{-1}-\Sigma^{-1}&=&\hat\Sigma^{-1}(\Sigma-\hat\Sigma)\Sigma^{-1}\\
    &=&\Sigma^{-1}(\Sigma-\hat\Sigma)\Sigma^{-1}+(\hat\Sigma^{-1}-\Sigma^{-1})(\Sigma-\hat\Sigma)\Sigma^{-1}\\
    &=&\Sigma^{-1}(\Sigma-\hat\Sigma)\Sigma^{-1}+\hat\Sigma^{-1}(\Sigma-\hat\Sigma)\Sigma^{-1}(\Sigma-\hat\Sigma)\Sigma^{-1}.
\end{eqnarray*} Plugging the expansion in the OLS expression gives that
\begin{eqnarray}
    \hat\beta -\beta
    &=&\Sigma^{-1}\n\sum_{i=1}^nX_i(Y_i-X_i\beta)+\Sigma^{-1}(\Sigma-\hat\Sigma)\Sigma^{-1}\n\sum_{i=1}^nX_i(Y_i-X_i^\top\beta)\label{eq:A.1.1}\\
    &&+\,\hat\Sigma^{-1}(\Sigma-\hat\Sigma)\Sigma^{-1}(\Sigma-\hat\Sigma)\Sigma^{-1}\n\sum_{i=1}^nX_i(Y_i-X_i^\top\beta).\label{eq:A.1.2}
\end{eqnarray} As in \eqref{eq:2.2.6}, the low-order approximation \eqref{eq:A.1.1} of $\hat\beta-\beta$ can be decomposed into two components: $\Bc$ and $\Uc$. In this decomposition, $\Bc$ represents the bias, which is precisely defined in \eqref{eq:bias_true}, while $\Uc$ is a combination of first and second-order $U$-statistics, expressed as:
\begin{eqnarray*}
    \Uc = \left(1+\frac{1}{n}\right)\sum_{i=1}^n\psi(X_i,Y_i)+\frac{1}{n(n-1)}\sum_{1\leq i\neq j\leq n}\phi(X_i,Y_i,X_j,Y_j),
\end{eqnarray*} and $\psi$ and $\phi$ are defined in \eqref{eq:psi_and_phi}. Finally, we denote the approximation error in \eqref{eq:A.1.2} as $\Rc$.

The proof involves two steps of approximations; (1) an approximation of the distribution of $\sqrt{n}c^\top(\hat\beta-\beta-\Bc)$ to that of $\sqrt{n}\,c^\top\Uc$ and (2) an approximation of the distribution of $\sqrt{n}\,c^\top\Uc$ to the adjusted Normal distribution. For any $x\in\Real$ and $\epsilon>0$,
\begin{eqnarray}\label{eq:A.2.1}
    \Pb(\sqrt{n}c^\top(\hat\beta-\beta-\Bc)\leq x)&=&\Pb(\sqrt{n}c^\top(\Uc+\Rc)\leq x)\nonumber\\
    &\leq& \Pb(\sqrt{n}\,c^\top\Uc\leq x+\epsilon)+\Pb(\Abs{\sqrt{n}c^\top\Rc}>\epsilon).
\end{eqnarray}
The first term on the right-hand side can be approximated to the adjusted Normal percentile by Lemma~\ref{lem:11} as it says that
\begin{eqnarray}\label{eq:A.2.2}
    \sup_{x\in\Real}\Abs{\Pb(\sqrt{n}\,c^\top\Uc\leq x+\epsilon)-N_c\left(\frac{x+\epsilon}{\sigma_c}\right)}\leq \Delta_n,
\end{eqnarray} where we write $N_c(x) = \Phi(x)-\kappa_c\Phi^{(3)}(x)$ (see \eqref{eq:2.2.6} for the definition of $\kappa_c$). In addition, we have
\begin{equation*}
    \Delta_n = C\left(\frac{\bar\lambda^{3/2}K_x^3K_y^3}{\sqrt{n}}+\frac{\bar\lambda K_x^6K_y^2d}{2n}\right)=O\left(\frac{1}{\sqrt{n}}+\frac{d}{n}\right),
\end{equation*} for some absolute constant $C>0$. Combining \eqref{eq:A.2.1} and \eqref{eq:A.2.2} leads to
\begin{eqnarray}\label{eq:A.2.3}
    \Pb(\sqrt{n}c^\top(\hat\beta-\beta-\Bc)\leq x)-N_c\left(\frac{x}{\sigma_c}\right) &\leq& N_c\left(\frac{x+\epsilon}{\sigma_c}\right)-N_c\left(\frac{x}{\sigma_c}\right) +\Delta_n+\Pb(\Abs{\sqrt{n}c^\top\Rc}>\epsilon)\nonumber\\
    &\leq&\frac{\epsilon}{\sigma_c}\sup_{t\in\Real}\Abs{N_c'(t)}+\Delta_n+\Pb(\Abs{\sqrt{n}c^\top\Rc}>\epsilon).
\end{eqnarray} Instead of \eqref{eq:A.2.1}, we can similarly argue that
\begin{equation*}
    \Pb(\sqrt{n}c^\top(\hat\beta-\beta-\Bc)> x)\leq \Pb(\sqrt{n}\,c^\top\Uc> x+\epsilon)+\Pb(\Abs{\sqrt{n}c^\top\Rc}>\epsilon),
\end{equation*} for any $x\in\Real$ and $\epsilon>0$. This leads to an analog of \eqref{eq:A.2.3}, and consequently, we have
\begin{equation*}
    \Abs{\Pb(\sqrt{n}c^\top(\hat\beta-\beta-\Bc)\leq x)-N_c\left(\frac{x}{\sigma_c}\right)}\leq\frac{\epsilon}{\sigma_c}\sup_{t\in\Real}\Abs{N_c'(t)}+\Delta_n+\Pb(\Abs{\sqrt{n}c^\top\Rc}>\epsilon).
\end{equation*} It follows from the definition of $N_c$ that $\sup_{t\in\Real}\Abs{N_c'(t)}=(1+3\kappa_c)/\sqrt{2\pi}$, and $\kappa_c$ can be further bounded through Lemma~\ref{lem:11} essentially leading to
\begin{eqnarray}\label{eq:A.2.4}
    \frac{1}{\sigma_c}\sup_{t\in\Real}\Abs{N_c'(t)}\leq\frac{1}{\sqrt{2\pi}\sigma_c}\left(1+\frac{3K_x^3K_y}{2\sigma_c}\sqrt{\frac{d}{n}}\right)&\leq&\frac{1}{\sqrt{2\pi}\sigma_c}\left(1+\frac{3K_x^3K_y}{2\sigma_c}\right)\nonumber\\
    &\leq&\sqrt{\frac{1}{2\pi\underline{\lambda}}}\left(1+\frac{3K_x^3K_y}{2\underline{\lambda}^{1/2}}\right).
\end{eqnarray} We used $d\leq n$ for the penultimate inequality, and the last inequality is due to that $\sigma_c^2\geq \underline{\lambda}$ which follows from Assumption~\ref{asmp:4}. Hence, we have
\begin{equation}\label{eq:A.2.5}
    \Abs{\Pb(\sqrt{n}c^\top(\hat\beta-\beta-\Bc)\leq x)-N_c\left(\frac{x}{\sigma_c}\right)} \leq \Delta_n + C\epsilon +\Pb\left(|\sqrt{n}c^\top\Rc|>\epsilon\right),
\end{equation} 
for $C=C(\underline{\lambda},K_x,K_y)$ in \eqref{eq:A.2.4}. The next and final step involves harnessing the value of $\epsilon$ on the right-hand side of \eqref{eq:A.2.5}, which will serve as our Berry Esseen bound. It is worth highlighting that Lemma \ref{lem:8}---\ref{lem:10} establishes the tail bound for the quantity $\sqrt{n}c^\top\Rc$, contingent upon three different moment assumptions on covariates (Assumption~\ref{asmp:3.a}---\ref{asmp:3.c}). In particular, their results have the form of
\begin{equation*}
    \Pb\left(|\sqrt{n}c^\top\Rc|>\epsilon_1(\delta)\right)\leq \delta.
\end{equation*} Hence, each of these scenarios necessitates distinct selections for $\delta$ which were chosen to minimize $\epsilon_1(\delta)+\delta$. We note that the choices we make for $\delta$ are such that $1/\delta$ is no larger than a polynomial of $n$, so that $\log(1/\delta)$ can be bounded with $\log(2n)$ up to constants.

\paragraph{Under Assumption~\ref{asmp:3.a}} We use Lemma~\ref{lem:8}. The choice of
\begin{eqnarray*}
    \delta = \left(\frac{d\log^3(2n)}{n^{4/5-8/(5q_x)}}\right)^{5q_x/(2q_x+8)},
\end{eqnarray*} yields the result.

\paragraph{Under Assumption~\ref{asmp:3.b}} The result of Lemma~\ref{lem:9} with $\delta=1/n$ yields the result.

\paragraph{Under Assumption~\ref{asmp:3.c}}
In Lemma~\ref{lem:10}, taking the right-hand side of \eqref{eq:lem:6} as $\epsilon = \epsilon(\delta)$ with the choice of
\begin{eqnarray*}
    \delta = \left(\frac{d^{1/2+4/q_x}\log^3(2d)}{n^{2-4/q_x}}\right)^{q_x/(q_x+4)}\leq\left(\frac{d^{3/2}\log^3(2d)}{n}\right)^{q_x/(q_x+4)},
\end{eqnarray*} yields the result. 

Moreover, in this scenario, we claimed that the bias $\Bc$ degenerates with the additional condition that $(3/q_x+1/q)^{-1}\geq2$. To see this, we revisit to inequality in \eqref{eq:A.2.1} and instead write,
\begin{eqnarray*}
    \Pb(\sqrt{n}c^\top(\hat\beta-\beta)\leq x)&=&\Pb(\sqrt{n}c^\top(\Uc+\Bc+\Rc)\leq x)\nonumber\\
    &\leq& \Pb(\sqrt{n}\,c^\top\Uc\leq x+\epsilon+\epsilon') +\Pb(\Abs{\sqrt{n}c^\top\Rc}>\epsilon)++\Pb(\Abs{\sqrt{n}c^\top\Bc}>\epsilon').
\end{eqnarray*} This leads to 
\begin{eqnarray}
    \Abs{\Pb(\sqrt{n}c^\top(\hat\beta-\beta-\Bc)\leq x)-N_c\left(\frac{x}{\sigma_c}\right)} &\leq& \Delta_n + C\epsilon +\Pb\left(|\sqrt{n}c^\top\Rc|>\epsilon\right)\label{eq:X-ind_case_org}\\&& + C\epsilon'+\Pb\left(|\sqrt{n}c^\top\Bc|>\epsilon'\right)\label{eq:X-ind_case_add}.
\end{eqnarray} It is noteworthy that \eqref{eq:X-ind_case_org} can still be controlled in the aforementioned way. The quantity in \eqref{eq:X-ind_case_add} can be controlled using Lemma~\ref{lem:12}. Let $\epsilon'=\epsilon'(\delta)$ be the right-hand side in the probability notation in \eqref{eq:lem:8}, then
\begin{eqnarray*}
    C\epsilon' + \Pb(|\sqrt{n} c^\top\Bc|>\epsilon')\leq C K_x^3K_y\sqrt{\frac{d}{n}}\left(1+\frac{d}{n\sqrt{\delta}}\right)+\delta.
\end{eqnarray*}Finally, choosing $\delta=d/n$ yields the desired result.

\subsection{Proof of Theorem~\ref{thm:2}}
Throughout the proof, we fix a contrast vector $c\in\Real^d$ such that $\norm{c}_{\Sigma^{-1}}=1$. To control $|c^\top(\hat\Bc-\Bc)|$, we begin by writing,
\begin{eqnarray*}
    c^\top\hat\Sigma^{-1}X_i = c^\top\Sigma^{-1}X_i + R_{1,i},&\quad& R_{1,i}=c^\top(\hat \Sigma^{-1}-\Sigma^{-1})X_i,\\
    (Y_i-X_i^\top\hat\beta)=(Y_i-X_i^\top\beta) + R_{2,i},&\quad& R_{2,i}=X_i^\top(\beta-\hat\beta),\\
    \norm{X_i}_{\hat\Sigma^{-1}}^2=\norm{X_i}_{\Sigma^{-1}}^2(1+R_{3,i}),&\quad& R_{3,i}=\frac{\norm{X_i}_{\hat\Sigma^{-1}}^2}{\norm{X_i}_{\Sigma^{-1}}^2}-1,
\end{eqnarray*} for $i=1,\ldots,n$. Then, it follows from the definition of $\hat\Bc$ in \eqref{eq:bias_est} that $|c^\top(\hat\Bc-\Bc)|$ can be bounded by seven distinct quantities, denoted as $\Sc_k$ for $k=1,\ldots,7$ with respect to the presented order.
\begin{eqnarray*}
    \frac{n}{d}\Abs{c^\top(\hat\Bc-\Bc)}\leq&&\Abs{\n\sum_{i=1}^nR_{1,i}(Y_i-X_i^\top\beta)\frac{\norm{X_i}_{\Sigma^{-1}}^2}{d}}+\Abs{\n\sum_{i=1}^nc^\top\Sigma^{-1}X_iR_{2,i}\frac{\norm{X_i}_{\Sigma^{-1}}^2}{d}}\\
    &+&\Abs{\n\sum_{i=1}^nc^\top\Sigma^{-1}X_i(Y_i-X_i^\top\beta)R_{3,i}\frac{\norm{X_i}_{\Sigma^{-1}}^2}{d}}+
    \Abs{\n\sum_{i=1}^nR_{1,i}R_{2,i}\frac{\norm{X_i}_{\Sigma^{-1}}^2}{d}}\\
    &+&\Abs{\n\sum_{i=1}^nc^\top\Sigma^{-1}X_iR_{2,i}R_{3,i}\frac{\norm{X_i}_{\Sigma^{-1}}^2}{d}}+\Abs{\n\sum_{i=1}^nR_{1,i}(Y_i-X_i^\top\beta)R_{3,i}\frac{\norm{X_i}_{\Sigma^{-1}}^2}{d}}\\
    &+&\Abs{\n\sum_{i=1}^nR_{1,i}R_{2,i}R_{3,i}\frac{\norm{X_i}_{\Sigma^{-1}}^2}{d}}\\
    =:&&\sum_{k=1}^7 \Sc_k.
\end{eqnarray*} Before analyzing individual remainder terms, we define
\begin{equation}
    \Dc_\Sigma = \Norm{\Sigma^{-1/2}\hat\Sigma\Sigma^{-1/2}-I_d}_{\rm op}.
\end{equation} It is noteworthy that the quantity $R_{3,i}$ can be stochastically bounded by in terms of $\Dc_\Sigma$. To see this, write the eigenvalue decomposition of $\Sigma^{-1/2}\hat\Sigma\Sigma^{-1/2}$ as $\Sigma^{-1/2}\hat\Sigma\Sigma^{-1/2}=U\Lambda U^\top$ where $U$ is a $d\times d$ orthogonal matrix and $\Lambda$ is a $d\times d$ diagonal matrix with decreasing diagonal entries. Then, the quantity $\Dc_\Sigma$ is determined by two extreme eigenvalues as $$\Dc_\Sigma = \max\set{\Abs{\lambda_{\rm max}(\Sigma^{-1/2}\hat\Sigma\Sigma^{-1/2})-1},\Abs{\lambda_{\rm min}(\Sigma^{-1/2}\hat\Sigma\Sigma^{-1/2})-1}}.$$ Now, we note the deterministic relation that 
\begin{eqnarray*}
    \sup_{\theta\in\Real^d}\Abs{\frac{\theta^\top\hat\Sigma^{-1}\theta}{\theta^\top\Sigma^{-1}\theta}-1} &=& \sup_{\omega\in\Real^d}\Abs{\frac{\omega^\top(\Sigma^{1/2}\hat\Sigma^{-1}\Sigma^{1/2}-I_d)\omega}{\omega^\top\omega}}\\
    &=&\sup_{\omega\in\Real^d}\Abs{\frac{\omega^\top(\Lambda^{-1}-I_d)\omega}{\omega^\top\omega}}\\
    &=&\max\set{\Abs{\frac{1-\lambda_{\rm max}(\Sigma^{-1/2}\hat\Sigma\Sigma^{-1/2})}{\lambda_{\rm max}(\Sigma^{-1/2}\hat\Sigma\Sigma^{-1/2})}},\Abs{\frac{1-\lambda_{\rm min}(\Sigma^{-1/2}\hat\Sigma\Sigma^{-1/2})}{\lambda_{\rm min}(\Sigma^{-1/2}\hat\Sigma\Sigma^{-1/2})}}}\\
    &\leq&\frac{\Dc_\Sigma}{\lambda_{\rm min}(\Sigma^{-1/2}\hat\Sigma\Sigma^{-1/2})}.
\end{eqnarray*} Define the event $\Ac:=\set{\lambda_{\rm min}(\Sigma^{-1/2}\hat\Sigma\Sigma^{-1/2})\geq1/2}$. On $\Ac$, we have
\begin{eqnarray*}
    R_{3,i}\leq\norm{\Sigma^{1/2}\Sigma^{-1}\Sigma^{1/2}-I}_{\rm op}\leq 2\Dc_\Sigma,\quad i=1,\ldots,n.
\end{eqnarray*} Moreover, Proposition~\ref{prop:oliveira} states that
\begin{equation}\label{eq:oliveira}
    \Pb\left(\lambda_{\rm min}(\Sigma^{-1/2}\hat\Sigma\Sigma^{-1/2})\geq 1-9K_x\sqrt{\frac{d+2\log(2/\delta)}{n}}\right)\leq 1-\delta,
\end{equation} for any $\delta\in(0,1)$. Therefore, $\Pb(\Ac)\leq 1-1/n$ provided that $d+2\log(2n)\leq n/(18K_x)^2$.

By inspecting the relationship between the remainders, we can reduce the number of remainders that we need to handle. On $\Ac$, we have
\begin{eqnarray*}
    \Sc_5\leq 2\Dc_\Sigma \Sc_2,\quad
    \Sc_6\leq 2\Dc_\Sigma \Sc_1,\quad
    \Sc_7\leq 2\Dc_\Sigma \Sc_4.
\end{eqnarray*} Combining all, we have that  with probability at least $1-1/n$ that
\begin{equation}\label{eq:thm2.1}
    \frac{n}{d}\Abs{c^\top(\hat\Bc-\Bc)}\leq \left(1+2\Dc_\Sigma\right)\left[\Sc_1+\Sc_2+\Sc_4\right]+\Sc_3.
\end{equation} Hence, it suffices to shift our attention to the remainder $\Sc_k$ for $k=1,2,3,4$, and these four remainders have been carefully analyzed in Lemmas \ref{lem:rem1}---\ref{lem:rem4}, respectively. We provide a brief summary of the key results for each remainder term: define the events with some proper constants $\set{C_i:i=1,2,3,4}$ which only depend on $K_x$ and $K_y$ that
\begin{eqnarray*}
    \Mc_1&=&\set{\Sc_1\leq \frac{C_1\Dc_\Sigma}{\lambda_{\rm min}(\Sigma^{-1/2}\hat\Sigma\Sigma^{-1/2})}}\Rightarrow\Mc_1\cap\Ac\subseteq\set{\Sc_1\leq 2C_1\Dc_\Sigma},\\
    \Mc_2&=&\set{\Sc_2\leq C_2\norm{\hat\beta-\beta}_\Sigma},\\
    \Mc_3&=&\set{\Sc_3\leq \frac{C_3\Dc_\Sigma}{\lambda_{\rm min}(\Sigma^{-1/2}\hat\Sigma\Sigma^{-1/2})}}\Rightarrow\Mc_3\cap\Ac\subseteq\set{\Sc_3\leq 2C_3\Dc_\Sigma}\\
    \Mc_4(\delta)&=&\set{\Sc_4\leq C_4 \frac{\Dc_\Sigma\norm{\hat\beta-\beta}_\Sigma}{\lambda_{\rm min}(\Sigma^{-1/2}\hat\Sigma\Sigma^{-1/2})}\left(1+\frac{d\log (2d)}{\delta\sqrt{n}}\right)},\quad\delta\in(0,1]\\
    &\Rightarrow&\Mc_4(\delta)\cap\Ac\subseteq\set{\Sc_4\leq 2C_4\Dc_\Sigma\norm{\hat\beta-\beta}_\Sigma\left(1+\frac{d\log (2d)}{\delta\sqrt{n}}\right)}
\end{eqnarray*} Lemma~\ref{lem:rem1}---\ref{lem:rem4} prove that under Assumption~\ref{asmp:3.a} the events $\Mc_i$, $i=1,\ldots,4$ occur with high probabilities as,
\begin{equation*}
    \Pb(\Mc_1)\geq1-\frac{d}{n},\quad\Pb(\Mc_2)\geq1-\frac{d}{n},\quad\Pb(\Mc_3)\geq1-\frac{1}{n},\quad\Pb(\Mc_4)\geq1-\delta.
\end{equation*} It follows from \eqref{eq:thm2.1} and the definitions of the events $\Mc_i$, $i=1,2,3,4$ that there exists a constant $C=C()$ such that on the event $\Ac\cap\Mc_1\cap\Mc_2\cap\Mc_3\cap\Mc_4(\delta)$, it holds that
\begin{equation}\label{eq:thm2:2}
    \frac{n}{d}\abs{c^\top(\hat\Bc-\Bc)}\leq C\left(1\vee\Dc_\Sigma\right)\left(\Dc_\Sigma+\norm{\hat\beta-\beta}_\Sigma+\Dc_\Sigma\norm{\hat\beta-\beta}_\Sigma\left(1+\frac{d\log (2d)}{\delta\sqrt{n}}\right)\right)
\end{equation}
Consequently, to manage $\lvert c^\top (\hat{\Bc} - \Bc) \rvert$, it is sufficient to derive upper bounds for $\mathcal{D}_\Sigma$ and $\lVert \hat{\beta} - \beta \rVert_\Sigma$ with high probabilities, which are described in Propositions \ref{prop:6} and \ref{prop:11}, respectively. We will separately control these quantities based on the assumptions pertaining to the covariates.

\paragraph{Under Assumption~\ref{asmp:3.a}.} By Proposition~\ref{prop:6} and Proposition~\ref{prop:11}, we get 
\begin{eqnarray*}
    \Pb(\Dc_\Sigma\leq C\Av_{n,d}(\eta))\geq 1-\eta,\quad\Pb(\norm{\hat\beta-\beta}_\Sigma\leq  C'\Bv_{n,d})\geq 1-\frac{1}{n^{\min\{1,s/2-1\}}},
\end{eqnarray*}where $C$ and $C'$ are some constants which depend on $q_x,q,K_x,K_y,$ and $\overline{\lambda}$, and
\begin{equation*}
    \Av_{n,d}(\eta) = \frac{d\log^{3/2}(2d/\eta)}{\eta^{2/q_x}n^{1-2/q_x}}+\sqrt{\frac{d\log(2d/\eta)}{n}},\quad\Bv_{n,d}=\sqrt{\frac{d+\log(2n)}{n}}.
\end{equation*} Denote the events $\Cc_1(\eta):=\{\Dc_\Sigma\leq C\Av_{n,d}(\eta)\}$ and $\Cc_2:=\{\norm{\hat\beta-\beta}_\Sigma\leq C'\Bv_{n,d}\}$, and we let
\begin{equation*}
    \delta^*:=\delta^*(\eta) = \left[\left(\Av_{n,d}^{1/2}(\eta)\vee\Av_{n,d}(\eta)\right)\Bv_{n,d}^{1/2}\frac{\log^{1/2}(2d)d}{\sqrt{n}}\right]\wedge1.
\end{equation*} Equipped with such choice of $\delta^*$, we have from \eqref{eq:thm2:2} that
\begin{eqnarray*}
    &&\sqrt{n}\abs{c^\top (\hat{\Bc} - \Bc)}\\
    &\leq& C\left[\frac{d}{\sqrt{n}}(1\vee\Av_{n,d}(\eta))\left(\Av_{n,d}(\eta)+\Bv_{n,d}+\Av_{n,d}(\eta)\Bv_{n,d}\right)+\frac{\log^{1/2}(2d)d}{\sqrt{n}}(\Av_{n,d}^{1/2}(\eta)\vee\Av_{n,d}(\eta))\Bv_{n,d}^{1/2}\right],\\
    &\leq& C\left[\frac{d}{\sqrt{n}}(\Av_{n,d}(\eta)\vee\Av_{n,d}^2(\eta))+\frac{\log^{1/2}(2d)d}{\sqrt{n}}(\Av_{n,d}^{1/2}(\eta)\vee\Av_{n,d}(\eta))\Bv_{n,d}^{1/2}\right]\\
    &\leq&C\frac{d}{\sqrt{n}}\left[(\Av_{n,d}(\eta)\vee\Av_{n,d}^2(\eta))+\Bv_{n,d}\log(2d)\right],
\end{eqnarray*} holds with probability at least
\begin{equation*}
    \Pb\left(\left(\Ac\cap\Mc_1\cap\Mc_2\cap\Mc_3\cap\Mc_4(\delta^*)\cap\Cc_1(\eta)\cap\Cc_2\right)^\complement\right)\geq \frac{2d}{n}+\frac{1}{n}+\frac{1}{n^{\min\{1,s/2-1\}}}+\delta^*+\eta.
\end{equation*}

\paragraph{Under Assumption~\ref{asmp:3.b}} If the covariates are drawn from a sub-Gaussian distribution, a more robust result regarding the concentration of the sample covariance matrix can be obtained, as demonstrated in Proposition~\ref{prop:6}. Specifically, we can establish that:

\begin{equation*}
\Pb(\Dc_\Sigma\leq C\Bv_{n,d})\geq 1-\frac{1}{n}.
\end{equation*} The remaining step of the proof is identical to the one presented under Assumption~\ref{asmp:3.a}, with the only difference being the utilization of $\Bv_{n,d}$ instead of $\Av_{n,d}$ as a stochastic bound for $\Dc_\Sigma$.

\subsection{Proof of Theorem~\ref{thm:3}}
For a given $c\in\Real^d$ with $\norm{c}_{\Sigma^{-1}}=1$, write $\hat S_n=\sqrt{n}c^\top(\hat\beta-\hat\Bc-\beta)$ and $\tilde S_n=\sqrt{n}c^\top(\hat\beta-\Bc-\beta)$. For any $x\in \Real$ and $\epsilon>0$, we have
\begin{eqnarray}\label{eq:thm3:1}
    \Pb(\hat S_n\leq x)&\leq&\Pb(\tilde S_n\leq x+\epsilon)+\Pb(|\tilde S_n-\hat S_n|\geq \epsilon),\nonumber\\
    \Pb(\hat S_n> x)&\leq&\Pb(\tilde S_n> x+\epsilon)+\Pb(|\tilde S_n-\hat S_n|\geq \epsilon).
\end{eqnarray} Define $\tilde\Delta_{n,d}:=\sup_{x\in\Real}|\Pb(\tilde S_n\leq x)-N_c(x/\sigma_c)|$ and $\hat\Delta_{n,d}:=\sup_{x\in\Real}|\Pb(\hat S_n\leq x)-N_c(x/\sigma_c)|$. Then, it follows from \eqref{eq:thm3:1} that
\begin{eqnarray}\label{eq:thm3:2}
    \hat\Delta_n&\leq&\tilde\Delta_n+\sigma_c^{-1}\norm{N_c'}_\infty\epsilon + \Pb\left(\sqrt{n}|\hat\Bc-\Bc|\geq\epsilon\right)\nonumber\\
    &\leq&\tilde\Delta_n+C\epsilon + \Pb\left(\sqrt{n}|\hat\Bc-\Bc|\geq\epsilon\right),
\end{eqnarray} where $C=C(K_x,K_y,\underline{\lambda})$ is defined in \eqref{eq:A.2.4}. Note that the quantity $\tilde\Delta_n$ is bounded in Theorem~\ref{thm:1}. Moreover, the tail probability of $\hat\Bc$ is controlled in Theorem~\ref{thm:2}.
\paragraph{Under Assumption~\ref{asmp:3.a}}
We have from Theorem~\ref{thm:2}(i) that
\begin{equation*}
    \Pb\left(\sqrt{n}|\hat\Bc-\Bc|\geq\delta_{n,d}(\eta)\right)\leq\frac{1}{\sqrt{n}} + \delta_{n,d}(\eta).
\end{equation*} Taking $\epsilon=\delta_{n,d}(\eta^*)$ in \eqref{eq:thm3:2} where
\begin{equation*}
    \eta^*=\left(\frac{d\log^{3/4}(2n)}{n^{3/4-1/q_x}}\right)^{(2q_x)/(2+q_x)}\vee\left(\frac{d\log(2n)}{n^{5/6-4/(3q_x)}}\right)^{(3q_x)/(4+q_x)},
\end{equation*} yields the result.
\paragraph{Under Assumption~\ref{asmp:3.b}} Taking $\epsilon=\delta_{n,d}$, which is defined in Theorem~\ref{thm:2}(ii), gives the desired result.
\subsection{Proof of Corollary~\ref{cor:1}}
We will show that Berry Esseen bound in Theorem~\ref{thm:3}(i) is dominated by $1/\sqrt{n}+d^{3/2}/n$ up to polylogarithmic factors. Since $q_x\geq 12$ and $d\lesssim n^{2/3}$, we observe that
\begin{equation*}
    \frac{d\log^3(2n)}{n^{4/5-8/(5q_x)}}, \frac{d\log^{3/4}(2n)}{n^{3/4-1/q_x}},\frac{d\log(2n)}{n^{5/6-4/(3q_x)}}\lesssim\frac{d}{n^{2/3}}\lesssim1.
\end{equation*}Furthermore, since the exponents of the above terms appeared in Berry Esseen bound is larger than $3/2$ as long as $q_x\geq12$, the proof is completed.

\subsection{Proof of Theorem~\ref{thm:5}} Throughout, we fix $c\in\Real^d$ with $\norm{c}_{\Sigma^{-1}}=1$. Theorem~\ref{thm:3} proves that
\begin{equation*}
    \sup_{t\in\Real}\Abs{\Pb\left(\sqrt{n}\frac{c^\top(\hat\beta_{\rm bc}-\beta)}{\sigma_c}\leq t\right) - N_c(t)}\leq \epsilon_{n,d}.
\end{equation*} We observe that for a small $\delta>0$,
\begin{equation*}
    \Pb\left(\sqrt{n}\frac{c^\top(\hat\beta_{\rm bc}-\beta)}{\hat\sigma_c}\leq t\right)\leq \Pb\left(\sqrt{n}\frac{c^\top(\hat\beta_{\rm bc}-\beta)}{\sigma_c}\leq t(1+\delta)\right) + \Pb\left(\frac{\hat\sigma_c}{\sigma_c}>1+\delta\right).
\end{equation*} Consequently,
\begin{eqnarray*}
    \Pb\left(\sqrt{n}\frac{c^\top(\hat\beta_{\rm bc}-\beta)}{\hat\sigma_c}\leq t\right)-N_c(t)&\leq& \Pb\left(\sqrt{n}\frac{c^\top(\hat\beta_{\rm bc}-\beta)}{\sigma_c}\leq t(1+\delta)\right)-N_c(t)+ \Pb\left(\frac{\hat\sigma_c}{\sigma_c}>1+\delta\right)\\
    &\leq& \epsilon_{n,d} + N_c\left((1+\delta)t\right)-N_c(t)+\Pb\left(\frac{\hat\sigma_c}{\sigma_c}>1+\delta\right).
\end{eqnarray*} Similarly, we can obtain
\begin{eqnarray*}
    \Pb\left(\sqrt{n}\frac{c^\top(\hat\beta_{\rm bc}-\beta)}{\hat\sigma_c}> t\right)-(1-N_c(t))&\leq& \Pb\left(\sqrt{n}\frac{c^\top(\hat\beta_{\rm bc}-\beta)}{\sigma_c}> t(1-\delta)\right)-(1-N_c(t))+ \Pb\left(\frac{\hat\sigma_c}{\sigma_c}\leq 1-\delta\right)\\
    &\leq& \epsilon_{n,d} +N_c(t)- N_c\left((1-\delta)t\right)+\Pb\left(\frac{\hat\sigma_c}{\sigma_c}\leq 1-\delta\right).
\end{eqnarray*} Combining these leads to
\begin{equation}\label{eq:thm:6.key}
    \Abs{\Pb\left(\sqrt{n}\frac{c^\top(\hat\beta_{\rm bc}-\beta)}{\hat\sigma_c}\leq t\right)-N_c(t)}\leq \epsilon_{n,d}+N_c((1+\delta)t)-N_c((1-\delta)t)+\Pb\left(\Abs{\frac{\hat\sigma_c}{\sigma_c}-1}>\delta\right).
\end{equation} To bound the middle term on the right-hand side in \eqref{eq:thm:6.key}, we note that for some constant $c>0$,
\begin{eqnarray*}
    \Abs{N_c((1+\delta)t)-N_c((1-\delta)t)} &=&\Abs{\int_{1-\delta}^{1+\delta} N_c'(ut)t\, du}\\
    &\leq& \int_{1-\delta}^{1+\delta} \Abs{N_c'(ut)t\, du}\leq \frac{1}{1-\delta}\int_{1-\delta}^{1+\delta} \Abs{N_c'(ut)ut\, du}\leq c\frac{2\delta}{1-\delta}.
\end{eqnarray*} For the last inequality, we use the fact that $xN_c'(x)=x[\Phi^{(1)}(x)-\kappa_c\Phi^{(4)}(x)]$ is uniformly bounded (see Lemma~\ref{lem:11}). For the rightmost term in \eqref{eq:thm:6.key}, define the event
\begin{equation*}
    \Ec_{c,n,d} = \set{\Abs{\frac{\hat\sigma_c^2}{\sigma_c^2}-1}\leq\eta_{n,d}},
\end{equation*} where $\eta_{n,d}$ is defined in Theorem~\ref{thm:4} and satisfying that $\Pb(\Ec_{c,n,d})\geq 1-\eta_{n,d}$. On $\Ec_{c,n,d}$, we have
\begin{equation*}
    1-(\eta_{n,d}\wedge 1)\leq \sqrt{1-(\eta_{n,d}\wedge 1)}\leq \frac{\hat\sigma_c}{\sigma_c}\leq \sqrt{1+\eta_{n,d}}\leq1+\frac{1}{2}\eta_{n,d}.
\end{equation*} Hence, taking $\delta = \eta_{n,d}\wedge 1/2$ in \eqref{eq:thm:6.key} yields
\begin{equation*}
    \Abs{\Pb\left(\sqrt{n}\frac{c^\top(\hat\beta_{\rm bc}-\beta)}{\hat\sigma_c}\leq t\right)-N_c(t)}\leq \epsilon_{n,d}+C\eta_{n,d}.
\end{equation*}

\subsection{Proof of Theorem~\ref{thm:6}}
For $c\in\Real^d\setminus\{0_d\}$, Theorem~\ref{thm:3} proves that
\begin{eqnarray*}
    \sup_{t\in\Real}\Abs{\Pb\left[\sqrt{n}c^\top(\hat\beta_{\rm bc}-\beta)\leq t)\right]-N_c\left(\frac{t}{\sigma_c}\right)}\leq \epsilon_{n,d},
\end{eqnarray*} for some rate $\epsilon_{n,d}>0$. Hence, the median bias of $c^\top(\hat\beta_{\rm bc}-\beta)$ is bounded as
\begin{eqnarray}\label{eq:thm4:1}
    {\rm Medbias}\left(c^\top(\hat\beta_{\rm bc}-\beta)\right)&:=&\frac{1}{2}-\min\left\{\Pb(c^\top(\hat\beta_{\rm bc}^{(b)}-\beta)>0),\Pb(c^\top(\hat\beta_{\rm bc}^{(b)}-\beta)<0)\right\}\nonumber\\
    &\leq&\epsilon_{n,d}+\Abs{N_c(0)-\frac{1}{2}}=\epsilon_{n,d}+\kappa_c\nonumber\\
    &\leq&\epsilon_{n,d} + C\sqrt{\frac{d}{n}}.    
\end{eqnarray} The last inequality follows from Lemma~\ref{lem:12} and the constant $C$ is described therein. Note that $\hat\beta_{\rm bc}^{(1)},\ldots,\hat\beta_{\rm bc}^{(B)}$ are independent bias-corrected estimators based on $\lfloor n/B\rfloor$. For any $1\leq b\leq B$, we note from \eqref{eq:thm4:1} that
\begin{eqnarray*}
   {\rm Medbias}\left(c^\top(\hat\beta_{\rm bc}^{(b)}-\beta)\right)\leq \epsilon_{\frac{n}{B},d}+C\sqrt{\frac{Bd}{n}}.
\end{eqnarray*} An application of Theorem~\ref{thm:2} of \cite{kuchibhotla2021hulc} results in
\begin{eqnarray*}
    \Pb\left(\beta\notin{\rm CI}_\alpha^{\rm HULC}\right)-\alpha&\leq&\frac{\alpha B(B-1)}{2}\left(\epsilon_{\frac{n}{B},d}+C\sqrt{\frac{Bd}{n}}\right)^2\\
    &\leq&\alpha B^2\epsilon_{\frac{n}{B},d}^2 + C^2\frac{\alpha B^3d}{n}.
\end{eqnarray*} This completes the proof.

\subsection{Consistency of the Sandwich Variance Estimator}\label{sec:A.Consistency of the Sandwich Variance Estimator}
\subsubsection{Proof of Theorem~\ref{thm:4}}
In this section, we collect various bounds that are used in the proof of Theorem~\ref{thm:4} about the consistency of the sandwich variance estimator with respect to a given query vector $c\in\Real^d\setminus\{0_d\}$. Note that $\abs{\hat\sigma_c^2/\sigma_c^2-1}$ is invariant to the scaling of $c$. Hence, we let $c$ be such that $\norm{c}_\Sigma=1$, and write the $L_2$-normalized version as $\tilde c$, so that $c = \Sigma^{1/2}\tilde c$ and $\norm{\tilde c}_2=1$. Then, the sandwich variance estimate $\hat \sigma_c^2$ can be expressed as
\begin{equation*}
    \hat \sigma_c^2 = \frac{1}{n}\sum_{i=1}^n (\tilde c^\top\Sigma^{1/2}\hat\Sigma^{-1}X_i)^2(Y_i-X_i^\top\hat\beta)^2.
\end{equation*}
We commence by introducing an an intermediary quantity $\tilde\sigma_c^2$, defined as 
\begin{equation}
    \tilde \sigma_c^2 = \frac{1}{n}\sum_{i=1}^n (\tilde c^\top\Sigma^{-1/2}X_i)^2(Y_i-X_i^\top\beta)^2.
\end{equation} It is immediate from its definition that $\Eb[\tilde\sigma_c^2]=\sigma_c^2$. Moreover, we have from Assumption~\ref{asmp:4} that
\begin{equation*}
    \Abs{\frac{\hat\sigma_c^2}{\sigma_c^2}-1}\leq\underline{\lambda}^{-2}\Abs{\hat\sigma_c^2-\sigma_c^2}\leq \underline{\lambda}^{-2}\left(\Abs{\hat\sigma_c^2-\tilde\sigma_c^2}+\Abs{\tilde\sigma_c^2-\sigma_c^2}\right).
\end{equation*} Since $(1/q_x+1/q)^{-1}\geq4$, the quantity $\Abs{\tilde\sigma_c^2-\sigma_c^2}$ can be controlled via Chebyshev's inequality as:
\begin{eqnarray*}
    \Pb\left(\Abs{\tilde\sigma_c^2-\sigma_c^2}\geq t\right)&\leq& t^{-2}{\rm Var}\left(\frac{1}{n}\sum_{i=1}^n (\tilde c^\top\Sigma^{-1/2}X_i)^2(Y_i-X_i^\top\beta)^2\right)=(nt^2)^{-1}{\rm Var}\left((\tilde c^\top\Sigma^{-1/2}X)^2(Y-X^\top\beta)^2\right)\\
    &\overset{(i)}{\leq}&(nt^2)^{-1}\Eb\left[(\tilde c^\top\Sigma^{-1/2}X)^4(Y-X^\top\beta)^4\right]\\
    &\overset{(ii)}{\leq}&(nt^2)^{-1}\left(\Eb\left[(\tilde c^\top\Sigma^{-1/2}X)^{q_x}\right]\right)^{4/q_x}\left(\Eb\left[(Y-X^\top\beta)^q\right]\right)^{4/q}\overset{(iii)}{\leq}\frac{K_x^4K_y^4}{nt^2}.
\end{eqnarray*} Here, $(i)$ is Jensen's inequality and the inequality, and $(ii)$ follows from $(1/q_x+1/q)^{-1}\geq4$ and H\"older's inequality. The last inequality $(iii)$ is due to Assumption~\ref{asmp:2} and \ref{asmp:3.a}. The choice of $t=K_x^2K_y^2n^{-1/3}$ yields
\begin{equation*}
    \Pb\left(\Abs{\tilde\sigma_c^2-\sigma_c^2}\geq K_x^2K_y^2n^{-1/3}\right)\leq n^{-1/3}.
\end{equation*}

Our next lemma presents a deterministic bound for $\abs{\hat\sigma_c^2-\tilde\sigma_c^2}$, implying that the rate of convergence depends on several quantities, which will be described therein.

\begin{lemma}\label{lem:det_bound_sand_var}[Deterministic Bound for the Sandwich Estimator] For a given $\tilde c\in\Sb^{d-1}$, define
\begin{eqnarray*}
    \Mc&=&\sup_{\theta\in\Sb^{d-1}}\frac{1}{n}\sum_{i=1}^n(\theta^\top\Sigma^{-1/2}X_i)^2(\tilde c^\top\Sigma^{-1/2}X_i)^2,\quad \Nc = \sup_{\theta\in\Sb^{d-1}}\frac{1}{n}\sum_{i=1}^n(\theta^\top\Sigma^{-1/2}X_i)^2(Y_i-X_i^\top\beta)^2,\\
    \Lc&=&\max\set{\Norm{\frac{1}{n}\sum_{i=1}^n\Sigma^{-1/2}X_i(\tilde c^\top\Sigma^{-1/2}X_i)^2(Y_i-X_i^\top\beta)}_2,\Norm{\frac{1}{n}\sum_{i=1}^n\Sigma^{-1/2}X_i(\tilde c^\top\Sigma^{-1/2}X_i)(Y_i-X_i^\top\beta)^2}_2}.
\end{eqnarray*} If $\{\lambda_{\rm min}(\Sigma^{-1/2}\hat\Sigma\Sigma^{-1/2})\geq 1/2\}$ holds true, then
\begin{eqnarray*}
    \Abs{\hat\sigma_c^2-\tilde\sigma_c^2}&\leq&2\Lc\norm{\hat\beta-\beta}_\Sigma+3\Mc\norm{\hat\beta-\beta}_\Sigma^2+\left(8\Lc+12\max_{1\leq i\leq n} \abs{X_i^\top(\hat\beta-\beta)}^2\right)\Dc_\Sigma+8\Nc\Dc_\Sigma^2.
\end{eqnarray*}
\end{lemma}
\begin{proof}[proof of Lemma~\ref{lem:det_bound_sand_var}]
    We begin by writing
    \begin{eqnarray*}
        \hat\Sigma^{-1}\Sigma^{1/2}\tilde c = \Sigma^{-1/2}(\tilde c + r_1),&\quad& r_1 = (\Sigma^{1/2}\hat\Sigma^{-1}\Sigma^{1/2}-I_d)\tilde c,\\
        Y_i-X_i^\top\hat\beta = Y_i-X_i^\top\beta + r_{2,i},&\quad& r_{2,i} = X_i^\top(\hat\beta-\beta)\quad\mbox{for}\quad i=1,\ldots,n.
    \end{eqnarray*} From their definitions, the difference $\abs{\hat\sigma_c^2-\tilde\sigma_c^2}$ can be bounded by the sum of eight quantities, denoted as $\Qc_k$ for $k=1,\ldots,8$ in the presented order.
    \begin{eqnarray*}
        &&\abs{\hat\sigma_c^2-\tilde\sigma_c^2}\leq\Abs{\frac{2}{n}\sum_{i=1}^n\abs{\tilde c^\top\Sigma^{-1/2}X_i}^2(Y_i-X_i^\top\beta)r_{2,i}}+\Abs{\frac{1}{n}\sum_{i=1}^n\abs{\tilde c^\top\Sigma^{-1/2}X_i}^2r_{2,i}^2}\\
        &&+\Abs{\frac{2}{n}\sum_{i=1}^n(\tilde c^\top\Sigma^{-1/2}X_i)(r_1^\top\Sigma^{-1/2}X_i)(Y_i-X_i^\top\beta)^2}+\Abs{\frac{4}{n}\sum_{i=1}^n(\tilde c^\top\Sigma^{-1/2}X_i)(r_1^\top\Sigma^{-1/2}X_i)(Y_i-X_i^\top\beta)r_{2,i}}\\
        &&+ \Abs{\frac{2}{n}\sum_{i=1}^n(\tilde c^\top\Sigma^{-1/2}X_i)(r_1^\top\Sigma^{-1/2}X_i)r_{2,i}^2} + \Abs{\frac{1}{n}\sum_{i=1}^n(r_1^\top\Sigma^{-1/2}X_i)^2(Y_i-X_i^\top\beta)^2}\\
        &&+\Abs{\frac{1}{n}\sum_{i=1}^n(r_1^\top\Sigma^{-1/2}X_i)^2(Y_i-X_i^\top\beta)r_{2,i}}+\Abs{\frac{1}{n}\sum_{i=1}^n(r_1^\top\Sigma^{-1/2}X_i)^2r_{2,i}^2}.
    \end{eqnarray*} We observe there exists following deterministic inequalities between the quantities. They are merely inequalities of arithmetic and geometric means.
    \begin{equation*}
        \Qc_4\leq \Qc_3 + \Qc_5,\quad
        \Qc_5\leq \Qc_2 + \Qc_8,\quad
        \Qc_7\leq \Qc_6 + \Qc_8.
    \end{equation*} Combining these lead to \begin{equation}\label{eq:lem:6.1}
        \abs{\hat\sigma_c^2-\tilde\sigma_c^2} \leq \Qc_1 + 3\Qc_2 + 2\Qc_3 + 2\Qc_6 + 4\Qc_8
    \end{equation} Before we analyze each individual quantity, we observe that
    \begin{eqnarray*}
        \norm{r_1}_2 = \norm{(\Sigma^{1/2}\hat\Sigma^{-1}\Sigma^{1/2}-I_d)\tilde c}_2 \leq \norm{\Sigma^{1/2}\hat\Sigma^{-1}\Sigma^{1/2}-I_d}_{\rm op}\leq \frac{\Dc_\Sigma}{\lambda_{\rm min}(\Sigma^{-1/2}\hat\Sigma\Sigma^{-1/2})}\leq 2\Dc_\Sigma.
    \end{eqnarray*} The last inequality holds on the event $\{\lambda_{\rm min}(\Sigma^{-1/2}\hat\Sigma\Sigma^{-1/2})\geq 1/2\}$. Furthermore, the following is useful in analyzing $\Qc_8$:
    \begin{eqnarray*}
        \frac{1}{n}\sum_{i=1}^n(r_1^\top\Sigma^{-1/2}X_i)^2 &=& r_1^\top\Sigma^{-1/2}\left(\frac{1}{n}\sum_{i=1}^nX_iX_i^\top\right) \Sigma^{-1/2}r_1 = r_1^\top\left(\Sigma^{-1/2}\hat\Sigma\Sigma^{-1/2}\right)r_1\\
        &=&\tilde c\left(\Sigma^{1/2}\hat\Sigma^{-1}\Sigma^{1/2}-I_d\right)\left(\Sigma^{-1/2}\hat\Sigma\Sigma^{-1/2}\right)\left(\Sigma^{1/2}\hat\Sigma^{-1}\Sigma^{1/2}-I_d\right)\tilde c\\
        &=&\tilde c\left(\Sigma^{1/2}\hat\Sigma^{-1}\Sigma^{1/2}-I_d+\Sigma^{-1/2}\hat\Sigma\Sigma^{-1/2}-I_d\right)\tilde c\\
        &\leq& \Dc_\Sigma+\norm{\Sigma^{1/2}\hat\Sigma^{-1}\Sigma^{1/2}-I_d}_{\rm op}\leq 3\Dc_\Sigma.
    \end{eqnarray*}
    We now present deterministic inequalities for each quantity, requiring no further explanation.
    \begin{eqnarray*}
        \Qc_1&=&\Abs{\frac{2}{n}\sum_{i=1}^n\abs{\tilde c^\top\Sigma^{-1/2}X_i}^2(Y_i-X_i^\top\beta)X_i^\top(\hat\beta-\beta)}\leq 2\Lc \norm{\hat\beta-\beta}_\Sigma,\\
        \Qc_2&=&\Abs{\frac{1}{n}\sum_{i=1}^n\abs{\tilde c^\top\Sigma^{-1/2}X_i}^2(\hat\beta-\beta)^\top X_iX_i^\top(\hat\beta-\beta)}\leq \Mc\norm{\hat\beta-\beta}_\Sigma^2,\\
        \Qc_3&=&\Abs{\frac{2}{n}\sum_{i=1}^n(\tilde c^\top\Sigma^{-1/2}X_i)(r_1^\top\Sigma^{-1/2}X_i)(Y_i-X_i^\top\beta)^2}\leq 2\Lc\norm{r_1}_2\leq 4\Lc\Dc_\Sigma.\\
        \Qc_6&=&\Abs{\frac{1}{n}\sum_{i=1}^n(r_1^\top\Sigma^{-1/2}X_i)^2(Y_i-X_i^\top\beta)^2}\leq \Nc\norm{r_1}_2^2\leq 4\Nc\Dc_\Sigma^2.\\
        \Qc_8&=& \Abs{\frac{1}{n}\sum_{i=1}^n(r_1^\top\Sigma^{-1/2}X_i)^2r_{2,i}^2}\leq \max_{1\leq i\leq n} r_{2,i}^2 \frac{1}{n}\sum_{i=1}^n(r_1^\top\Sigma^{-1/2}X_i)^2\leq 3\Dc_\Sigma\max_{1\leq i\leq n} r_{2,i}^2.
    \end{eqnarray*} Combining all in \eqref{eq:lem:6.1} gives the result.
\end{proof}

Lemma~\ref{lem:det_bound_sand_var} underscores the intricacies involved in the convergence of the sandwich variance, which hinge upon several quantities: $\Dc_\Sigma$, $\norm{\hat\beta-\beta}_\Sigma$, $\Mc$, $\Nc$, $\Lc$, and $\max_{1\leq i\leq n}\abs{X_i^\top(\hat\beta-\beta)}$. Each of these elements is individually addressed and examined in Proposition~\ref{prop:6}, Proposition~\ref{prop:11}, Lemma~\ref{lem:onM}, Lemma~\ref{lem:onN}, Lemma~\ref{lem:onL}, and Theorem~\ref{thm:A.1}.

\begin{lemma}\label{lem:onM} For $\tilde c\in\Sb^{d-1}$, define
\begin{eqnarray*}
    \overline{\Mc}&=&\sup_{\theta\in\Sb^{d-1}}\Abs{\frac{1}{n}\sum_{i=1}^n(\theta^\top\Sigma^{-1/2}X_i)^2(\tilde c^\top\Sigma^{-1/2}X_i)^2-\Eb(\theta^\top\Sigma^{-1/2}X)^2(\tilde c^\top\Sigma^{-1/2}X)^2}
\end{eqnarray*} If Assumption~\ref{asmp:3.b} holds with $q_x\geq 4$, then there exists a universal constant $C$ such that
\begin{equation*}
    \Eb\left[\overline{\Mc}\right]\leq CK_x^4\left[\sqrt{\frac{\log(2n)d}{n^{1-4/q_x}}}+\frac{\log(2n)d}{n^{1-4/q_x}}\right].
\end{equation*}Consequently, we have
\begin{equation*}
    \Eb\left[\Mc\right]\leq CK_x^4\left[1+\frac{\log(2n)d}{n^{1-4/q_x}}\right],
\end{equation*} for some (possibly) different universal constant $C$.
\end{lemma}
\begin{proof}[proof of Lemma~\ref{lem:onM}]
    Let $Z_i = (\tilde c^\top\Sigma^{-1/2}X_i)\Sigma^{-1/2}X_i$ for $i=1,\ldots,n$. The unit ball $\Sb^{d-1}$ is a symmetric and convex body of radius 1, and it has a modulus of convexity\footnote{The definition of the modulus of convexity of geometric objects can be found in Section 2 of \cite{guedon2007lp}.} of power type 2. Hence, Theorem 3 of \cite{guedon2007lp} applies and yields,
    \begin{eqnarray}\label{eq:lem:onM:key}
        \Eb\sup_{\theta\in\Sb^{d-1}}\Abs{\frac{1}{n}\sum_{i=1}^n \abs{\theta^\top Z_i}^2 - \Eb \abs{\theta^\top Z_i}^2 } \leq && C\frac{\log(2n)}{n}\left[\Eb\max_{1\leq i\leq n}\norm{Z_i}_2^2\right]\nonumber\\
        &+& C\sqrt{\frac{\log(2n)}{n}}\left[\Eb\max_{1\leq j\leq n}\norm{Z_i}_2^2\right]^{1/2}\sup_{\theta\in\Sb^{d-1}}\left(\Eb\abs{\theta^\top Z_1}^2\right)^{1/2},
    \end{eqnarray}where $C$ is a universal constant. We note that
    \begin{eqnarray}\label{eq:lem7.1}
        \sup_{\theta\in\Sb^{d-1}} \Eb[\abs{\theta^\top Z_i}^2]\leq \sup_{\theta\in\Sb^{d-1}}\Eb\left[\abs{\theta^\top\Sigma^{-1/2}X_i}^4\right]\leq K_x^4.
    \end{eqnarray} Furthermore, Jensen's inequality yields that
    \begin{eqnarray}\label{eq:lem7.2}
        \Eb\max_{1\leq j\leq n}\norm{Z_i}_2^2&\leq& \left(\Eb\max_{1\leq j\leq n}\norm{Z_i}_2^{q_x/2}\right)^{2/q_x}\leq\left(\sum_{j=1}^n\Eb\norm{Z_i}_2^{q_x/2}\right)^{2/q_x}\nonumber\\
        &\leq& n^{4/q_x} \left(\Eb\norm{Z_i}_2^{q_x/2}\right)^{4/q_x}\leq n^{4/q_x}K_x^4d.
    \end{eqnarray} Combining inequalities \eqref{eq:lem7.1} and \eqref{eq:lem7.2} in \eqref{eq:lem:onM:key} gives the desired result.
\end{proof}
\begin{lemma}\label{lem:onN} For $\tilde c\in\Sb^{d-1}$, define
\begin{eqnarray*}
    \overline{\Nc} &=& \sup_{\theta\in\Sb^{d-1}}\Abs{\frac{1}{n}\sum_{i=1}^n(\theta^\top\Sigma^{-1/2}X_i)^2(Y_i-X_i^\top\beta)^2-\Eb(\theta^\top\Sigma^{-1/2}X)^2(Y_i-X_i^\top\beta)^2},
\end{eqnarray*} If $s=(1/q_x+1/q)^{-1}\geq4$, then there exists a universal constant $C$ such that
\begin{eqnarray*}
    \Pb\left(\overline{\Nc}\leq C\frac{\overline{\lambda}K_x^2K_y^2}{\underline{\lambda}^2}\left[\sqrt{\frac{d\log(2d/\delta)}{n}}+\frac{d\log^{3/2}(2d/\delta)}{\delta^{2/s}n^{1-2/s}}\right]\right)\geq 1-\delta,
\end{eqnarray*} for all $\delta\in(0,1)$.
\end{lemma}
\begin{proof}[proof of Lemma~\ref{lem:onN}] Let $W_i = \Sigma^{-1/2}X_i(Y_i-X_i^\top\beta)$, so that ${\rm Var}(W_i)=V$. We use Proposition~\ref{prop:6} which proved the concentration of $\norm{V^{-1/2}(\sum_{i=1}^nW_iW_i^\top/n) V^{-1/2}-I_d}_{\rm op}$. This implies the concentration of $\overline{\Nc}$ as Assumption~\ref{asmp:4} indicates that
\begin{equation*}
    \overline{\Nc}\leq\norm{V}_{\rm op}\Norm{V^{-1/2}\left(\frac{1}{n}\sum_{i=1}^nW_iW_i^\top\right) V^{-1/2}-I_d}_{\rm op}\leq \overline{\lambda}\Norm{V^{-1/2}\left(\frac{1}{n}\sum_{i=1}^nW_iW_i^\top\right) V^{-1/2}-I_d}_{\rm op}.
\end{equation*} To ensure the assumptions made in Proposition~\ref{prop:6}, we observe from Assumption~\ref{asmp:4} and H\"older's inequality that 
\begin{eqnarray*}
    \sup_{\theta\in\Sb^{d-1}}\Eb\left[\abs{\theta^\top V^{-1}W_i}^s\right]&\leq& \lambda_{\rm min}(V)^{-s}\sup_{\theta\in\Sb^{d-1}}\Eb\left[\abs{\theta^\top W_i}^s\right]\leq \underline{\lambda}^{-s}\sup_{\theta\in\Sb^{d-1}}\Eb\left[\Abs{(\theta^\top\Sigma^{-1/2}X_i)(Y_i-X_i^\top\beta)}^s\right]\\
    &\leq&
    \underline{\lambda}^{-s}\sup_{\theta\in\Sb^{d-1}}\left[\Eb(\theta^\top\Sigma^{-1/2}X_i)^{q_x}\right]^{s/q_x}\left[\Eb(Y_i-X_i^\top\beta)^{q}\right]^{s/q}\leq \left(\frac{K_xK_y}{\underline{\lambda}}\right)^s.
\end{eqnarray*} Hence, an application of Proposition~\ref{prop:6} yields the desired result.
\end{proof}

\begin{lemma}\label{lem:onL}For $\tilde c \in\Sb^{d-1}$, define
\begin{eqnarray*}
    \Lc_1 &:=& \Norm{\frac{1}{n}\sum_{i=1}^n\Sigma^{-1/2}X_i(\tilde c^\top\Sigma^{-1/2}X_i)^2(Y_i-X_i^\top\beta)}_2,\\
    \Lc_2&:=& \Norm{\frac{1}{n}\sum_{i=1}^n\Sigma^{-1/2}X_i(\tilde c^\top\Sigma^{-1/2}X_i)(Y_i-X_i^\top\beta)^2}_2.
\end{eqnarray*}Suppose that Assumption~\ref{asmp:2} and Assumption~\ref{asmp:3.b} hold.
\begin{enumerate}
    \item\label{lem:9.1} If $(3/q_x+1/q)^{-1}\geq 2$, then $\Pb\left(\Lc_1\leq K_x^3K_y\{1+\sqrt{d/(n\delta)}\}\right)\geq1-\delta$ for all $\delta\in(0,1)$.
    \item\label{lem:9.2} If $s=(1/q_x+1/q)^{-1}\geq 4$, then $\Pb\left(\Lc_2\leq K_x^2K_y^2\{1+\sqrt{d/(n\delta)}\}\right)\geq1-\delta$ for all $\delta\in(0,1)$.
\end{enumerate} Taking $\delta=d/(2n)$ in both inequalities shows that $\Pb(\Lc\leq (1+\sqrt{2})(K_x^3K_y+K_x^2K_y^2))\geq 1-d/n$.
    
\end{lemma}
\begin{proof}[proof of Lemma~\ref{lem:onL}] To show the first claim, we let
    $S_i = \abs{\tilde c^\top\Sigma^{-1/2}X_i}^2(Y_i-X_i^\top\beta)\Sigma^{-1/2}X_i$ for $i=1,\ldots,n$, and observe that
    \begin{eqnarray*}
        \Norm{\frac{1}{n}\sum_{i=1}^nS_i}_2\leq \Norm{\frac{1}{n}\sum_{i=1}^n\Eb S_i}_2 + \Norm{\frac{1}{n}\sum_{i=1}^nS_i-\Eb S_i}_2.
    \end{eqnarray*} The expectation of $S_i$ can be controlled as
    \begin{eqnarray*}
        \norm{\Eb S_i}_2 &=& \sup_{u\in\Sb^{d-1}}\Eb\left[\abs{\tilde c^\top\Sigma^{-1/2}X_i}^2(Y_i-X_i^\top\beta)(u^\top\Sigma^{-1/2}X_i)\right]\leq K_x^3K_y,
    \end{eqnarray*} provided that $3/q_x+1/q\leq 1$. Meanwhile,
    \begin{eqnarray*}
        \Eb\Norm{\frac{1}{n}\sum_{i=1}^nS_i-\Eb S_i}_2^2 &=& {\rm tr}\left({\rm Var}\left[\frac{1}{n}\sum_{i=1}^n S_i\right]\right)=\frac{1}{n}{\rm tr}\left({\rm Var}\left[W_1\right]\right)\\
        &\leq&\frac{1}{n}\Eb \left[\abs{\tilde c^\top\Sigma^{-1/2}X_1}^4(Y_1-X_1^\top\beta)^2\norm{X_1}_{\Sigma^{-1}}^2\right]\leq K_x^6K_y^2\frac{d}{n},
    \end{eqnarray*}provided that $3/q_x+1/q\leq 1/2$. An application of Chebyshev's inequality proves the claim.

    The second claim follows similarly. Denote $T_i = \abs{\tilde c^\top\Sigma^{-1/2}X_i}(Y_i-X_i^\top\beta)^2\Sigma^{-1/2}X_i$, and observe that
\begin{eqnarray*}
        \Norm{\frac{1}{n}\sum_{i=1}^nT_i}_2\leq \Norm{\frac{1}{n}\sum_{i=1}^n\Eb T_i}_2 + \Norm{\frac{1}{n}\sum_{i=1}^nT_i-\Eb T_i}_2.
    \end{eqnarray*} The leading term on the right-hand side can be bounded as
    \begin{eqnarray*}
        \norm{\Eb T_i}_2 &=& \sup_{u\in\Sb^{d-1}}\Eb\left[\abs{\tilde c^\top\Sigma^{-1/2}X_i}(Y_i-X_i^\top\beta)^2(u^\top\Sigma^{-1/2}X_i)\right]\leq K_x^2K_y^2,
    \end{eqnarray*} provided that $1/q_x+1/q\leq 1/2$. Furthermore, we observe that
    \begin{eqnarray*}
        \Eb\Norm{\frac{1}{n}\sum_{i=1}^nT_i-\Eb T_i}_2^2 &=& {\rm tr}\left({\rm Var}\left[\frac{1}{n}\sum_{i=1}^n T_i\right]\right)=\frac{1}{n}{\rm tr}\left({\rm Var}\left[W_1\right]\right)\\
        &\leq&\frac{1}{n}\Eb \left[\abs{\tilde c^\top\Sigma^{-1/2}X_1}^2(Y_1-X_1^\top\beta)^4\norm{X_1}_{\Sigma^{-1}}^2\right]\leq K_x^4K_y^4\frac{d}{n},
    \end{eqnarray*}provided that $1/q_x+1/q\leq 1/4$. The Chebyshev's inequality completes the proof.
\end{proof}

\subsubsection{Maximal Concentration of the Fitted Values} In this section, we aim to prove the following theorem. For the convenience of notation and clarity, we write the weighted harmonic sum of $q_x$ and $q$ as $q_{m,n}=(m/q_x+n/q)^{-1}$ for $m,n>0$.

\begin{theorem}\label{thm:A.1}
    Suppose that Assumption~\ref{asmp:2}, \ref{asmp:3.a}, \ref{asmp:4} hold for $q_x\geq4$ and $(1/q_x+1/q)^{-1}>2$. There exists constants $C = C(q_x,q,K_x,K_y,\underline{\lambda})$ such that for any $\delta_0,\delta_1,\delta_2>0$,
    \begin{equation*}
    \max_{1\leq i\leq n}\abs{X_i^\top(\hat\beta-\beta)}\leq C\left(1-9K_x\sqrt{\frac{d+2\log(2n/\delta_0)}{n}}\right)_+^{-1}\left\{\sqrt{\frac{d+\log(2n/\delta_1)}{\delta_1^{2/q_x}n^{1-2/q_x}}}+\frac{\sqrt{d}}{\delta_1^{1/q_{2,1}}n^{1-1/q_{3,2}}}+\frac{d}{\delta_2^{1/q_{2,1}}n^{1-1/q_{2,1}}}\right\},
\end{equation*} with probability at least $1-\delta_0-\delta_1-\delta_2$.
\end{theorem}

\begin{remark} In the context of Theorem~\ref{thm:A.1}, the concentration inequality we investigate is not a novel concept; one of its earliest demonstrations was provided by \cite{portnoy1985asymptotic}. They established that if the covariate vector and errors have independent Gaussian entries, then $\max_{1\leq i\leq n}\abs{X_i^\top(\hat\beta-\beta)}\xrightarrow{P}0$ holds, given that $d=o(n^{2/3}/\log^{1/3}(n))$ and $n\to\infty,$ albeit without explicitly presenting the convergence rate. It is noteworthy that given the infinite number of moments of covariates and errors, our bound tends to 0 as long as $d=o(n)$.
\end{remark}

\begin{remark}
    A crude bound can be established as $\abs{X_i^\top(\hat\beta-\beta)}\leq\norm{X_i}_{\Sigma^{-1}}\norm{\hat\beta-\beta}_\Sigma$. However, the right-hand side approximately scales as $O_p(d/\sqrt{n})$ (see Proposition~\ref{prop:11}), leading to a loose upper bound. The primary limitation of this bound stems from its failure to account for the dependency relationship between the two quantities, $X_i$ and $\hat\beta-\beta$. Heuristically, these two quantities should not exhibit high dependence, given that a single observation is expected to contribute a fraction of $1/n$ when forming the estimate $\hat\beta$. Consequently, we adopt a leave-one-out analysis to more effectively manage the subtle dependency structure.
\end{remark}

\begin{proof}[proof of Theorem~\ref{thm:A.1}]
    We consider the leave-one-out least square estimate, defined as
\begin{equation}\label{eq:loo_lse}
    \hat\beta_{(-i)}:=\hat\Sigma^{-1}_{(-i)}\frac{1}{n-1}\sum_{j\neq i}X_jY_j,
\end{equation} for $i=1,\ldots,n$ where the $\hat\Sigma_{(-i)}=\sum_{j\neq i}X_jX_j^\top/(n-1)$ denotes the leave-one-out sample Gram matrix correspondingly. It follows from the Sherman-Morrison-Woodbury matrix identity that
\begin{eqnarray*}
    X_i^\top(\hat\beta-\beta)&=&X_i^\top\hat\Sigma^{-1}\frac{1}{n}\sum_{j=1}^nX_j(Y_j-X_j^\top\beta)\\
    &=&X_i^\top\left(\frac{n}{n-1}\hat\Sigma_{(-i)}^{-1}-\frac{\frac{n}{(n-1)^2}\hat\Sigma_{(-i)}^{-1}X_iX_i^\top\hat\Sigma_{(-i)}^{-1}}{1+\frac{1}{n-1}X_i^\top\hat\Sigma_{(-i)}^{-1}X_i}\right)\frac{1}{n}\sum_{j=1}^nX_j(Y_j-X_j^\top\beta)\\
    &=&\frac{n}{n-1}\left(1+\frac{1}{n-1}X_i^\top\hat\Sigma_{(-i)}^{-1}X_i\right)^{-1}X_i^\top\hat\Sigma_{(-i)}^{-1}\frac{1}{n}\sum_{j=1}^nX_j(Y_j-X_j^\top\beta)\\
    &=&\frac{n}{n-1}\left(1+\frac{1}{n-1}X_i^\top\hat\Sigma_{(-i)}^{-1}X_i\right)^{-1}\left[\frac{n-1}{n}X_i^\top(\hat\beta_{(-i)}-\beta)+\frac{1}{n}(X_i^\top\hat\Sigma_{(-i)}^{-1}X_i)(Y_i-X_i^\top\beta)\right].
\end{eqnarray*} Hence, we have
\begin{equation}\label{eq:them:A.1.1}
    \abs{X_i^\top(\hat\beta-\beta)}\leq\abs{X_i^\top(\hat\beta_{(-i)}-\beta)}+\frac{1}{n-1}(X_i^\top\hat\Sigma_{(-i)}^{-1}X_i)\abs{Y_i-X_i^\top\beta}.
\end{equation}Inspecting the rightmost term in \eqref{eq:them:A.1.1}, we note that
\begin{equation}
    \frac{X_i^\top\hat\Sigma_{(-i)}^{-1}X_i}{X_i^\top\Sigma^{-1}X_i}\leq\sup_{\theta\in\Sb^{d-1}\theta}\frac{\theta^\top\hat\Sigma_{(-i)}^{-1}\theta}{\theta^\top\Sigma^{-1}\theta}\leq\frac{1}{\lambda_{\rm min}(\Sigma^{-1/2}\hat\Sigma_{(-i)}\Sigma^{-1/2})}.
\end{equation} Hence, we have
\begin{equation*}
    \max_{1\leq i\leq n}\abs{X_i^\top(\hat\beta-\beta)}\leq\max_{1\leq i\leq n}\abs{X_i^\top(\hat\beta_{(-i)}-\beta)}+\frac{1}{n-1}\max_{1\leq i\leq n}\set{\frac{(X_i^\top\Sigma^{-1}X_i)\abs{Y_i-X_i^\top\beta}}{\lambda_{\rm min}(\Sigma^{-1/2}\hat\Sigma_{(-i)}\Sigma^{-1/2})}}.
\end{equation*} Combining Markov's inequality and H\"older's inequality yields that for any $t\geq 0$,
\begin{eqnarray*}
    \Pb\left((X_i^\top\Sigma^{-1}X_i)\abs{Y_i-X_i^\top\beta}\geq t\right)&\leq& t^{-q_{2,1}}\Eb\left[(X_i^\top\Sigma^{-1}X_i)^{q_{2,1}}\abs{Y_i-X_i^\top\beta}^{q_{2,1}}\right]\\
    &\leq& t^{-q_{2,1}}\left\{\left(\Eb\left[(X_i^\top\Sigma^{-1}X_i)^{q_x/2}\right]\right)^{2/q_x}\left(\Eb\left[\abs{Y_i-X_i^\top\beta}^{q}\right]\right)^{1/q}\right\}^{q_{2,1}}\\
    &\leq&\left(\frac{K_x^2K_yd}{t}\right)^{q_{2,1}}.
\end{eqnarray*} Equivalently, we have
\begin{equation*}
    \Pb\left((X_i^\top\Sigma^{-1}X_i)\abs{Y_i-X_i^\top\beta}\geq K_x^2K_y d\delta^{-1/q_{2,1}}\right)\leq \delta,
\end{equation*}for any $\delta\in(0,1)$. Meanwhile, Proposition~\ref{prop:oliveira} reads
\begin{equation*}
    \Pb\left(\lambda_{\rm min}(\Sigma^{-1/2}\hat\Sigma_{(-i)}\Sigma^{-1/2})\geq 1-9K_x\sqrt{\frac{d+2\log(2/\delta_0)}{n-1}}\right)\leq \delta_0.
\end{equation*} Combining these yields that 
\begin{eqnarray*}
    \frac{1}{n-1}\frac{(X_i^\top\Sigma^{-1}X_i)\abs{Y_i-X_i^\top\beta}}{\lambda_{\rm min}(\Sigma^{-1/2}\hat\Sigma_{(-i)}\Sigma^{-1/2})}\leq C\left(1-9K_x\sqrt{\frac{d+2\log(4/\delta_0)}{n}}\right)_+^{-1}K_x^2K_y \frac{d\delta^{-1/q_{2,1}}}{n},
\end{eqnarray*} holds with probability at least $1-\delta_0-\delta$ with an absolute constant $C$. Consequently, the union-bound yields 
\begin{eqnarray*}
    \frac{1}{n-1}\max_{1\leq i\leq n}\set{\frac{(X_i^\top\Sigma^{-1}X_i)\abs{Y_i-X_i^\top\beta}}{\lambda_{\rm min}(\Sigma^{-1/2}\hat\Sigma_{(-i)}\Sigma^{-1/2})}}\leq C\left(1-9K_x\sqrt{\frac{d+2\log(2n/\delta_0)}{n}}\right)_+^{-1}K_x^2K_y \frac{d\delta^{-1/q_{2,1}}}{n^{1-1/q_{2,1}}},
\end{eqnarray*} with probability at least $1-\delta_0-\delta$ for all $\delta_0,\delta\in(0,1)$.

 We now analyze $X_i^\top(\hat\beta_{(-i)}-\beta)$. Since it is an inner product of independent quantities, we get from the conditional Markov's inequality and Assumption~\ref{asmp:3.b} that
\begin{eqnarray}\label{eq:them:A.1:cond_mar}
    \Pb\left(\abs{X_i^\top(\hat\beta_{(-i)}-\beta)}\geq t\big|\hat\beta_{(-i)}\right)\leq \norm{\hat\beta_{(-i)}-\beta}_{\Sigma}^{q_x}(K_x/t)^{q_x}\quad\mbox{almost surely,}
\end{eqnarray} for $t\geq0$. An application of Proposition~\ref{prop:11} yields that
\begin{eqnarray}\label{eq:them:A.1.2}
    \norm{\hat\beta_{(-i)}-\beta}_{\Sigma}\leq C\left(1-9K_x\sqrt{\frac{d+2\log(2/\delta_0)}{n}}\right)_+^{-1}\left[2\sqrt{\frac{d+\log(2/\delta)}{n\underline{\lambda}}}+\frac{C_{q_x,q}K_xK_yd^{1/2}}{(\delta_2/n)^{1/q_{1,1}}n}\right],
\end{eqnarray} with probability at least $1-\delta_0-\delta$. Denote the event in \eqref{eq:them:A.1.2} as $\Cc_i$ for $i=1,\ldots,n$. The law of total expectation combined with the conditional Markov's inequality in \eqref{eq:them:A.1:cond_mar} gives that
\begin{eqnarray*}
    &&\Pb\left(\abs{X_i^\top(\hat\beta_{(-i)}-\beta)}\geq t\right)=\Eb\left[\Pb\left(\abs{X_i^\top(\hat\beta_{(-i)}-\beta)}\geq t\big|\hat\beta_{(-i)}\right)\right]\leq \Eb\left[\norm{\hat\beta_{(-i)}-\beta}_{\Sigma}^{q_x}(K_x/t)^{q_x}\mathbbm{1}(\Cc_i)+\mathbbm{1}(\Cc_i^\complement)\right]\\
    &\leq&\underbrace{C\left(1-9K_x\sqrt{\frac{d+2\log(2/\delta_0)}{n}}\right)_+^{-q_x}}_{=:C(\delta_0)}t^{-q_x}\left[\left(\frac{d+\log(2/\delta)}{n}\right)^{q_x/2}+\frac{d^{q_x/2}}{\delta^{q_x/q_{1,1}}n^{q_x-q_x/q_{1,1}}}\right]+\delta_0+\delta.
\end{eqnarray*} To minimize the right-hand side with respect to $\delta$, we take $\delta = (d^{1/2}/(tn^{1-1/q_{1,1}}))^{q_{2,1}}$, resulting in
\begin{eqnarray*}
    \Pb\left(\abs{X_i^\top(\hat\beta_{(-i)}-\beta)}\geq t\right)&\leq& C(\delta_0)\left[\left(\frac{d+\log(2nt)}{t^2n}\right)^{q_x/2}+\left(\frac{d^{1/2}}{tn^{(1-1/q_{1,1})}}\right)^{q_{2,1}}\right]+\delta_0.
\end{eqnarray*} The union-bound gives that
\begin{eqnarray*}
    \Pb\left(\max_{1\leq i\leq n}\abs{X_i^\top(\hat\beta_{(-i)}-\beta)}\geq t\right)&\leq& C(\delta_0/n)n\left[\left(\frac{d+\log(2nt)}{t^2n}\right)^{q_x/2}+\left(\frac{d^{1/2}}{tn^{(1-1/q_{1,1})}}\right)^{q_{2,1}}\right]+\delta_0\\
    &=&C(\delta_0/n)\left[\left(\frac{d+\log(2nt)}{t^2n^{1-2/q_x}}\right)^{q_x/2}+\left(\frac{d^{1/2}}{tn^{(1-1/q_{3,2})}}\right)^{q_{2,1}}\right]+\delta_0,
\end{eqnarray*} for all $t\geq0$. We take
\begin{equation*}
    t = C(\delta_0/n)^{1/q_x}\left[\sqrt{\frac{d+\log(2n)}{\delta^{2/q_x}n^{1-2/q_x}}}+\frac{d^{1/2}}{\delta^{1/q_{2,1}}n^{1-1/q_{3,2}}}\right].
\end{equation*}
This yields that
\begin{eqnarray*}
    \max_{1\leq i\leq n}\abs{X_i^\top(\hat\beta_{(-i)}-\beta)}\leq C\left(1-9K_x\sqrt{\frac{d+2\log(2n/\delta_0)}{n}}\right)_+^{-1}\left[\sqrt{\frac{d+\log(2n/\delta)}{\delta^{2/q_x}n^{1-2/q_x}}}+\frac{d^{1/2}}{\delta^{1/q_{2,1}}n^{1-1/q_{3,2}}}\right],
\end{eqnarray*} with probability at least $1-\delta-\delta_0$. The right-hand side tends to $0$ as long as $d =o(n^{1-2/q_x})$ if $(1/q_x+1/q)^{-1}\geq 4$.
\end{proof}

\section{Technical Lemmas and Propositions}
\subsection{Proof of Auxiliary Results for Theorem~\ref{thm:1}}

\begin{lemma}\label{lem:5}
Let $\Dc_\Sigma=\norm{\Sigma^{-1/2}(\Sigma-\hat\Sigma)\Sigma^{-1/2}}_{\rm op}$. The following deterministic inequality holds for any $\beta\in\Real^d$.
\begin{eqnarray*}
    \Norm{\Rc}_\Sigma\leq\frac{\Dc_\Sigma^2}{\lambda_{\rm min}(\Sigma^{-1/2}\hat\Sigma\Sigma^{-1/2})}\Norm{\n\sum_{i=1}^n\Sigma^{-1/2}X_i(Y_i-X_i^\top\beta)}_2.
\end{eqnarray*}
\end{lemma}
\begin{proof}
We note from \eqref{eq:A.1.1} that
\begin{eqnarray*}
    &&\Norm{\hat\beta-\beta-\left\{\Sigma^{-1}\n\sum_{i=1}^nX_i(Y_i-X_i^\top\beta)+\Sigma^{-1}(\Sigma-\hat\Sigma)\Sigma^{-1}\n\sum_{i=1}^nX_i(Y_i-X_i^\top\beta)\right\}}_\Sigma\\
    &=&\Norm{\Sigma^{1/2}\hat\Sigma^{-1}(\Sigma-\hat\Sigma)\Sigma^{-1}(\Sigma-\hat\Sigma)\Sigma^{-1}\n\sum_{i=1}^nX_i(Y_i-X_i^\top\beta)}_2\\
    &\leq&\Norm{\Sigma^{1/2}\hat\Sigma^{-1}\Sigma^{1/2}}_{\rm op}\Norm{\Sigma^{-1/2}(\Sigma-\hat\Sigma)\Sigma^{-1/2}}_{\rm op}^2\Norm{\n\sum_{i=1}^n\Sigma^{-1/2}X_i(Y_i-X_i^\top\beta)}_2\\
    &\leq&\frac{\Dc_\Sigma^2}{\lambda_{\rm min}(\Sigma^{-1/2}\hat\Sigma\Sigma^{-1/2})}\Norm{\n\sum_{i=1}^n\Sigma^{-1/2}X_i(Y_i-X_i^\top\beta)}_2.
\end{eqnarray*}
The desired result follows from the definition of $\Rc$ in \eqref{eq:A.1.2} and Cauchy Schwarz's inequality.

\end{proof}

\begin{lemma}\label{lem:8}
    Suppose that Assumption~\ref{asmp:2},\ref{asmp:3.a}, and \ref{asmp:4} holds such that $q_x\geq4$ and $s=(1/q_x+1/q)^{-1}>2$. If $d+2\log(2n)\leq n/(18K_x)^2$, then there exists constant $C=C(\overline{\lambda}, q_x,q,K_x,K_y)$ such that
    \begin{eqnarray}\label{eq:lem:4}
        \sup_{c\in\Real^d:\norm{c}_{\Sigma^{-1}}=1}\sqrt{n}c^\top\Rc\leq C \left[\frac{d^{5/2}\log^3(2d/\delta)}{n^{2-4/q_x}\delta^{4/q_x}}+\frac{d^{3/2}\log(2d/\delta)}{n}\right]\left(1\vee\sqrt{\frac{\log(2n)}{d}}\right),
    \end{eqnarray} occurs with probability at least 
    \begin{equation*}
        1-\frac{1}{n^{\min\{1,s/2-1\}}}-\delta
    \end{equation*} for any $\delta\in(0,1)$.
\end{lemma}

\begin{lemma}\label{lem:9}
    Suppose that Assumption~\ref{asmp:2},\ref{asmp:3.b}, and \ref{asmp:4} holds such that $q>2$. Then, there exists some constant $C=C(\overline{\lambda},q,K_y)$  such that
    \begin{eqnarray}\label{eq:lem:5}
        \Pb\left[\sup_{c\in\Real^d:\norm{c}_{\Sigma^{-1}}=1}\sqrt{n}c^\top\Rc\leq C\left(\frac{d^{3/2}}{n}+\frac{d\log(2n/\delta)}{n}\right)\right]\geq 1-\frac{1}{n^{\min\{1,q/2-1\}}}-\delta,
    \end{eqnarray} for $\delta\in(0,1)$ and any $c\in\Real^d$ with $\norm{c}_{\Sigma^{-1}}=1$.
\end{lemma}

\begin{lemma}\label{lem:10}
    Suppose that Assumption~\ref{asmp:2},\ref{asmp:3.c}, and \ref{asmp:4} holds such that $q_x>4$ and $s=(1/q_x+1/q)^{-1}>2$. If $d+2\log(2n)\leq n/(18K_x)^2$, then there exists some constant $C=C(\overline{\lambda}, q_x,q,K_x,K_y)$ such that
    \begin{eqnarray}\label{eq:lem:6}
        \sup_{c\in\Real^d:\norm{c}_{\Sigma^{-1}}=1}\sqrt{n}c^\top\Rc\leq C \left[\frac{d^{1/2+4/q_x}\log^3(2d/\delta)}{n^{2-4/q_x}\delta^{4/q_x}}+\frac{d^{3/2}\log(2d/\delta)}{n}\right]\left(1\vee\sqrt{\frac{\log(2n)}{d}}\right),
    \end{eqnarray} occurs with probability at least 
    \begin{equation*}
        1-\frac{1}{n^{\min\{1,s/2-1\}}}-\delta
    \end{equation*} for any $\delta\in(0,1)$.
\end{lemma}

Lemma \ref{lem:8}---\ref{lem:10} share a common structure, differing only in the moment assumption for covariates upon which they rely. As evidenced by Lemma~\ref{lem:5}, the deterministic upper bound for $\Rc$ hinges on the computation of three pivotal quantities, specifically, $\Dc_\Sigma$, $\lambda_{\rm min}(\Sigma^{-1/2}\hat\Sigma\Sigma^{-1/2})$, and the average of influence functions. Notably, Proposition~\ref{prop:6} offers three distinct concentration inequalities for $\Dc_\Sigma$ based on the three different moment assumptions. The quantity $\lambda_{\rm min}(\Sigma^{-1/2}\hat\Sigma\Sigma^{-1/2})$ is controlled via Theorem 4.1 of \cite{oliveira2016lower} as already done in \eqref{eq:oliveira}. Lastly, the average of influence functions is controlled in Proposition~\ref{prop:7}. 

\begin{proof}[proof of Lemma~\ref{lem:8}---\ref{lem:10}]
    Proposition~\ref{prop:6} presents the series of tail bounds which has the form of,
    \begin{equation}\label{eq:lem.1}
        \Pb\left(\Dc_\Sigma\leq \epsilon_1(\delta)\right)\geq 1-\delta.
    \end{equation} Meanwhile, an application of Proposition~\ref{prop:7} with $\delta=n^{1-s/2}$, we have
    \begin{eqnarray}\label{eq:lem.2}
        \Norm{\frac{1}{n}\sum_{i=1}^n\Sigma^{-1/2}X_i(Y_i-X_i^\top\beta)}_2&\leq&\sqrt{\frac{d+(s/2-1)\log(2n)}{n\overline{\lambda}}} +C_sK_xK_y\sqrt{\frac{d}{n}}\nonumber\\
        &\leq&\left(\frac{1}{\overline{\lambda}^{1/2}}\vee C_sK_xK_y\right)\sqrt{\frac{d}{n}}+\sqrt{\frac{(s/2-1)\log(2n)}{n\overline{\lambda}}}\nonumber\\
        &\leq&\left(\frac{s-2}{2\overline{\lambda}^{1/2}}\vee C_sK_xK_y\right)\left[\sqrt{\frac{d}{n}}+\sqrt{\frac{\log(2n)}{n}}\right],
    \end{eqnarray} with probability at least $1-n^{1-s/2}$. Lastly, an application of Proposition~\ref{prop:oliveira} yields that if $d+2\log(2n)\leq n/(18K_x)^2$, then
    \begin{equation}\label{eq:lem.3}
        \Pb\left(\lambda_{\rm min}(\Sigma^{-1/2}\hat\Sigma\Sigma^{-1/2})\geq \frac{1}{2}\right)\leq 1-\frac{1}{n}.
    \end{equation}
    Inspecting the deterministic inequality in Lemma~\ref{lem:5}, and by combining probabilistic inequality \eqref{eq:lem.1}, \eqref{eq:lem.2}, and \eqref{eq:lem.3}, we can deduce that 
    \begin{equation*}
        \sup_{c\in\Real^d:\norm{c}_{\Sigma^{-1}}=1}\sqrt{n}c^\top\Rc=\sqrt{n}\Norm{\Rc}_\Sigma\leq \sqrt{n}C\epsilon_1(\delta)^2\left[\sqrt{\frac{d}{n}}+\sqrt{\frac{\log(2n)}{n}}\right],
    \end{equation*} with probability at least $1-2n^{-\min\{1,s/2-1\}}-\delta$. By following the argument in Proposition~\ref{prop:6} and substituting $\epsilon_1(\delta)$ with their respective values under the three assumptions, the proof is completed.
    
\end{proof}

\begin{lemma}\label{lem:11}
    For $c\in\Real^d$, let $N_c(x)=\Phi(x)-\kappa_c\Phi^{(3)}(x)$ where $\kappa_c$ is defined in \eqref{eq:2.2.6}. Suppose that Assumption~\ref{asmp:2},\ref{asmp:3.a}, and \ref{asmp:4} hold for some $q_x\geq4$ and $(1/q+1/q_x)^{-1}\geq 3$. Then, the distribution of (scaled) $c^\top\Uc$ can be approximated with $N_c$ as
    \begin{equation*}
        \sup_{c\in\Real^d:\norm{c}_{\Sigma^{-1}}=1}\sup_{x\in\Real} \Abs{\Pb\left(\sqrt{n}\frac{c^\top\Uc}{\sigma_c}\leq x\right)-N_c(x)}\leq C\left(\frac{\bar\lambda^{3/2}K_x^3K_y^3}{\sqrt{n}}+\frac{\bar\lambda K_x^6K_y^2d}{2n}\right),
    \end{equation*} for some absolute constant $C>0$. Moreover, there exists a (possibly different) absolute constant $C>0$ such that
    \begin{eqnarray*}
        \sup_{c\in\Real^d:\norm{c}_{\Sigma^{-1}}=1}\sup_{x\in\Real} \Abs{\Pb\left(\sqrt{n}\frac{c^\top\Uc}{\sigma_c}\leq x\right)-\Phi(x)}\leq C\left(\frac{\bar\lambda^{3/2}K_x^3K_y^3}{\sqrt{n}}+\frac{\bar\lambda^{1/2} K_x^3K_y\sqrt{d}}{\sqrt{2n}}\right).
    \end{eqnarray*} Finally, the parameter $\kappa_c$ is uniformly bounded as
    \begin{equation*}
        \kappa_c\leq \frac{\bar\lambda^{1/2}K_x^3K_y}{2}\sqrt{\frac{d}{n}},\quad \forall c\in\Real^d.
    \end{equation*}
\end{lemma}
\begin{proof}
    For any $c\in\Real^d$ with $\norm{c}_{\Sigma^{-1}}=1$, let
    \begin{eqnarray*}
        \eta_c&=&n^{-1/2}\Eb \Abs{c^\top\psi(X_1,Y_1)}^3/\sigma_c^3,\\
    \gamma_c&=&\binom{n}{2}n^{-3}\Eb \Abs{c^\top\phi(X_1,Y_1,X_2,Y_2)}^2/\sigma_c^2,
    \end{eqnarray*} where $\psi$ and $\phi$ are defined in \eqref{eq:psi_and_phi}. Theorem~\ref{thm:1} of \cite{bentkus2009normal} implies that there exists a universal constant $C>0$ such that
\begin{eqnarray*}
    \sup_{x\in\Real} \Abs{\Pb\left(\sqrt{n}\frac{c^\top\Uc}{\sigma_c}\leq x\right)-N_c(x)}\leq C\left(\eta_c+\gamma_c\right).
\end{eqnarray*} We first control the asymptotic variance $\sigma_c^2$. From the definition of $\sigma_c^2$, Assumption~\ref{asmp:4} leads to that
\begin{eqnarray*}
    \left(1+\frac{1}{n}\right)^2\bar\lambda^{-1}\leq\sigma_c^2=\left(1+\frac{1}{n}\right)^2{\rm Var}\left[c^\top\Sigma^{-1}X(Y-X^\top\beta)\right]\leq\left(1+\frac{1}{n}\right)^2\overline{\lambda}^{-1}.
\end{eqnarray*} Now, we bound two quantities, $\eta_c$ and $\gamma_c$, respectively. First, Jensen's inequality yields
\begin{eqnarray*}
    \eta_c &=& \frac{(1+1/n)^3}{\sigma_c^3\sqrt{n}}\Eb\Abs{c^\top\Sigma^{-1}X_1(Y_1-X_1^\top\beta)}^3\\
    &\leq& \frac{\bar\lambda^{3/2}}{\sqrt{n}}\left(\Eb\Abs{c^\top\Sigma^{-1}X_1(Y_1-X_1^\top\beta)}^s\right)^{3/s},
\end{eqnarray*} where $s=(1/q_x+1/q)^{-1}\geq 3$. An application of H\"older's inequality with Assumption~\ref{asmp:2} and \ref{asmp:3.a} implies that
\begin{eqnarray*}
    \eta_c &\leq& \frac{\bar\lambda^{3/2}}{\sqrt{n}}\left(\Eb\Abs{c^\top\Sigma^{-1}X_1(Y_1-X_1^\top\beta)}^{q_xq/(q_x+q)}\right)^{3/s}\\
    &\leq&\frac{\bar\lambda^{3/2}}{\sqrt{n}}\left\{\left(\Eb\Abs{c^\top\Sigma^{-1}X_1}^{q_x}\right)^{q/(q_x+q)}\left(\Eb\Abs{Y_1-X_1^\top\beta}^q\right)^{q_x/(q_x+q)}\right\}^{3/s}\\
    &\leq&\frac{\bar\lambda^{3/2}K_x^3K_y^3}{\sqrt{n}}.
\end{eqnarray*} On the other hand, the definition of $\phi$ implies that
\begin{eqnarray*}
    \Eb\Abs{c^\top\phi(X_1,Y_1,X_2,Y_2)}^2 &=& \left(1-\frac{1}{n}\right)^2\Eb|c^\top\Sigma^{-1}(\Sigma-X_1X_1^\top)\Sigma^{-1}X_2(Y_2-X_2^\top\beta)|^2\\&\leq&\Eb|c^\top\Sigma^{-1}(\Sigma-X_1X_1^\top)\Sigma^{-1}X_2(Y_2-X_2^\top\beta)|^2.
\end{eqnarray*} For any $v\in\Real^d$, we note that Jensen's inequality yields
\begin{eqnarray}\label{eq:M_2,2}
    \Eb\Abs{v^\top \Sigma^{-1/2}X_2(Y_2-X_2^\top\beta)}^2&\leq& \norm{v}_2^2\left(\Eb\Abs{(v/\norm{v}_2)^\top \Sigma^{-1/2}X_2(Y_2-X_2^\top\beta)}^s\right)^{2/s}\nonumber\\
    &\leq& \norm{v}_2^2\left\{\left(\Eb\Abs{(v/\norm{v}_2)^\top\Sigma^{-1}X_2}^{q_x}\right)^{q/(q_x+q)}\left(\Eb\Abs{Y_2-X_2^\top\beta}^q\right)^{q_x/(q_x+q)}\right\}^{2/s}\nonumber\\
    &\leq& \norm{v}_2^2K_x^2K_y^2,
\end{eqnarray} where the first inequality is Jensen's inequality, and the second inequality follows from H\"older's inequality. Combining this with the independence of $(X_1,Y_1)$ and $(X_2,Y_2)$ yields
\begin{eqnarray*}
    \Eb\Abs{c^\top\Sigma^{-1}(\Sigma-X_1X_1^\top)\Sigma^{-1}X_2(Y_2-X_2^\top\beta)}^2\leq K_x^2K_y^2\Eb\norm{c^\top\Sigma^{-1}(\Sigma-X_1X_1^\top)\Sigma^{-1/2}}_2^2.
\end{eqnarray*} Inspecting the right-hand side, we have
\begin{eqnarray*}
    \Eb\norm{c^\top\Sigma^{-1}(\Sigma-X_1X_1^\top)\Sigma^{-1/2}}_2^2 &=& \Eb\{c^\top\Sigma^{-1}(\Sigma-X_1X_1^\top)\Sigma^{-1}(\Sigma-X_1X_1^\top)\Sigma^{-1}c\}\\
    &=&c^\top\Sigma^{-1/2}\Eb\left(\Sigma^{-1/2}X_1X_1^\top\Sigma^{-1/2}\norm{X_1}_{\Sigma^{-1}}^2-I\right)\Sigma^{-1/2}c\\
    &\leq&c^\top\Sigma^{-1/2}\Eb\left(\Sigma^{-1/2}X_1X_1^\top\Sigma^{-1/2}\norm{X_1}_{\Sigma^{-1}}^2\right)\Sigma^{-1/2}c\\&\leq&\sup_{u\in\Sb^{d-1}}\Eb\left\{\left(u^\top\Sigma^{-1/2}X_1\right)^2\norm{X_1}_{\Sigma^{-1}}^2\right\}.
\end{eqnarray*} To control the quantity $\norm{X_1}_{\Sigma^{-1}}^2$, let $\ev_j$ be the $j$:th canonical basis of $\Real^d$ for $1\leq j \leq d$. Then,
\begin{eqnarray*}
    \sup_{u\in\Sb^{d-1}}\Eb\left\{\left(u^\top\Sigma^{-1/2}X_1\right)^2\norm{X_1}_{\Sigma^{-1}}^2\right\} &=& \sup_{u\in\Sb^{d-1}}\Eb\left\{\left(u^\top\Sigma^{-1/2}X_1\right)^2\sum_{j=1}^d (\ev_j^\top\Sigma^{-1/2}X_1)^2\right\}\\
    &=& \sup_{u\in\Sb^{d-1}}\sum_{j=1}^d
    \Eb\left\{\left(u^\top\Sigma^{-1/2}X_1\right)^2(\ev_j^\top\Sigma^{-1/2}X_1)^2\right\}\\
    &\leq&d\sup_{u\in\Sb^{d-1}}\max_{1\leq j\leq d}
    \Eb\left\{\left(u^\top\Sigma^{-1/2}X_1\right)^2(\ev_j^\top\Sigma^{-1/2}X_1)^2\right\}\\
    &\leq&d\sup_{u\in\Sb^{d-1}}
    \Eb\left\{\left(u^\top\Sigma^{-1/2}X_1\right)^4\right\}\leq dK_x^4.
\end{eqnarray*} Here, the next-to-last inequality is Cauchy Schwarz inequality, and the last inequality is due to Assumption~\ref{asmp:3.a}. Combining all, we have
\begin{eqnarray}
    \gamma_c&\leq& \frac{\bar\lambda}{2n}\Eb\Abs{c^\top\phi(X_1,Y_1,X_2,Y_2)}^2\nonumber\\
    &\leq&\frac{\bar\lambda K_x^6K_y^2d}{2n}\label{eq:bound_phi}.
\end{eqnarray} This concludes the proof of the first part. The last part can be done by applying Theorem~\ref{thm:2} of \cite{bentkus2009normal}, which proves the following: there exists a universal constant $C>0$ such that
\begin{eqnarray*}
    \sup_{x\in\Real} \Abs{\Pb\left(\sqrt{n}\frac{c^\top\Uc}{\sigma_c}\leq x\right)-\Phi(x)}\leq C\left(\eta_c+\gamma_c^{1/2}\right).
\end{eqnarray*} Finally, we control $\kappa_c$. From its definition in \eqref{eq:2.2.7} and Cauchy Schwarz's inequality, we have
\begin{eqnarray*}
    \kappa_c&\leq& \frac{1}{2\sigma_c^3\sqrt{n}}\left(\Eb\left[\left\{c^\top\psi(X_1,Y_1)\right\}^2\left\{c^\top\psi(X_2,Y_2)\right\}^2\right]\right)^{1/2}\left(\Eb\left[\left\{c^\top\phi(X_1,Y_1,X_2,Y_2)\right\}^2\right]\right)^{1/2}\\
    &=&\frac{1}{2\sigma_c\sqrt{n}}\left(\Eb\left[\left\{c^\top\phi(X_1,Y_1,X_2,Y_2)\right\}^2\right]\right)^{1/2}\\
    &\leq&\frac{K_x^3K_y\sqrt{d}}{2\sigma_c\sqrt{n}}\leq\frac{\bar\lambda^{1/2}K_x^3K_y\sqrt{d}}{2\sqrt{n}}.
\end{eqnarray*} This completes the proof.
\end{proof}
\begin{lemma}\label{lem:12}
    Suppose that Assumption~\ref{asmp:2} and \ref{asmp:3.c} hold with $(3/q_x+1/q)^{-1}\geq 2$. Then, 
    \begin{equation}\label{eq:lem:8}
        \Pb\left(\Abs{\sqrt{n}c^\top\Bc}>K_x^3K_y\left(\sqrt{\frac{d}{n}}+\frac{d}{n\sqrt{\delta}}\right)\right)\leq \delta,
    \end{equation} for $\delta\in(0,1)$.
\end{lemma}
\begin{proof}
    Recall that
    \begin{eqnarray*}
        \Bc = -\frac{1}{n^2}\sum_{i=1}^n V_i,
    \end{eqnarray*} where $V_i = \Sigma^{-1}X_i(Y_i-X_i^\top\beta)\norm{X_i}_{\Sigma^{-1}}^2$ for $i=1,\ldots,n$. Chebyshev's inequality implies that for any $\epsilon >0$
    \begin{eqnarray}\label{eq:lem8:basic_ineq}
        \Pb\left(\Abs{\sqrt{n}c^\top\Bc}>\Abs{\Eb[\sqrt{n}c^\top\Bc]}+\epsilon\right)\leq \epsilon^{-2}{\rm Var}(\sqrt{n}c^\top\Bc).
    \end{eqnarray} Since $\beta$ is the projection parameter, note that
    \begin{eqnarray*}
        \Eb c^\top V_i &=& \Eb c^\top\Sigma^{-1}X_i(Y_i-X_i^\top\beta)\norm{X_i}_{\Sigma^{-1}}^2\\
        &=& {\rm Cov}\left(c^\top\Sigma^{-1}X_i(Y_i-X_i^\top\beta), \norm{X_i}_{\Sigma^{-1}}^2\right).
    \end{eqnarray*} This leads to
    \begin{eqnarray*}
        \Abs{\Eb c^\top V_i}\leq \left[{\rm Var}\left\{c^\top\Sigma^{-1}X_i(Y_i-X_i^\top\beta)\right\}\right]^{1/2}\left\{{\rm Var}\left(\norm{X_i}_{\Sigma^{-1}}^2\right)\right\}^{1/2}.
    \end{eqnarray*} The primary term on the right-hand side can be bounded in a similar manner to the approach outlined in equation \eqref{eq:M_2,2}:
    \begin{equation}\label{eq:lem.13.1}
        {\rm Var}\left\{c^\top\Sigma^{-1}X_i(Y_i-X_i^\top\beta)\right\}\leq \Eb \Abs{c^\top\Sigma^{-1}X_i}^2\Abs{Y_i-X_i^\top\beta}^2\leq K_x^2K_y^2.
    \end{equation} To control $\norm{X_i}_{\Sigma^{-1}}$, write $Z_i = \Sigma^{-1/2}X_i$ and $Z_i = (Z_i(1),\ldots,Z_i(d))^\top$ where $\Eb Z_i(j)^2=1$ and $Z_i(j)$ for $1\leq j\leq d$ are independent. Then, we get
\begin{eqnarray}\label{eq:lem.13.2}
    {\rm Var}(\norm{X_i}_{\Sigma^{-1}}^2) &=& \Eb \norm{X_i}_{\Sigma^{-1}}^4 - \left(\Eb \norm{X_i}_{\Sigma^{-1}}^2\right)^2\nonumber\\
    &=&\sum_{j=1}^d\Eb [Z_i(j)^4] +\sum_{j\neq j'}\Eb [Z_i(j)^2Z_i(j')^2] - d^2\nonumber\\
    &=&\sum_{j=1}^d\Eb [Z_i(j)^4] - d \leq d(K_x^4-1).
\end{eqnarray} The last inequality is due to Assumption~\ref{asmp:3.a}. Combining \eqref{eq:lem.13.1} and \eqref{eq:lem.13.2}, we get $|\Eb c^\top V_i|\leq K_x^3K_y \sqrt{d}$ for $i=1,\ldots,n$, and this implies that
\begin{eqnarray}\label{eq:bias_mean_bound}
    \Abs{\Eb\sqrt{n} c^\top\Bc}\leq K_x^3K_y \sqrt{\frac{d}{n}}.
\end{eqnarray} Now we focus on the variance of the bias $\Bc$. We note that
\begin{eqnarray*}
    {\rm Var}\left(\sqrt{n} c^\top\Bc\right) &=&\frac{1}{n^2} {\rm Var}(c^\top V_1)\leq \frac{1}{n^2} \Eb [(c^\top V_1)^2]=\frac{1}{n^2}\Eb(c^\top\Sigma^{-1}X_i)^2(Y_i-X_i^\top\beta)^2\norm{X_i}_{\Sigma^{-1}}^4\\
    &=&\frac{1}{n^2}\Eb(c^\top\Sigma^{-1}X_i)^2(Y_i-X_i^\top\beta)^2\left(\sum_{j=1}^d Z_i(j)^2\right)^2\\
    &\leq& \frac{d}{n^2}\Eb(c^\top\Sigma^{-1}X_i)^2(Y_i-X_i^\top\beta)^2\left(\sum_{j=1}^d Z_i(j)^4\right)\\
    &=&\frac{d}{n^2}\sum_{j=1}^d\Eb(c^\top\Sigma^{-1}X_i)^2(Y_i-X_i^\top\beta)^2 Z_i(j)^4,
\end{eqnarray*} where two inequalities are Cauchy Schwarz's inequalities. Let $l := (6/q_x+2/q)^{-1}\geq 1$. Combining Jensen's inequality and H\"older's inequality implies that
\begin{eqnarray*}
    \Eb(c^\top\Sigma^{-1}X_i)^2(Y_i-X_i^\top\beta)^2 Z_i(j)^4 &\leq& \left[\Eb|c^\top\Sigma^{-1}X_i|^{2l}|Y_i-X_i^\top\beta|^{2l} |Z_i(j)|^{4l}\right]^{1/l}\\
    &\leq&\left[\left\{\Eb|c^\top\Sigma^{-1}X_i|^{q_x}\right\}^{2l/q_x}\left\{\Eb|Y_i-X_i^\top\beta|^{q}\right\}^{2l/q} \left\{\Eb|Z_i(j)|^{q_x}\right\}^{4l/q_x}\right]^{1/l}\\
    &\leq&K_x^6K_y^2.
\end{eqnarray*} Hence, we get ${\rm Var}(\sqrt{n} c^\top\Bc)\leq K_x^6K_y^2d^2/n^2$. Combining this with \eqref{eq:bias_mean_bound} in \eqref{eq:lem8:basic_ineq} completes the proof.
\end{proof}

\subsection{Proof of Auxiliary Results for Theorem~\ref{thm:2}}

\begin{lemma}\label{lem:rem1} Suppose that Assumption~\ref{asmp:2},\ref{asmp:3.a}, and \ref{asmp:4} holds with $(3/q_x+1/q)^{-1}\geq 1/2$. Then, it holds that $$\Sc_1\leq \frac{2K_x^3K_y\Dc_\Sigma}{\lambda_{\rm min}(\Sigma^{-1/2}\hat\Sigma\Sigma^{-1/2})},$$ with probability at least $1-d/n$.
\end{lemma}
\begin{proof}
    Note that
\begin{eqnarray}
    \Sc_1 &\leq&\norm{c^\top(\hat\Sigma^{-1}-\Sigma^{-1})\Sigma^{1/2}}_2\Norm{\n\sum_{i=1}^n\Sigma^{-1/2}X_i(Y_i-X_i^\top\beta)\frac{\norm{X_i}_{\Sigma^{-1}}^2}{d}}_2\nonumber\\
    &\leq& \frac{\Dc_\Sigma}{\lambda_{\rm min}(\Sigma^{-1/2}\hat\Sigma\Sigma^{-1/2})}\Norm{\n\sum_{i=1}^n\Sigma^{-1/2}X_i(Y_i-X_i^\top\beta)\frac{\norm{X_i}_{\Sigma^{-1}}^2}{d}}_2\label{eq:lem:rem1.1},
\end{eqnarray} since $\norm{c}_{\Sigma^{-1}}=1$. Hence, it suffices to control the rightmost quantity in \eqref{eq:lem:rem1.1}. We note that
\begin{align*}
    &\Norm{\n\sum_{i=1}^n\Sigma^{-1/2}X_i(Y_i-X_i^\top\beta)\frac{\norm{X_i}_{\Sigma^{-1}}^2}{d}}_2\leq \Norm{\Eb\Sigma^{-1/2}X_1(Y_1-X_1^\top\beta)\frac{\norm{X_1}_{\Sigma^{-1}}^2}{d}}_2\\
    &+\Norm{\n\sum_{i=1}^n\Sigma^{-1/2}X_i(Y_i-X_i^\top\beta)\frac{\norm{X_i}_{\Sigma^{-1}}^2}{d}-\Eb\Sigma^{-1/2}X_1(Y_1-X_1^\top\beta)\frac{\norm{X_1}_{\Sigma^{-1}}^2}{d}}_2.
\end{align*} The leading term on the right-hand side can be bounded via Assumption~\ref{asmp:2} and \ref{asmp:3.a} as
\begin{eqnarray}
    \Norm{\Eb\Sigma^{-1/2}X_1(Y_1-X_1^\top\beta)\frac{\norm{X_1}_{\Sigma^{-1}}^2}{d}}_2&\leq&\sup_{\theta\in\Sb^{d-1}}\Abs{\Eb\left[\theta^\top\Sigma^{-1/2}X_1(Y_1-X_1^\top\beta)\frac{\norm{X_1}_{\Sigma^{-1}}^2}{d}\right]}\nonumber\\
    &\leq&\sup_{\theta\in\Sb^{d-1}}\max_{1\leq j\leq d}\Abs{\Eb\left[\theta^\top\Sigma^{-1/2}X_1(Y_1-X_1^\top\beta)(\ev_j^\top\Sigma^{-1/2}X_1)^2\right]}\nonumber\\
    &\leq&\sup_{\theta\in\Sb^{d-1}}\Eb\Abs{\theta^\top\Sigma^{-1/2}X_1}^3\Abs{Y_1-X_1^\top\beta}\nonumber\\
    &\leq& K_x^3K_y\label{eq:rem1.2}.
\end{eqnarray} To control the second term, we bound its second moment as
\begin{eqnarray}
    &&\Eb\Norm{\n\sum_{i=1}^n\Sigma^{-1/2}X_i(Y_i-X_i^\top\beta)\frac{\norm{X_i}_{\Sigma^{-1}}^2}{d}-\Eb\Sigma^{-1/2}X_1(Y_1-X_1^\top\beta)\frac{\norm{X_1}_{\Sigma^{-1}}^2}{d}}_2^2\nonumber\\
    &=&\frac{1}{n}{\rm Var}\left(\Sigma^{-1/2}X_i(Y_i-X_i^\top\beta)\frac{\norm{X_i}_{\Sigma^{-1}}^2}{d}\right)\nonumber\\
    &\leq&\n\Eb\left[\norm{X_1}_{\Sigma^{-1}}^2(Y_1-X_1^\top\beta)^2\frac{\norm{X_1}_{\Sigma^{-1}}^4}{d^2}\right]\nonumber\\
    &\leq&\frac{d}{n}\sup_{\theta\in\Sb^{d-1}}\Eb\left[(\theta^\top\Sigma^{-1/2}X_1)^6(Y_1-X_1^\top\beta)^2\right]\nonumber\\
    &\leq&\frac{K_x^6K_y^2d}{n}\label{eq:rem1.3}.
\end{eqnarray} Combining \eqref{eq:rem1.2} and \eqref{eq:rem1.3} with Chebyshev inequality yields that for any $\delta\in(0,1)$,
\begin{eqnarray*}
    \Pb\left(\Norm{\n\sum_{i=1}^n\Sigma^{-1/2}X_i(Y_i-X_i^\top\beta)\frac{\norm{X_i}_{\Sigma^{-1}}^2}{d}}_2\geq K_x^3K_y+K_x^3K_y\sqrt{\frac{d }{n\delta}}\right)\leq \delta.
\end{eqnarray*} Taking $\delta = d/n$ leads to the intended conclusion.
\end{proof}

\begin{lemma}\label{lem:rem2}
    Suppose that Assumption~\ref{asmp:2},\ref{asmp:3.a}, and \ref{asmp:4} holds with $q_x\geq 8$. Then, it holds that $$\Sc_2\leq 2K_x^4\norm{\hat\beta-\beta}_\Sigma,$$ with probability at least $1-d/n$.
\end{lemma}
\begin{proof} The definition of $\Sc_2$ leads to
    \begin{eqnarray*}
        \Sc_2\leq\norm{\hat\beta-\beta}_\Sigma\Norm{\n\sum_{i=1}^nc^\top\Sigma^{-1}X_iX_i^\top\Sigma^{-1/2}\frac{\norm{X_i}_{\Sigma^{-1}}^2}{d}}_2.
    \end{eqnarray*} Hence, it suffices to control the following term on the right-hand side. We note that
    \begin{eqnarray*}
    &&\Norm{\n\sum_{i=1}^nc^\top\Sigma^{-1}X_iX_i^\top\Sigma^{-1/2}\frac{\norm{X_i}_{\Sigma^{-1}}^2}{d}}_2\leq \Norm{\Eb c^\top\Sigma^{-1}X_1X_1^\top\Sigma^{-1/2}\frac{\norm{X_1}_{\Sigma^{-1}}^2}{d}}_2\\
    &&+\Norm{\n\sum_{i=1}^nc^\top\Sigma^{-1}X_iX_i^\top\Sigma^{-1/2}\frac{\norm{X_i}_{\Sigma^{-1}}^2}{d}-\Eb c^\top\Sigma^{-1}X_1X_1^\top\Sigma^{-1/2}\frac{\norm{X_1}_{\Sigma^{-1}}^2}{d}}_2.
    \end{eqnarray*} The first term on the right-hand side can be bounded as
\begin{eqnarray*}
    \Norm{\Eb c^\top\Sigma^{-1}X_1X_1^\top\Sigma^{-1/2}\frac{\norm{X_1}_{\Sigma^{-1}}^2}{d}}_2&\leq&\Norm{\Eb\Sigma^{-1/2}X_1X_1^\top\Sigma^{-1/2}\frac{\norm{X_1}_{\Sigma^{-1}}^2}{d}}_{\rm op}\\
    &\leq&\sup_{\theta\in\Sb^{d-1}}\Eb\Abs{\theta^\top\Sigma^{-1/2}X_1}^4\\
    &\leq& K_x^4.
\end{eqnarray*} The second term is bounded using the second moment;
\begin{eqnarray*}
    &&\Eb\Norm{\n\sum_{i=1}^nc^\top\Sigma^{-1}X_iX_i^\top\Sigma^{-1/2}\frac{\norm{X_i}_{\Sigma^{-1}}^2}{d}-\Eb c^\top\Sigma^{-1}X_1X_1^\top\Sigma^{-1/2}\frac{\norm{X_1}_{\Sigma^{-1}}^2}{d}}_2^2\\
    &\leq&\frac{1}{n}\Eb\Norm{c^\top\Sigma^{-1}X_1X_1^\top\Sigma^{-1/2}\frac{\norm{X_1}_{\Sigma^{-1}}^2}{d}}_2^2\\
    &=&\frac{1}{n}\Eb\left[(\tilde c^\top\Sigma^{-1/2}X_1)^2\frac{\norm{X_1}_{\Sigma^{-1}}^6}{d^2}\right]\\
    &\leq&\frac{d}{n}\sup_{\theta\in\Sb^{d-1}}\Eb\Abs{\theta^\top\Sigma^{-1/2}X_1}^8\\&\leq& \frac{K_x^8d}{n}.
\end{eqnarray*} Consequently, the tail bound follows from Chebyshev's inequality as
\begin{equation*}
    \Pb\left(\Norm{\n\sum_{i=1}^nc^\top\Sigma^{-1}X_iX_i^\top\Sigma^{-1/2}\frac{\norm{X_i}_{\Sigma^{-1}}^2}{d}}_2\geq K_x^4+K_x^4\sqrt{\frac{d}{n\delta}}\right)\leq\delta,
\end{equation*} for $\delta\in(0,1)$. Taking $\delta=d/n$ yields the result.
\end{proof}
\begin{lemma}\label{lem:rem3}
    Suppose that Assumption~\ref{asmp:2},\ref{asmp:3.a}, and \ref{asmp:4} holds with $(3/q_x+1/q)^{-1}\geq 1/2$. Then, 
    \begin{equation*}
    \Sc_3 \leq \frac{2K_x^3K_y\Dc_\Sigma}{\lambda_{\rm min}(\Sigma^{-1/2}\hat\Sigma\Sigma^{-1/2})},
\end{equation*} holds with probability at least $1-1/n$.
\end{lemma}
\begin{proof} Recall that
    \begin{equation*}
        \Abs{R_{3,i}}\leq \frac{\Dc_\Sigma}{\lambda_{\rm min}(\Sigma^{-1/2}\hat\Sigma\Sigma^{-1/2})}.
    \end{equation*} This implies that
\begin{equation*}
    \Sc_3\leq\frac{\Dc_\Sigma}{\lambda_{\rm min}(\Sigma^{-1/2}\hat\Sigma\Sigma^{-1/2})}\Abs{\frac{1}{n}\sum_{i=1}^nc^\top\Sigma^{-1}X_i(Y_i-X_i^\top\beta)\frac{\norm{X_i}_{\Sigma^{-1}}^2}{d}}.
\end{equation*} We can control the rightmost quantity routinely. Note that
\begin{align*}
    &\Abs{\frac{1}{n}\sum_{i=1}^nc^\top\Sigma^{-1}X_i(Y_i-X_i^\top\beta)\frac{\norm{X_i}_{\Sigma^{-1}}^2}{d}}\leq \Abs{\Eb c^\top\Sigma^{-1}X_1(Y_1-X_1^\top\beta)\frac{\norm{X_1}_{\Sigma^{-1}}^2}{d}}\\
    &+\Abs{\frac{1}{n}\sum_{i=1}^nc^\top\Sigma^{-1}X_i(Y_i-X_i^\top\beta)\frac{\norm{X_i}_{\Sigma^{-1}}^2}{d}-\Eb c^\top\Sigma^{-1}X_1(Y_1-X_1^\top\beta)\frac{\norm{X_1}_{\Sigma^{-1}}^2}{d}}.
\end{align*} The first term on the right-hand side can be bounded using Assumption~\ref{asmp:2} and \ref{asmp:3.a} as
\begin{eqnarray*}
    \Abs{\Eb c^\top\Sigma^{-1}X_1(Y_1-X_1^\top\beta)\frac{\norm{X_1}_{\Sigma^{-1}}^2}{d}}&\leq&\norm{c}_{\Sigma^{-1}}\Norm{\Eb\Sigma^{-1/2}X_1(Y_1-X_1^\top\beta)\frac{\norm{X_1}_{\Sigma^{-1}}^2}{d}}_{\rm op}\\
    &=&\sup_{\theta\in\Sb^{d-1}}\Eb\Abs{\theta^\top\Sigma^{-1/2}X_1(Y_1-X_1^\top\beta)\frac{\norm{X_1}_{\Sigma^{-1}}^2}{d}}\\
    &\leq&\sup_{\theta\in\Sb^{d-1}}\Eb\Abs{(\theta^\top\Sigma^{-1/2}X_1)^3(Y_1-X_1^\top\beta)}\\
    &\leq& K_x^3K_y.
\end{eqnarray*} The second term can be bounded as
\begin{eqnarray*}
    &&\Abs{\frac{1}{n}\sum_{i=1}^nc^\top\Sigma^{-1}X_i(Y_i-X_i^\top\beta)\frac{\norm{X_i}_{\Sigma^{-1}}^2}{d}-\Eb c^\top\Sigma^{-1}X_1(Y_1-X_1^\top\beta)\frac{\norm{X_1}_{\Sigma^{-1}}^2}{d}}^2\\
    &\leq&\frac{1}{n}\Eb\Abs{c^\top\Sigma^{-1}X_1(Y_1-X_1^\top\beta)\frac{\norm{X_1}_{\Sigma^{-1}}^2}{d}}^2\\
    &=&\frac{1}{n}\Eb\left[(c^\top\Sigma^{-1}X_1)^2(Y_1-X_1^\top\beta)^2\frac{\norm{X_1}_{\Sigma^{-1}}^4}{d^2}\right]\\
    &\leq&\frac{1}{n}\sup_{\theta\in\Sb^{d-1}}\Eb\left[(\theta^\top\Sigma^{-1}X_1)^6(Y_1-X_1^\top\beta)^2\right]\leq \frac{K_x^6K_y^2}{n}.
\end{eqnarray*} Now, Chebyshev's inequality leads to that
\begin{equation*}
    \Pb\left(\Abs{\frac{1}{n}\sum_{i=1}^nc^\top\Sigma^{-1}X_i(Y_i-X_i^\top\beta)\frac{\norm{X_i}_{\Sigma^{-1}}^2}{d}}\geq K_x^3K_y+K_x^3K_y\sqrt{\frac{1}{n\delta}}\right)\leq \delta.
\end{equation*} The choice of $\delta=1/n$ gives the result.
\end{proof}
\begin{lemma}\label{lem:rem4}
    Suppose that Assumption~\ref{asmp:2},\ref{asmp:3.a}, and \ref{asmp:4} holds with $q_x\geq4$. Then, there exists a universal constant $C$ such that for any $\delta\in(0,1)$,
\begin{equation*}
    \Pb\left(\Sc_4 \leq \frac{\Dc_\Sigma\norm{\hat\beta-\beta}_\Sigma}{\lambda_{\rm min}(\Sigma^{-1/2}\hat\Sigma\Sigma^{-1/2})}\left(K_x^4+\frac{16K_x^4d\log (2d)}{\delta\sqrt{n}}\right)\right)\geq 1-\delta.
\end{equation*}
\end{lemma}
\begin{proof}
It follows from the definition of $\Sc_4$ that
\begin{eqnarray}\label{eq:revisit}
    \Sc_4&\leq&\norm{c^\top(\hat\Sigma^{-1}-\Sigma^{-1})\Sigma^{1/2}}_2\norm{\hat\beta-\beta}_\Sigma\Norm{\frac{1}{n}\sum_{i=1}^n\Sigma^{-1/2}X_iX_i^\top\Sigma^{-1/2}\frac{\norm{X_i}^2_{\Sigma^{-1}}}{d}}_{\rm op}\nonumber\\
    &\leq&\frac{\Dc_\Sigma}{\lambda_{\rm min}(\Sigma^{-1/2}\hat\Sigma\Sigma^{-1/2})}\norm{\hat\beta-\beta}_\Sigma\Norm{\frac{1}{n}\sum_{i=1}^n\Sigma^{-1/2}X_iX_i^\top\Sigma^{-1/2}\frac{\norm{X_i}^2_{\Sigma^{-1}}}{d}}_{\rm op}.
\end{eqnarray}
Hence, we focus on the last term in \eqref{eq:revisit}. We note that
\begin{eqnarray*}
    &&\Norm{\frac{1}{n}\sum_{i=1}^n\Sigma^{-1/2}X_iX_i^\top\Sigma^{-1/2}\frac{\norm{X_i}^2_{\Sigma^{-1}}}{d}}_{\rm op}\\
    &\leq& \Norm{\Eb\Sigma^{-1/2}X_1X_1^\top\Sigma^{-1/2}\frac{\norm{X_1}^2_{\Sigma^{-1}}}{d}}_{\rm op}+\Norm{\frac{1}{n}\sum_{i=1}^n\Sigma^{-1/2}X_iX_i^\top\Sigma^{-1/2}\frac{\norm{X_i}^2_{\Sigma^{-1}}}{d}-\Eb\Sigma^{-1/2}X_1X_1^\top\Sigma^{-1/2}\frac{\norm{X_1}^2_{\Sigma^{-1}}}{d}}_{\rm op}.
\end{eqnarray*} The first part of the right-hand side can be bounded as
\begin{eqnarray*}
    \Norm{\Eb\Sigma^{-1/2}X_1X_1^\top\Sigma^{-1/2}\frac{\norm{X_1}^2_{\Sigma^{-1}}}{d}}_{\rm op}&=&\sup_{\theta\in\Sb^{d-1}}\Eb\left[(\theta^\top\Sigma^{-1/2}X_1)^2\frac{\norm{X_1}^2_{\Sigma^{-1}}}{d}\right]\\
    &\leq&\sup_{\theta\in\Sb^{d-1}}\Eb\left[(\theta^\top\Sigma^{-1/2}X_1)^4\right]\leq K_x^4.
\end{eqnarray*} To control the second part, we let
\begin{equation*}
    Z_i=\Sigma^{-1/2}X_iX_i^\top\Sigma^{-1/2}\frac{\norm{X_i}^2_{\Sigma^{-1}}}{d},\quad\mbox{for}\quad i=1,\ldots,n.
\end{equation*} An application of Theorem 5.1(2) of \cite{tropp2016expected} gives that
\begin{equation}\label{eq:bound for the operator norm}
    \left(\Eb\Norm{\frac{1}{n}\sum_{i=1}^nZ_i-\Eb Z_i}_{\rm op}^2\right)^{1/2}\leq C_d^{1/2}\Norm{\Eb\left[\left(\frac{1}{n}\sum_{i=1}^nZ_i-\Eb Z_i\right)^2\right]}_{\rm op}^{1/2}+\frac{C_d}{n}\left(\Eb\left[\max_{1\leq i\leq n}\Norm{Z_i-\Eb Z_i}_{\rm op}^2\right]\right)^{1/2},
\end{equation}where $C_d=2(1+4\lceil\log d\rceil)$. The leading term on \eqref{eq:bound for the operator norm} can be bounded as
\begin{align*}
    \Norm{\Eb\left[\left(\frac{1}{n}\sum_{i=1}^nZ_i-\Eb Z_i\right)^2\right]}_{\rm op}&=\frac{1}{\sqrt{n}}\Norm{\Eb[Z_1^2]-(\Eb Z_1)^2}_{\rm op}\\
    &\leq\frac{1}{\sqrt{n}}\Norm{\Eb[Z_1^2]}_{\rm op}\quad \because\Eb Z_1\mbox{ is symmetric and positive definite.}\\
    &\leq\frac{1}{n}\left(\sup_{\theta\in\Sb^{d-1}}\Eb\left[(\theta^\top\Sigma^{-1/2}X_i)^2\frac{\norm{X_1}^6_{\Sigma^{-1}}}{d^2}\right]\right)\\
    &\leq \frac{d}{n}\sup_{\theta\in\Sb^{d-1}}\Eb[(\theta^\top\Sigma^{-1/2}X_i)^8]\leq\frac{K_x^8d}{n}.
\end{align*} Meanwhile, the second part is controlled as
\begin{align*}
    \Eb\left[\max_{i\in[n]}\Norm{Z_i-\Eb Z_i}_{\rm op}^2\right]&\leq \Eb\left[\max_{i\in[n]}\Norm{Z_i}_{\rm op}^2\right]\quad\because\Eb Z_1\mbox{ is positive definite.}\\
    &\leq\Eb\left[\sum_{i=1}^n\Norm{Z_i}_{\rm op}^2\right]=n\Eb\Norm{Z_1}_{\rm op}^2\\
    &=n\Eb\left[\frac{\norm{X_1}^8_{\Sigma^{-1}}}{d^2}\right]\\
    &\leq nd^2\sup_{\theta\in\Sb^{d-1}}\Eb[(\theta^\top\Sigma^{-1/2}X_i)^8]\leq K_x^8nd^2.
\end{align*} Combining these with \eqref{eq:bound for the operator norm}, we get
\begin{eqnarray*}
    \left(\Eb\Norm{\frac{1}{n}\sum_{i=1}^nZ_i-\Eb Z_i}_{\rm op}^2\right)^{1/2}&\leq& K_x^4(C_d^{1/2}\vee C_d)\left(\sqrt{\frac{d}{n}}+\frac{d}{\sqrt{n}}\right)\leq \frac{16K_x^4d\log (2d)}{\sqrt{n}}.
\end{eqnarray*} Consequently, Chebyshev's inequality yields
\begin{equation*}
    \Pb\left(\Norm{\frac{1}{n}\sum_{i=1}^nZ_i-\Eb Z_i}_{\rm op}\geq t\right)\leq \left(\frac{16K_x^4d\log (2d)}{t\sqrt{n}}\right)^2.
\end{equation*}
Combining all, it follows that
\begin{equation*}
    \Pb\left(\Norm{\frac{1}{n}\sum_{i=1}^n\Sigma^{-1/2}X_iX_i^\top\Sigma^{-1/2}\frac{\norm{X_i}^2_{\Sigma^{-1}}}{d}}_{\rm op}\geq K_x^4+\frac{16K_x^4d\log (2d)}{\delta\sqrt{n}}\right)\leq \delta,
\end{equation*}for any $\delta\in(0,1)$.
\end{proof}

\subsection{Concentration Inequalities for the Sample Gram Matrix}\label{sec:conc_ineq}
As an independent branch of research, estimating the covariance matrix $\Sigma$ in multidimensional distributions is a longstanding problem in statistics. Recent developments in this area have focused on various distributional assumptions, including log-concave distributions \citep{rudelson1999random, adamczak2010quantitative}, sub-Gaussian distributions \cite{koltchinskii2017concentration, vershynin2018high}, and distributions with finite moments \citep{vershynin2012close, mendelson2012generic, srivastava2013covariance, mendelson2014singular}.

In our application, we primarily utilize the standard sample covariance matrix $\hat\Sigma$ with the standard operator norm $\norm{\cdot}_{\rm op}$ as a measure. Notably, some papers have explored truncated estimates, as seen in \citep{abdalla2023covariance}. Moreover, alternative studies have considered using the Frobenius norm as a metric \citep{puchkin2023sharper}.

\begin{proposition}\label{prop:6}Recall that $\Dc_\Sigma=\norm{\Sigma^{-1/2}\hat\Sigma\Sigma^{-1/2}-I_d}_{\rm op}$. The following concentration inequalities for $\Dc_\Sigma$ hold under different moment assumptions on covariates.
\begin{enumerate}
    \item\label{lem:conc_D_sigma_1} Under Assumption~\ref{asmp:3.b}, there exists a universal constant $C>0$ such that
    \begin{equation}\label{eq:conc_D_sigma_1}
        \Pb\left(\Dc_\Sigma\leq C K_x^2\left[\sqrt{\frac{d+\log(1/\delta)}{n}}+\frac{d+\log(1/\delta)}{n}\right]\right)\geq 1-\delta,
    \end{equation} for all $\delta\in(0,1)$.
    \item\label{lem:conc_D_sigma_2} Under Assumption~\ref{asmp:3.a} with $q_x\geq4$, there exists a universal constant $C>0$ such that
    \begin{equation}\label{eq:conc_D_sigma_2}
        \Pb\left(\Dc_\Sigma\leq CK_x^2\left[\sqrt{\frac{d\log(2d/\delta)}{n}}+\frac{d\log^{3/2}(2d/\delta)}{\delta^{2/q_x}n^{1-2/q_x}}\right]\right)\geq 1-\delta,
    \end{equation} for all $\delta\in(0,1)$. 
    \item\label{lem:conc_D_sigma_3} Under Assumption~\ref{asmp:3.c} with $q_x>4$, there exists a constant $C>0$ that only depends on $q_x$ that
    \begin{equation}\label{eq:conc_D_sigma_3}
        \Pb\left(\Dc_\Sigma\leq CK_x^2\left[\sqrt{\frac{d\log(2d/\delta)}{n}}+\frac{d^{2/q_x}\log^{3/2}(2d/\delta)}{\delta^{2/q_x}n^{1-2/q_x}}\right]\right)\geq 1-\delta,
    \end{equation} for all $\delta\in(0,1)$. 
\end{enumerate}
\end{proposition}
\paragraph{proof of \eqref{eq:conc_D_sigma_1}} 
    The result under Assumption~\ref{asmp:3.b}, which pertains to sub-Gaussianity, is well-known. Comprehensive proofs and discussions can be found in existing references such as Theorem 4.7.1 of \cite{vershynin2018high} or Theorem~\ref{thm:1} of \cite{koltchinskii2017concentration}. 
\paragraph{proof of \eqref{eq:conc_D_sigma_2}}
We use Theorem~2.7 of \cite{brailovskaya2023universality}. To state their results, some matrix parameters deserve to be defined. We put
    \begin{eqnarray}\label{eq:matrix_params}
        \sigma_\star^2&:=&\Norm{\Eb\left[\left(\Sigma^{-1/2}\hat\Sigma\Sigma^{-1/2}-I_d\right)^2\right]}_{\rm op},\quad
        \sigma_\circ^2:=\sup_{u,v\in\Sb^{d-1}}\Eb\left[\Abs{u^\top\left(\hat\Sigma-\Eb\hat\Sigma\right)v}^2\right],\nonumber\\
        Z_n &:=& \frac{1}{n}\max_{1\leq i\leq n}\Norm{\Sigma^{-1/2}X_iX_i^\top\Sigma^{-1/2}-I_d}_{\rm op},\quad
        \bar R^2:=\Eb[Z_n^2].
    \end{eqnarray} The existences of $\sigma_\circ^2$, $\sigma_\star^2$, and $\bar R^2$ are guaranteed as long as $q_x\geq4$.
    
    Define a $d\times d$ random matrix $G$ with jointly Gaussian entries such that $\Eb G = I_d$ and ${\rm Cov}(G)= {\rm Cov}(\Sigma^{-1/2}\hat\Sigma\Sigma^{-1/2})$. More precisely, we define the $d^2\times d^2$ entry covariance matrix of $Y\in\Real^{d\times d}$ as
    \begin{equation}\label{eq:cov_of_mat}
        {\rm Cov}(Y)_{ij,kl}=\Eb[(Y-\Eb Y)_{ij}(Y-\Eb Y)_{kl}],\quad 1\leq i,j,k,l\leq d.
    \end{equation} Since jointly Gaussian random variables can be characterized with their first and second moments, $\Eb G$ and ${\rm Cov}(G)$ uniquely define the distribution of $G$ (See, e.g., Section 2.1 of \cite{brailovskaya2023universality}). We denote the spectrum of the operator $M$, that is, a collection of eigenvalues, as  ${\rm sp}(M)$. Let $d_H(A,B)$ be a Hausdorff distance between two sets in $\Real$, which is defined as $d_{H}(A,B)=\inf\set{\epsilon>0:A\subset B+[-\epsilon,\epsilon],\,B\subset A+[-\epsilon,\epsilon]}$. We are ready to state Theorem~2.7 of \cite{brailovskaya2023universality}; there exists a universal constant $C>0$ such that
    \begin{equation}\label{eq:universal_ineq_Brail}
        \Pb\left(d_H\left({\rm sp}(\Sigma^{-1/2}\hat\Sigma\Sigma^{-1/2}),{\rm sp}(G)\right)\geq C\left(\sigma_\circ t^{1/2} + R^{1/3}\sigma_\star^{2/3}t^{2/3}+R\right)\mbox{ and }Z_n\leq R\right)\leq de^{-t},
    \end{equation} for all $t\geq0$ and $R\geq\bar R^{1/2}\sigma_\star^{1/2}+\sqrt{2}\bar R$. This result readily controls the differences between extreme eigenvalues of $\Sigma^{-1/2}\hat\Sigma\Sigma^{-1/2}$ and $G$. In particular, we observe
    \begin{equation*}
        \max\set{\abs{\lambda_{\rm max}(\Sigma^{-1/2}\hat\Sigma\Sigma^{-1/2})-\lambda_{\rm max}(G)},\abs{\lambda_{\rm min}(\Sigma^{-1/2}\hat\Sigma\Sigma^{-1/2})-\lambda_{\rm min}(G)}}\leq d_H\left({\rm sp}(\Sigma^{-1/2}\hat\Sigma\Sigma^{-1/2}),{\rm sp}(G)\right).
    \end{equation*} Since $\norm{M-I}_{\rm op}=\max\{\abs{\lambda_{\rm max}(M)-1},\abs{\lambda_{\rm min}(M)-1}\}$, the triangle inequality gives
    \begin{eqnarray*}
        \Dc_\Sigma \leq \norm{G-I}_{\rm op} + d_H\left({\rm sp}(\Sigma^{-1/2}\hat\Sigma\Sigma^{-1/2}),{\rm sp}(G)\right).
    \end{eqnarray*} Since $G$ has jointly Gaussian entries, its spectral properties are relatively well-known, and will be analyzed in Proposition~\ref{prop:tail_bd_gaussian}.

    Now, we will have a closer look at \eqref{eq:universal_ineq_Brail} by analyzing the parameters $\sigma_\star^2,\sigma_\circ^2$ and $\bar R$. For $\sigma_\star^2$, we note that 
    \begin{eqnarray*}
        \sigma_\star^2&=&\sup_{\theta\in \Sb^{d-1}}\Abs{\Eb\left[\theta^\top\left(\frac{1}{n}\sum_{i=1}^n(\Sigma^{-1/2}X_iX_i^\top\Sigma^{-1/2}-I)\right)^2\theta\right]}=\sup_{\theta\in \Sb^{d-1}}\Abs{\frac{1}{n^2}\sum_{i=1}^n\Eb\left[\theta^\top(\Sigma^{-1/2}X_iX_i^\top\Sigma^{-1/2}-I)^2\theta\right]}\\
        &=&\frac{1}{n}\sup_{\theta\in \Sb^{d-1}}\Abs{\Eb\left[\theta^\top(\Sigma^{-1/2}XX^\top\Sigma^{-1/2}-I)^2\theta\right]}=\frac{1}{n}\sup_{\theta\in \Sb^{d-1}}\Abs{\Eb\left[(\theta^\top \Sigma^{-1/2}X)^2\norm{X}_{\Sigma^{-1}}^2\right]-1}\\
        &\leq&\frac{K_x^4d}{n}.
    \end{eqnarray*} Meanwhile, $\sigma_\circ^2$ can be controlled as
    \begin{eqnarray*}
        \sigma_\circ^2&=&\sup_{\theta\in \Sb^{d-1}}\Eb\left(\theta^\top\left(\frac{1}{n}\sum_{i=1}^n(\Sigma^{-1/2}X_iX_i^\top\Sigma^{-1/2}-I)\right)\theta\right)^2=\sup_{\theta\in \Sb^{d-1}}\Eb\left(\frac{1}{n}\sum_{i=1}^n[(\theta^\top\Sigma^{-1/2}X_i)^2-1]\right)^2\\
        &=&\frac{1}{n^2}\sup_{\theta\in \Sb^{d-1}}\sum_{i=1}^n\Eb\left[((\theta^\top \Sigma^{-1/2}X_i)^2-1)^2\right]=\frac{1}{n}\sup_{\theta\in \Sb^{d-1}}\Abs{\Eb\left[(\theta^\top \Sigma^{-1/2}X)^4\right]-1}\\
        &\leq&\frac{K_x^4}{n}.
    \end{eqnarray*} To analyze $\bar R^2$, we note that $\norm{\Sigma^{-1/2}X_iX_i^\top\Sigma^{-1/2}-I_d}_{\rm op}=\abs{\norm{X_i}_{\Sigma^{-1}}^2-1}\leq \norm{X_i}_{\Sigma^{-1}}^2 + 1$. This gives that
    \begin{eqnarray*}
        \bar R^2 &\leq& \frac{1}{n^2}\Eb\left[\max_{1\leq i\leq n}\norm{X_i}_{\Sigma^{-1}}^4 + 1\right]\leq \frac{1}{n^2}+\frac{1}{n^2}\left(\Eb\left[\max_{1\leq i\leq n}\norm{X_i}_{\Sigma^{-1}}^{q_x}\right]\right)^{4/q_x}\\
        &\leq&\frac{1+K_x^4d^2}{n^{2-4/q_x}}\leq\frac{2K_x^4d^2}{n^{2-4/q_x}}.
    \end{eqnarray*}

    We let
    \begin{eqnarray*}
        R(\delta)=K_x^2\left[\frac{2d}{\delta^{2/q_x}n^{1-2/q_x}}+\left(\frac{\sqrt{2}d}{n^{1-2/q_x}}\right)^{1/2}\left(\frac{d}{n}\right)^{1/4}\right],
    \end{eqnarray*} for $\delta\in(0,1)$, so that $R(\delta)\geq \bar R^{1/2}\sigma_\star^{1/2}+\sqrt{2}\bar R$. In addition, we observe that
    \begin{eqnarray}\label{eq:union_bd_Z_n}
        \Pb\left(Z_n\geq R(\delta)\right)&=&\Pb\left(\max_{1\leq i\leq n}\Abs{\norm{X_i}_{\Sigma^{-1}}^2-1}\geq nR(\delta)\right)\nonumber\\
        &\leq& n\Pb\left(\Abs{\norm{X}_{\Sigma^{-1}}^2-1}\geq nR(\delta)\right)\nonumber\\
        &\leq& \frac{n\Eb\left[\Abs{\norm{X}_{\Sigma^{-1}}^2-1}^{q_x/2}\right]}{[nR(\delta)]^{q_x/2}}\nonumber\\
        &\leq&\frac{2K_x^{q_x}d^{q_x/2}}{n^{q_x/2-1}[R(\delta)]^{q_x/2}}\leq\frac{\delta}{2}.
    \end{eqnarray} On the other hand, since Young's inequality yields $3R(\delta)^{1/3}\sigma_\star^{2/3}t^{2/3}\leq 2\sigma_\star t^{1/2}+R(\delta)t$, we have
    \begin{eqnarray*}
        \epsilon_{R(\delta)}(t) = \sigma_\circ t^{1/2} + R(\delta)^{1/3}\sigma_\star^{2/3}t^{2/3}+R(\delta)t\leq C\left(\sigma_\star t^{1/2}+R(\delta)t\right)\leq CK_x^2\left(\sqrt{\frac{dt}{n}}+\frac{dt^{3/2}}{\delta^{2/q_x}n^{1-2/q_x}}\right).
    \end{eqnarray*} Here, we used the fact that $\sigma_\circ\leq \sigma_\star$. Hence, the inequality \eqref{eq:universal_ineq_Brail} reads
    \begin{equation*}
        \Pb\left(d_H\left({\rm sp}(\hat\Sigma),{\rm sp}(G)\right)\geq CK_x^2\left(\sqrt{\frac{dt}{n}}+\frac{dt^{3/2}}{\delta^{2/q_x}n^{1-2/q_x}}\right)\mbox{ and }Z_n\leq R(\delta)\right)\leq de^{-t}.
    \end{equation*} Taking $t=\log(2d/\delta)$ and combining this with \eqref{eq:union_bd_Z_n} implies that
    \begin{equation*}
        \Pb\left(d_H\left({\rm sp}(\hat\Sigma),{\rm sp}(G)\right)\geq CK_x^2\left(\sqrt{\frac{d\log(2d/\delta)}{n}}+\frac{d\log^{3/2}(2d/\delta)}{\delta^{2/q_x}n^{1-2/q_x}}\right)\right)\leq \delta.
    \end{equation*} Finally, we note that
    \begin{equation*}
        \Dc_\Sigma\leq d_H\left({\rm sp}(\hat\Sigma),{\rm sp}(G)\right)+\norm{G-I_d}_{\rm op}.
    \end{equation*} The tail bound of $\norm{G-I_d}_{\rm op}$, which is given in Proposition~\ref{prop:tail_bd_gaussian} completes the proof.

\paragraph{Proof of \eqref{eq:conc_D_sigma_3}} Under Assumption~\ref{asmp:3.c}, we prove that the random variable $Z_n$ and its second moment $\bar R^2$ (both defined in \eqref{eq:matrix_params}) can be better controlled. Recall that $Z_n=\max_{1\leq i\leq n}\Abs{\norm{X_i}_{\Sigma^{-1}}^2-1}$. Applying equation (1.9) of \cite{rio2017constants} or equation (E.16) of \cite{yang2021finite} gives that if $q_x\geq 4$, then
\begin{equation*}
    \Pb\left(Z_n\geq\frac{d+1}{n}+\frac{K_x^2\sqrt{d\log(2n/\delta)}}{n}+C_{q_x}K_x^2\frac{d^{2/q_x}}{\delta^{2/q_x}n^{1-2/q_x}}\right)\leq \delta,
\end{equation*} for $\delta\in(0,1)$ where $C_{q_x}=1+4/q_x+q_x/6$. Let the right-hand side inside the probability as $\tilde R(\delta)$, so that
\begin{equation*}
    \tilde R(\delta)^2\leq CK_x^2\left[\frac{d^2+d\log(2n/\delta)}{n^2}+\frac{d^{4/q_x}}{\delta^{4/q_x}n^{2-4/q_x}}\right],
\end{equation*} holds for some constant $C=C(q_x)>0$. We note that with some possibly different constants,
\begin{eqnarray*}
    \Eb[Z_n^2]&=&\int_0^\infty\Pb(Z_n^2\geq t)\,dt=\left(\int_0^{\tilde R(1)^2}+\int_{\tilde R(1)^2}^{\tilde R(0)^2}\right)\Pb(Z_n^2\geq t)\,dt\\
    &\leq& \tilde R(1)^2 - \int_0^1 \Pb(Z_n\geq\tilde R(\delta)) d(\tilde R(\delta)^2)\\
    &\leq&\tilde R(1)^2 - \int_0^1 \delta d(\tilde R(\delta)^2)\\
    &\leq& \tilde R(1)^2 +CK_x^2\left[ \int_0^1 \delta \left(\frac{d}{n^2\delta}+\frac{d^{4/q_x}}{\delta^{1+4/q_x}n^{2-4/q_x}}\right)d\delta\right]\\
    &\leq& CK_x^2\left[\frac{d^2+d\log(2n)}{n^2}+\frac{d^{4/q_x}}{n^{2-4/q_x}}\right].
\end{eqnarray*} For the last inequality, the integral $\int_0^11/\delta^{4/q_x}\,d\delta$ needs to be finite, so the condition $q_x>4$ is required. Consequently, we get
\begin{equation*}
    \bar R\leq CK_x^2\left[\frac{d}{n}+\frac{\sqrt{d\log(2n)}}{n}+\frac{d^{2/q_x}}{n^{1-2/q_x}}\right],
\end{equation*} for some constant $C=C(q_x)>0$. To apply \eqref{eq:universal_ineq_Brail}, we let 
\begin{equation*}
    R(\delta) = CK_x^2\left[\sqrt{\frac{d}{n}}+\frac{d^{2/q_x}}{\delta^{2/q_x}n^{1-2/q_x}}\right],
\end{equation*} for some constant $C=C(q_x)$ such that $R(\delta)\geq\bar R^{1/2}\sigma_\star^{1/2}+\sqrt{2}\bar R$. The rest of the proof is analogous to the proof of \eqref{eq:conc_D_sigma_1}.

\begin{remark}
    The concentration of the sample Gram matrix requires control of the both upper tail and the lower tail. Under a moment assumption of Assumption~\ref{asmp:3.a}, not only the current bound in Proposition~\ref{prop:6} but also the existing bounds in recent literature \citep{vershynin2012close,mendelson2012generic, srivastava2013covariance,mendelson2014singular,guedon2017interval,tikhomirov2018sample} requires an additional term compared to the optimal rate achieved under sub-Gaussian covariates (Assumption~\ref{asmp:3.b}). Interestingly, this discrepancy is responsible for the upper tail behavior of $\hat\Sigma$ as \cite{oliveira2016lower} showed that a lower tail is defacto sub-Gaussian under a simple fourth-moment assumption, which is presented in Proposition~\ref{prop:oliveira}. The asymmetry between upper and lower tails may be of independent interest.

\end{remark}

\paragraph{Comparison to Existing Inequalities}
We present Proposition~\ref{prop:tik}, regarded as the sharpest concentration inequalities available for the sample Gram matrix. These inequalities are contingent solely on finite-moment assumptions. This proposition essentially mirrors Lemma 3 of \cite{yang2021finite} whose proof predominantly relies on the groundbreaking result of Theorem 1.1 of \cite{tikhomirov2018sample}.
\begin{proposition}\label{prop:tik}
\begin{itemize} Let $n\geq 2d$.
    \item[(i)] Suppose that Assumption~\ref{asmp:3.a} holds with $q_x\geq 2$. Then,
    \begin{equation}\label{eq:prop:tik.1}
        \Pb\left(\Dc_\Sigma\leq K_x^2C_{q_x}\left[\frac{d}{n^{1-2/q_x}\delta^{2/q_x}}+\left(\frac{d}{n}\right)^{1-2/q_x}\log^4\left(\frac{n}{d}\right)+\left(\frac{d}{n}\right)^{1-2/\min\{q_x,4\}}\right]\right)\geq 1-\frac{1}{d}-\delta,
    \end{equation} for all $\delta\in(0,1)$.
    \item[(ii)] Suppose that Assumption~\ref{asmp:3.c} holds with $q_x\geq 2$. Then,
    \begin{equation}\label{eq:prop:tik.2}
        \Pb\left(\Dc_\Sigma\leq K_x^2C_{q_x}\left[\frac{\sqrt{d\log(n/\delta)}}{n}+\frac{d^{2/q_x}}{n^{1-2/q_x}\delta^{q_x/2}}+\left(\frac{d}{n}\right)^{1-2/q_x}\log^4\left(\frac{n}{d}\right)+\left(\frac{d}{n}\right)^{1-2/\min\{q_x,4\}}\right]\right)\geq 1-\frac{1}{d}-\delta,
    \end{equation} for all $\delta\in(0,1)$.
\end{itemize}
    
\end{proposition}
\begin{proof}[proof of Proposition~\ref{prop:tik}] 
    One can refer to Lemma 3 of \cite{yang2021finite}, albeit contains typographical errors\footnote{
The authors have made an error in their application of Theorem~\ref{thm:1} of \cite{tikhomirov2018sample}, wherein they used $1/n$ instead of $1/d$.}.
\end{proof}

Preference might lean towards the results from Proposition~\ref{prop:tik} over those from Proposition~\ref{prop:6}, mainly due to at least two reasons: (1) Proposition~\ref{prop:tik} allowing for $q_x\geq2$, and (2) the suboptimal logarithmic factors in the bounds in Proposition~\ref{prop:6}. However, our preference for our own bounds in Proposition~\ref{prop:6} stems from the intention to avoid dependence on the $1/d$ factor, as evident in Proposition~\ref{prop:tik}. Although employing Proposition~\ref{prop:tik} might yield sharper rates concerning logarithmic factors in the results in the main article, it necessitates consideration of a `large' $d$ regime to ensure such concentration of the sample gram matrix holds with a `high' probability. Nevertheless, for the sake of accommodating diverse combinations of $(n,d)$, we opt for Proposition~\ref{prop:6} in proving the main theorem, albeit with the inclusion of seemingly spurious logarithmic factors.

\subsection{Useful Propositions}

\begin{proposition}\label{prop:oliveira}[The lower tail of the Sample Gram Matrix] Suppose that Assumption~\ref{asmp:3.a} holds for $q_x\geq4$. Then,
\begin{equation*}
    \Pb\left(\lambda_{\rm min}(\Sigma^{-1/2}\hat\Sigma\Sigma^{-1/2})\geq 1-9K_x\sqrt{\frac{d+2\log(2/\delta)}{n}}\right)\leq \delta.
\end{equation*} If $d+2\log(2n)\leq n/(18K_x)^2$, then the choice of $\delta=1/n$ leads $\Pb(\lambda_{\rm min}(\Sigma^{-1/2}\hat\Sigma\Sigma^{-1/2})\geq 1/2)\leq 1/n$.
\end{proposition}
\begin{proof}[proof of Proposition~\ref{prop:oliveira}] See Theorem~4.1 of \cite{oliveira2016lower}.
\end{proof}
\begin{proposition}\label{prop:tail_bd_gaussian} Suppose that Assumption~\ref{asmp:3.a} holds for $q_x\geq 4$. Let a $d\times d$ matrix $G$ with jointly Gaussian entries be such that $\Eb(G)=I_d$ and ${\rm Cov}(G)={\rm Cov}(\Sigma^{-1/2}\hat\Sigma\Sigma^{-1/2})$\footnote{See \eqref{eq:cov_of_mat} for the definition of the variance-covariance matrix of matrix.}. Then,
    \begin{equation*}
        \Pb\left(\norm{G-I_d}_{\rm op}\geq K_x^2\sqrt{\frac{2d\log(2d/\delta)}{n}}\right)\leq \delta,
    \end{equation*} for all $\delta\in(0,1)$. Furthermore, $\Eb\norm{G-I_d}_{\rm op}\leq K_x^2\sqrt{\dfrac{2d\log(2d)}{n}}$.
\end{proposition}
\begin{proof}[proof of Proposition~\ref{prop:tail_bd_gaussian}]
    Since the jointly Gaussian variables can be expressed as the linear combination of independent standard Gaussian variables, $G$ can be expressed as
    \begin{equation*}
        G = I_d + \sum_{b=1}^B A_bg_b,
    \end{equation*} for some deterministic symmetric matrices $A_1,\ldots,A_B$, and i.i.d. standard Gaussian variables $g_1,\ldots,g_B$. Therefore, Theorem 4.1.1 of \cite{tropp2015introduction} applies and yields
    \begin{equation}\label{eq:tropp2015}
        \Pb\left(\norm{G-I_d}_{\rm op}\geq t\right)\leq 2d\exp\left(-\frac{t^2}{2\sigma_\star^2}\right)\mbox{ for all }t\geq0,\quad\Eb\norm{G-I_d}_{\rm op}\leq\sqrt{2\sigma_\star^2\log(2d)},
    \end{equation} where $\sigma_\star^2=\norm{\Eb[(G-\Eb G)(G-\Eb G)^\top]}_{\rm op}$. Since ${\rm Cov}(G)={\rm Cov}(\hat\Sigma)$, the matrix parameter $\sigma_\star^2$ matches the one defined in \eqref{eq:matrix_params}. That is, $\sigma_\star^2=\norm{\Eb[(\Sigma^{-1/2}\hat\Sigma\Sigma^{1/2}-I_d)^2]}_{\rm op}$, and is shown to be bounded as
    \begin{equation*}
        \sigma_\star^2\leq\frac{K_x^4d}{n}.
    \end{equation*} Combining this with \eqref{eq:tropp2015} yields the desired result.
\end{proof}

\begin{proposition}\label{prop:7}
    Suppose that Assumption~\ref{asmp:2}, \ref{asmp:3.a}, and \ref{asmp:4} holds such that $s=(1/q_x+1/q)^{-1}>2$. Then, there exists a constant $C_s$ which only depends on $s$ such that
    \begin{eqnarray*}
        \Pb\left(\Norm{\frac{1}{n}\sum_{i=1}^n\Sigma^{-1/2}X_i(Y_i-X_i^\top\beta)}_2\geq 2\sqrt{\frac{d+\log(2/\delta)}{n\overline{\lambda}}}+C_sK_xK_y\frac{d^{1/2}}{\delta^{1/s}n^{1-1/s}}\right)\leq \delta,
    \end{eqnarray*} for all $\delta\in(0,1)$.
\end{proposition}
\begin{proof}[proof of Proposition~\ref{prop:7}]
This proposition has been demonstrated in \cite{kuchibhotla2020berry}, but we aim to furnish a self-contained proof for the sake of completeness and clarity. We let
\begin{equation*}
    W_i :=\Sigma^{-1/2}X_i(Y_i-X_i^\top\beta).
\end{equation*} Since $\beta$ is the projection parameter, we have $\Eb W_i=0_d$ for $1\leq i\leq n$. An application of Theorem 4 of \cite{einmahl2008characterization} with $\eta=\delta=1$ yields
\begin{eqnarray}\label{eq:conc_ineq_infl}
    \Pb\left(\Norm{\frac{1}{n}\sum_{i=1}^n W_i}_2\geq2\Eb\Norm{\frac{1}{n}\sum_{i=1}^nW_i}_2+t\right)\leq\exp\left(-\frac{n\overline{\lambda} t^2}{3}\right)+\frac{C_s}{(nt)^s}\sum_{i=1}^n\Eb\norm{W_i}_2^s,
\end{eqnarray} for some constant $C_s$ that only depends on $s$. Jensen's inequality yields
\begin{eqnarray*}
    \Eb\Norm{\frac{1}{n}\sum_{i=1}^nW_i}_2&\leq& \left(\Eb\Norm{\frac{1}{n}\sum_{i=1}^nW_i}_2^2\right)^{1/2}\\
    &=&\left[{\rm tr}\left\{{\rm Var}\left(\frac{1}{n}\sum_{i=1}^nW_i\right)\right\}\right]^{1/2}=\frac{1}{\sqrt{n}}\left\{{\rm tr}\left(\Sigma^{-1/2}V\Sigma^{-1/2}\right)\right\}^{1/2}\\
    &\leq& \sqrt{\frac{d}{\overline{\lambda}^{1/2}n}},
\end{eqnarray*} where the last inequality follows from Assumption~\ref{asmp:4}. Now, $\Eb\norm{W_i}_2^s$ is controlled as
\begin{eqnarray*}
    \Eb\norm{W_i}_2^s&\leq&\Eb\left[\left(\sum_{j=1}^d(e_j^\top\Sigma^{-1/2}X_i)^2(Y_i-X_i^\top\beta)^2\right)^{s/2}\right]\\
    &\leq& d^{s/2}\Eb\left[\frac{1}{d}\sum_{j=1}^d\Abs{(e_j^\top\Sigma^{-1/2}X_i)(Y_i-X_i^\top\beta)}^s\right]\quad \mbox{(Jensen's inequality)}\\
    &\leq& d^{s/2}\max_{j\in[d]}\Eb\Abs{(e_j^\top\Sigma^{-1/2}X_i)(Y_i-X_i^\top\beta)}^s\\
    &\leq& d^{s/2}\max_{j\in [d]}\left[\Eb(e_j^\top\Sigma^{-1/2}X_i)^{q_x}\right]^{s/q_x}\left[\Eb(Y_i-X_i^\top\beta)^{q}\right]^{s/q}\quad \mbox{(H\"older's inequality)}\\
    &\leq& d^{s/2}(K_xK_y)^s.    
\end{eqnarray*} The last inequality is due to Assumption~\ref{asmp:2} and \ref{asmp:3.a}. Taking
\begin{eqnarray*}
    t=\sqrt{\frac{3\log(2/\delta)}{n\overline{\lambda}}}+C_sK_xK_y\frac{d^{1/2}}{\delta^{1/s}n^{1-1/s}},
\end{eqnarray*} in \eqref{eq:conc_ineq_infl} yields the desired result.
\end{proof} 

\begin{proposition}\label{prop:11}
    Suppose that Assumption~\ref{asmp:2},\ref{asmp:3.a}, and \ref{asmp:4} holds with $q_x\geq4$ and $s=(1/q_x+1/q)^{-1}>2$. Then, there exists a constant $C_s>0$, which only depends on $s$, such that for any $\delta\in(0,1)$, with probability at least $1-\delta_1-\delta_2$,
    \begin{eqnarray*}
        \norm{\hat\beta-\beta}_{\Sigma}\leq\left(1-9K_x\sqrt{\frac{d+2\log(2/\delta_1)}{n}}\right)^{-1}_+\left[2\sqrt{\frac{d+\log(2/\delta_2)}{n\overline{\lambda}}}+C_sK_xK_y\frac{d^{1/2}}{\delta_2^{1/s}n^{1-1/s}}\right].
    \end{eqnarray*} Provided that $d+2\log(2n)\leq n/(18K_x)^2$, taking $\delta_1=1/(2n)$ and $\delta_2=1/(2n^{s/2-1})$ yields that
    \begin{equation*}
        \Pb\left(\norm{\hat\beta-\beta}_{\Sigma}\leq C\sqrt{\frac{d+\log(2n)}{n}}\right)\geq 1-\frac{1}{n^{\min\{1,s/2-1\}}},
    \end{equation*} for some constant $C=C(\overline{\lambda},q_x,q,K_x,K_y)$.
\end{proposition}
\begin{proof}[proof of Proposition~\ref{prop:11}]
    We follow the proof of Proposition~26 of \cite{kuchibhotla2020berry} or Theorem~1.2 of \cite{oliveira2016lower}. We note that $$\Sigma^{1/2}(\hat\beta-\beta)=\Sigma^{1/2}\hat\Sigma^{-1}\Sigma^{-1/2}\frac{1}{n}\sum_{i=1}^n\Sigma^{-1/2}X_i(Y_i-X_i^\top\beta).$$ This implies that
    \begin{eqnarray*}
        \norm{\hat\beta-\beta}_{\Sigma}&\leq&\Norm{\Sigma^{1/2}\hat\Sigma^{-1}\Sigma^{-1/2}}_{\rm op}\Norm{\frac{1}{n}\sum_{i=1}^n\Sigma^{-1/2}X_i(Y_i-X_i^\top\beta)}_2\\
        &=&\frac{1}{\lambda_{\rm min}(\Sigma^{-1/2}\hat\Sigma\Sigma^{-1/2})}\Norm{\frac{1}{n}\sum_{i=1}^n\Sigma^{-1/2}X_i(Y_i-X_i^\top\beta)}_2.
    \end{eqnarray*} The quantities on the right-hand side are controlled using Proposition~\ref{prop:oliveira} and Proposition~\ref{prop:7}, respectively. 
\end{proof}

\section{Auxilary Results}
\begin{lemma}\label{lem:aux:1}
    Under the data generating process described in Section~\ref{sec:5.2}, the projection parameter $\beta$ is given by $\beta = 3(1+2\rho^2)\norm{\theta}_2^2\theta + 6(1-\rho^2)\theta^{\odot 3}$.
\end{lemma}
\begin{proof}
    Let $X=(X(1),\ldots,X(d))^\top$, $Z=(Z(1),\ldots,Z(d))^\top$, and $W=(W(1),\ldots,W(d))^\top$, so that $X(j)=Z(j)W(j)$ for $1\leq j\leq d$. We note that $\Eb XX^\top =I_d$ since
    \begin{eqnarray*}
        \Eb X(1)^2 &=& \Eb Z(1)^2\Eb W(1)^2=1\quad\mbox{ and }\\
        \Eb X(1)X(2) &=& \Eb Z(1)\Eb Z(2)\Eb W(1)W(2)=0.
    \end{eqnarray*} Consequently, the projection parameter is $\beta = \Eb XY$. Write $\theta = (\theta_1,\ldots,\theta_d)^\top$, and the $i$th coordinate $\beta(i)$ of $\beta$ can be written as
    \begin{eqnarray*}
        \beta(i) &=& \Eb X(i)(X^\top\theta)^3 +\Eb X(i)\epsilon\\
        &=& \Eb X(i)\left(\sum_{j=1}^d\theta_jX(j)\right)^3,
    \end{eqnarray*} where the second equality follows from the independence of $\epsilon$ and $X$. Unfolding the last expression, we get
    \begin{eqnarray*}
        \beta(i) &=& \Eb\left[X(i)\left\{\sum_{j=1}^d\theta_j^3X(j)^3+3\sum_{j\neq k}\theta_j^2\theta_k X(j)^2X(k)^2+6\sum_{j<k<l}\theta_j\theta_k\theta_l X(j)X(k)X(l)\right\}\right].
    \end{eqnarray*} Each individual term on the right-hand side can be simply computed using the following moments;
    \begin{eqnarray*}
        \Eb X(1)^4 &=& \Eb Z(1)^4\Eb W(1)^4 =9,\\
        \Eb X(1)^3 X(2) &=& \Eb Z(1)^3\Eb Z(2)\Eb W(1)^3W(2)=0,\\
        \Eb X(1)^2 X(2)^2 &=& \Eb Z(1)^2\Eb Z(2)^2\Eb W(1)^2 W(2)^2 = 1+2\rho^2,\\
        \Eb X(1)^2X(2)X(3) &=& \Eb Z(1)^2\Eb Z(2)\Eb Z(3)\Eb W(1)^2W(2)W(3)=0,\\
        \Eb X(1)X(2)X(3)X(4) &=& \Eb Z(1)\Eb Z(2)\Eb Z(3)\Eb Z(4)\Eb W(1)W(2)W(3)W(4)=0.
    \end{eqnarray*}
    This leads to that
\begin{eqnarray*}
    \beta(i) = 9\theta_i^3 + 3(1+2\rho^2)\theta_i \sum_{j\neq i}\theta_j^2 = (6-6\rho^2)\theta_i^3 + 3(1+2\rho^2)\norm{\theta}_2^2\theta_i.
\end{eqnarray*} This completes the proof.
\end{proof}
\section{Additional Numerical Results}\label{sec:D}
In this section, we provide the simulation results in addition to those presented in Section~\ref{sec:5}. These encompass the implementation under various combinations of sample sizes and model parameters. Furthermore, we also compare the confidence interval obtained by the OLS (without bias correction) based on the resampling bootstrap and the wild bootstrap. Due to the computational limitation, the resampling bootstraps are only implemented for small sample size cases, specifically when $n\in\set{1000, 2000}$. 

\subsection{Additional Figures for Well-specified Model}

Here we adhere to the well-defined linear setting outlined in Section~\ref{sec:5.1}. Figures~\ref{fig:D.1.1} and \ref{fig:D.1.2} present the empirical coverages and the lengths of confidence intervals for the first coefficient of the projection parameter $\beta$ attained from various inferential methods. These results are based on a sample size of $n=1000$ and $2000$, respectively.

\begin{figure}[t]
     \centering
     \centering
     \begin{subfigure}[b]{\textwidth}
         \centering
         \includegraphics[width=\textwidth]{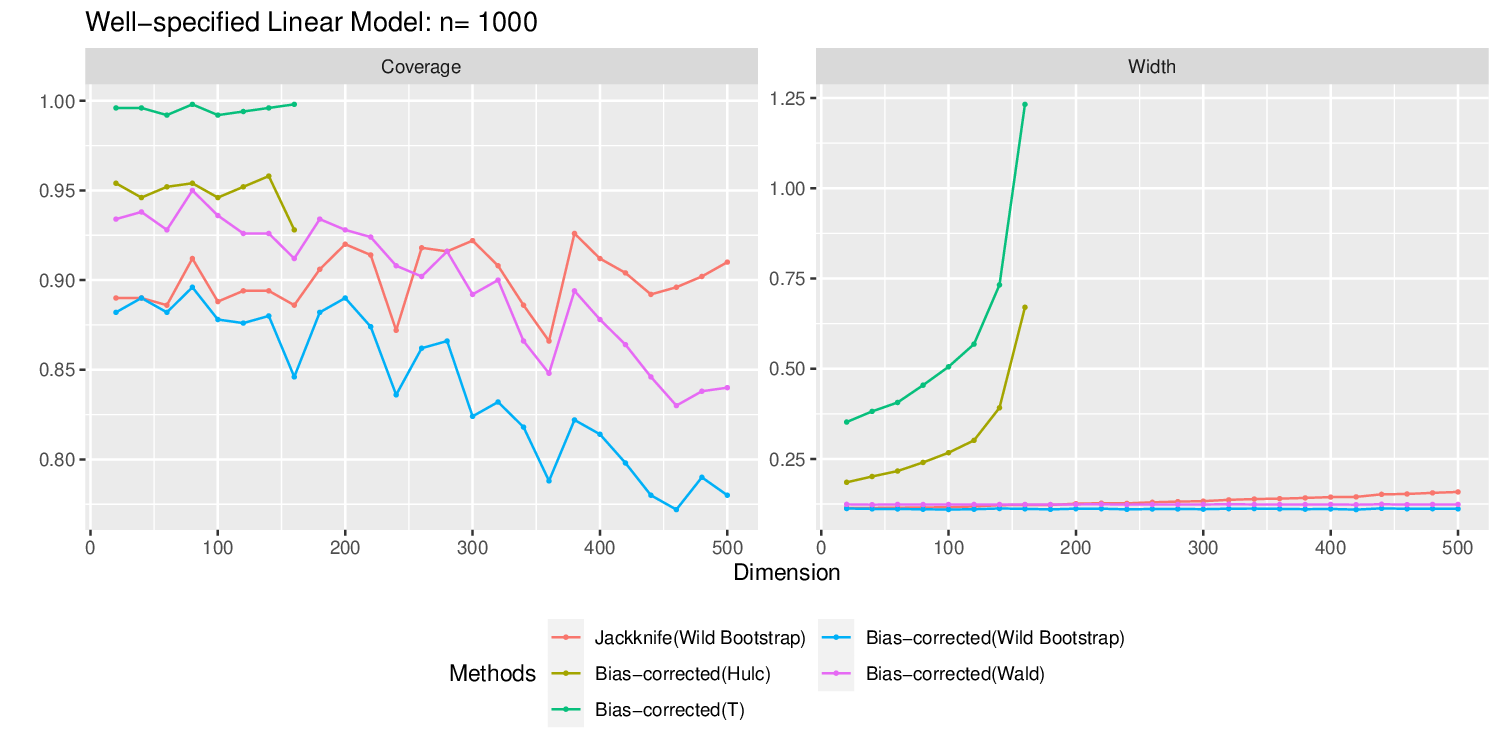}
         \caption{Proposed vs. Jackknife }
     \end{subfigure}\\
     \centering
     \begin{subfigure}[b]{\textwidth}
         \centering
         \includegraphics[width=\textwidth]{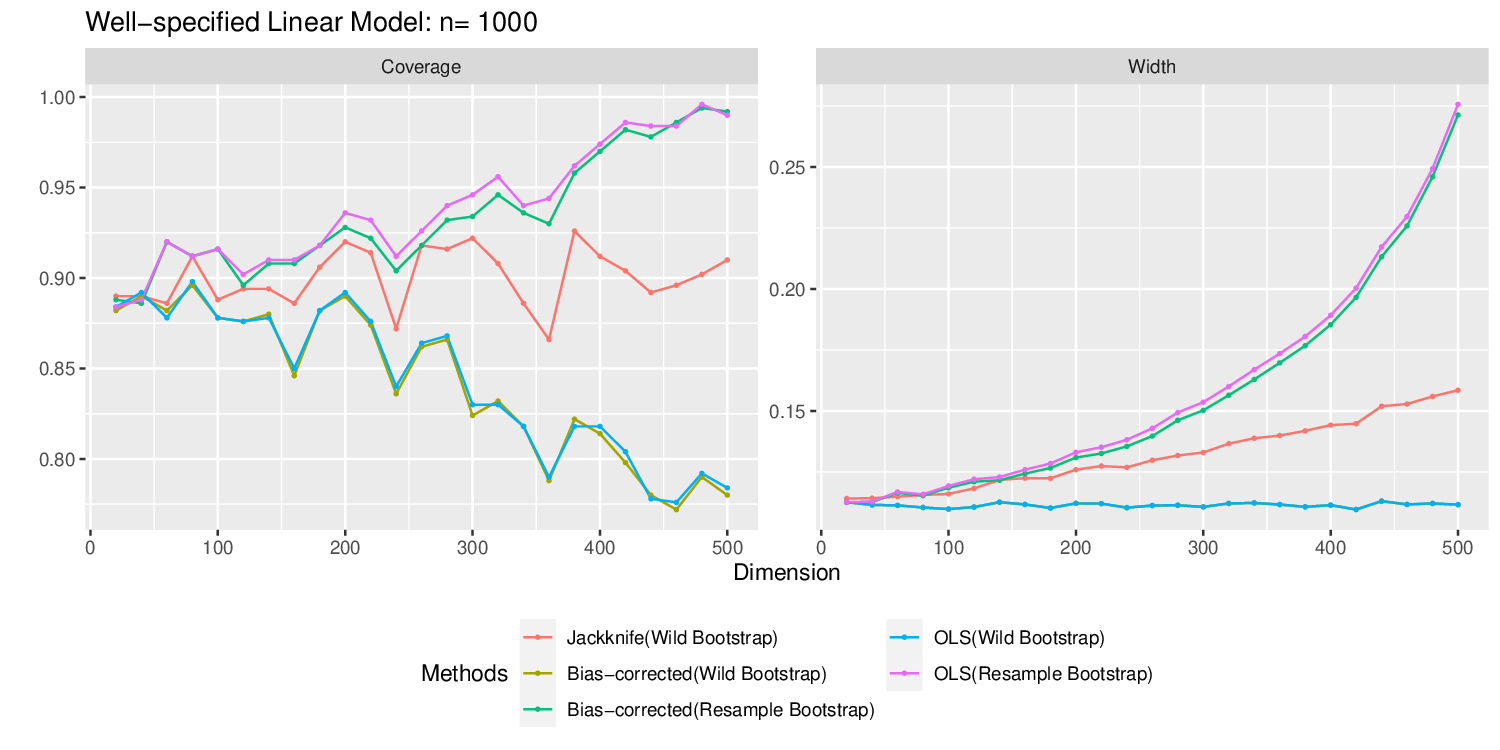}
         \caption{Comparison of Bootstrap-based inferential methods}
     \end{subfigure}
     \caption{Empirical coverages and widths of 95\% confidence intervals under a well-specified case with $n=1000$. \textbf{Top row}: Comparison of coverages and widths of confidence intervals; wild bootstrap based on Jackknife estimators, wild bootstrap based on proposed bias-corrected estimators, HulC using proposed bias-corrected estimators, and $t$-statistic based inference using proposed bias-corrected estimators, and Wald confidence interval based on proposed estimators. \textbf{Bottom row}: Comparison of coverages and widths of confidence intervals obtained from five different bootstrap methods; wild bootstrap CI based on Jackknife estimators, wild bootstrap CI based on proposed estimators, resample bootstrap CI based on proposed estimators, wild bootstrap CI based on OLS, resample bootstrap CI based on OLS. The empirical coverages and widths of CI are computed based on 1000 replications.}
     \label{fig:D.1.1}
\end{figure}

\begin{figure}[t]
     \centering
     \begin{subfigure}[b]{\textwidth}
         \centering
         \includegraphics[width=\textwidth]{Figures/CI_n=2000_simple.eps}
         \caption{Proposed vs. Jackknife}
     \end{subfigure}\\
     \centering
     \begin{subfigure}[b]{\textwidth}
         \centering
         \includegraphics[width=\textwidth]{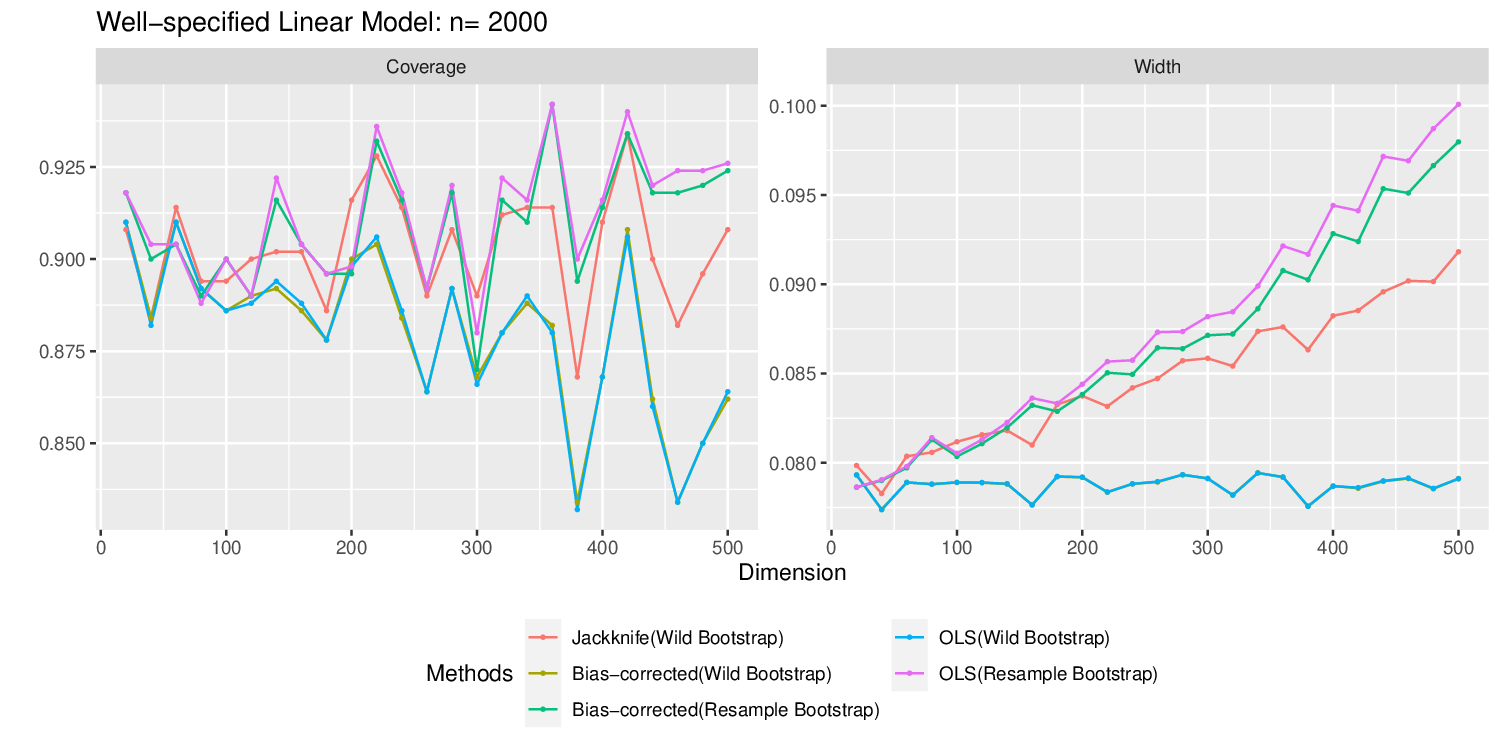}
         \caption{Comparison of Bootstrap-based inferential methods}
     \end{subfigure}
     \caption{Empirical coverages and widths of proposed confidence intervals under a well-specified case with $n=2000$. See Figure~\ref{fig:D.1.1} for the abbreviations of inferential methods.}
     \label{fig:D.1.2}
\end{figure}

\subsection{Additional Figures for Misspecified Model}
This section presents the additional numerical results on comparing confidence intervals for the projection parameter under the misspecified model. The simulation settings are summarized in Table~\ref{tab:1}.

\begin{table}[ht]
    \centering
    \begin{tabular}{|l|c|c|c|c|c|c|}
    \hline
    \multicolumn{1}{|c|}{$\theta$} & \multicolumn{3}{c|}{$\ev_1$} & \multicolumn{3}{c|}{$1_d^\top/\sqrt{d}$}\\
    \hline
    \multicolumn{1}{|c|}{$\rho$} & $0.0$ & $0.2$ & $0.5$ & $0.0$ & $0.2$ & $0.5$ \\
    \hline
    $n\in\{1000, 2000\}$ & Figure~\ref{fig:D.2.1.1},\ref{fig:D.2.1.2} &
    Figure~\ref{fig:D.2.2.1},\ref{fig:D.2.2.2} &
    Figure~\ref{fig:D.2.3.1},\ref{fig:D.2.3.2} &
    Figure~\ref{fig:D.2.4.1},\ref{fig:D.2.4.2} &
    Figure~\ref{fig:D.2.5.1},\ref{fig:D.2.5.2} &
    Figure~\ref{fig:D.2.6.1},\ref{fig:D.2.6.2} \\
    \hline
    $n\in\{5000, 10000, 20000\}$ & Figure~\ref{fig:D.2.7} & Figure~\ref{fig:D.2.8} & Figure~\ref{fig:D.2.9} & $\cdot$ & $\cdot$ & $\cdot$\\
    \hline
    \end{tabular}
    \caption{A table of figures corresponding to different combinations of $n$, $\theta$, and $\rho$}
    \label{tab:1}
\end{table}
\begin{figure}
     \centering
     \begin{subfigure}[b]{\textwidth}
         \centering
         \includegraphics[width=\textwidth]{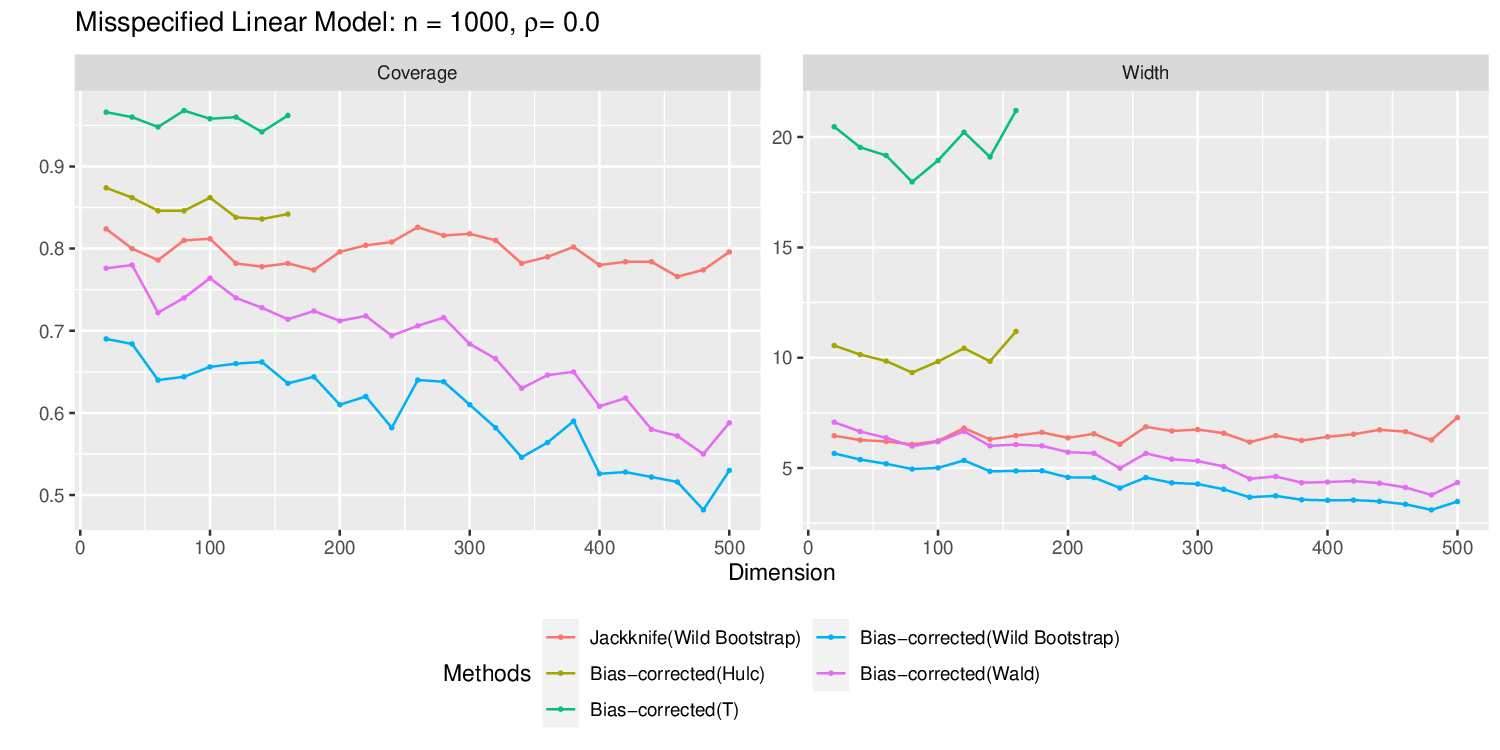}
         \caption{Proposed vs. Jackknife}
     \end{subfigure}\\
     \begin{subfigure}[b]{\textwidth}
         \centering
         \includegraphics[width=\textwidth]{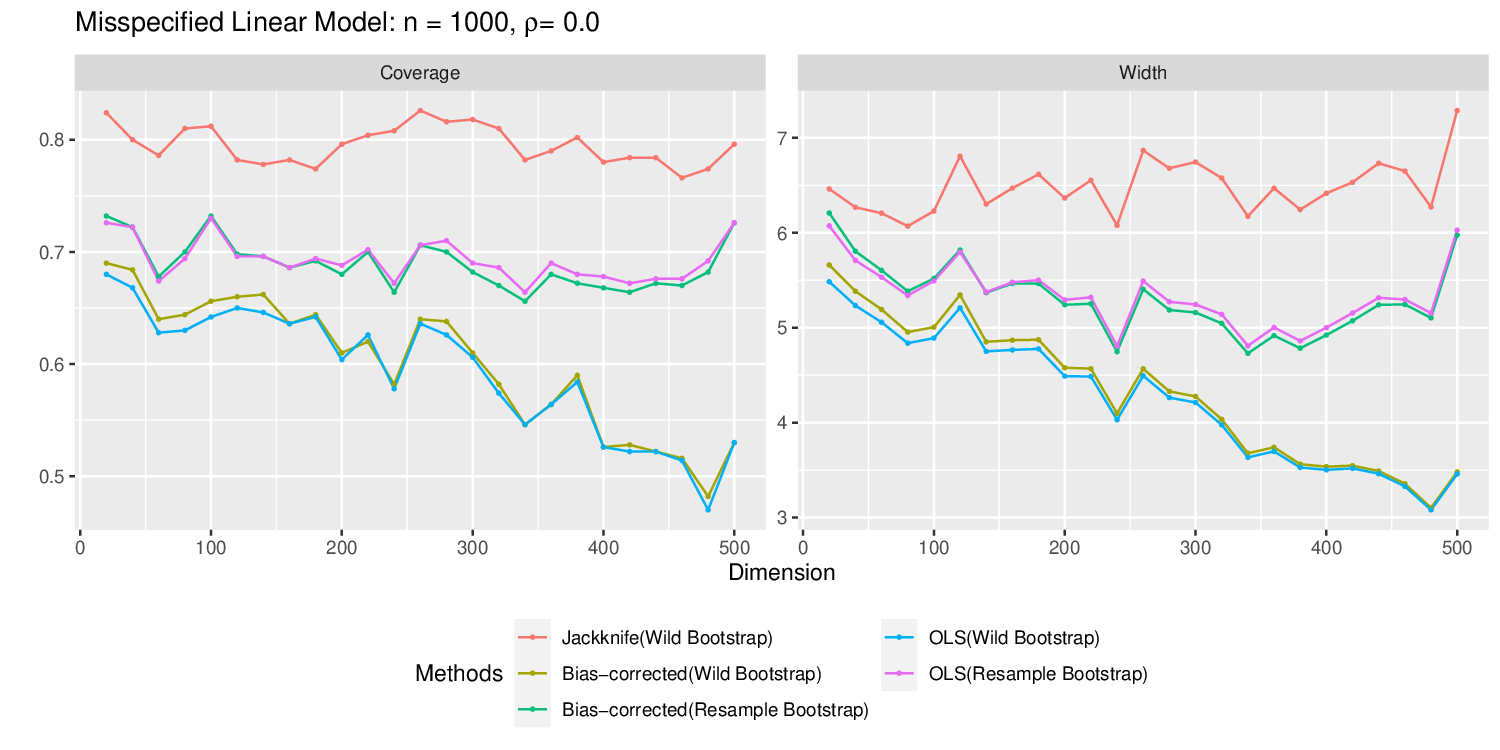}
         \caption{Comparison of Bootstrap-based inferential methods}
     \end{subfigure}     
     \caption{Comparison of coverages and widths of confidence intervals for $\theta^\top\beta$ under the misspecified model with $n=1000$, $\theta=\ev_1$, and $\rho=0.0$.}
     \label{fig:D.2.1.1}
\end{figure}

\begin{figure}
     \centering
     \begin{subfigure}[b]{\textwidth}
         \centering
         \includegraphics[width=\textwidth]{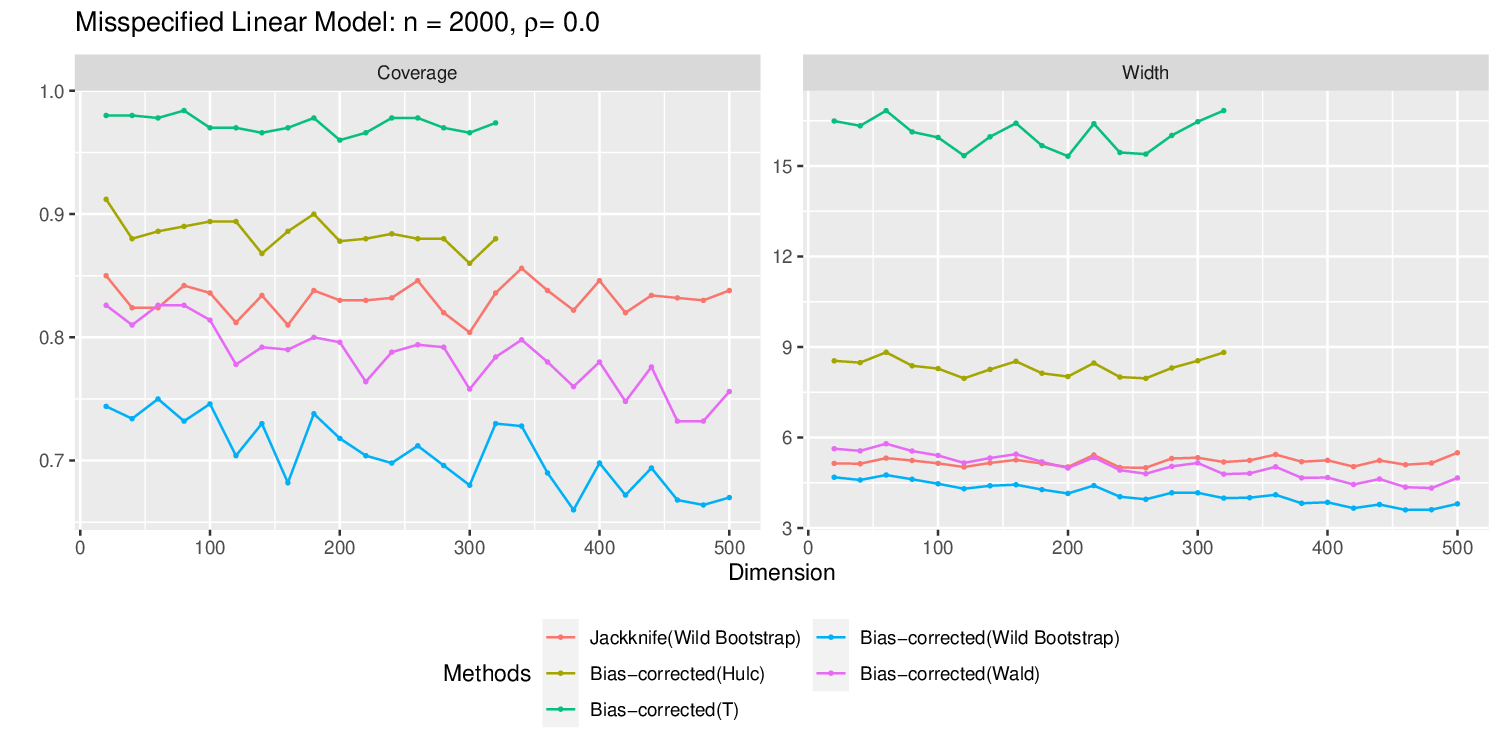}
         \caption{Proposed vs. Jackknife}
     \end{subfigure}\\
     \begin{subfigure}[b]{\textwidth}
         \centering
         \includegraphics[width=\textwidth]{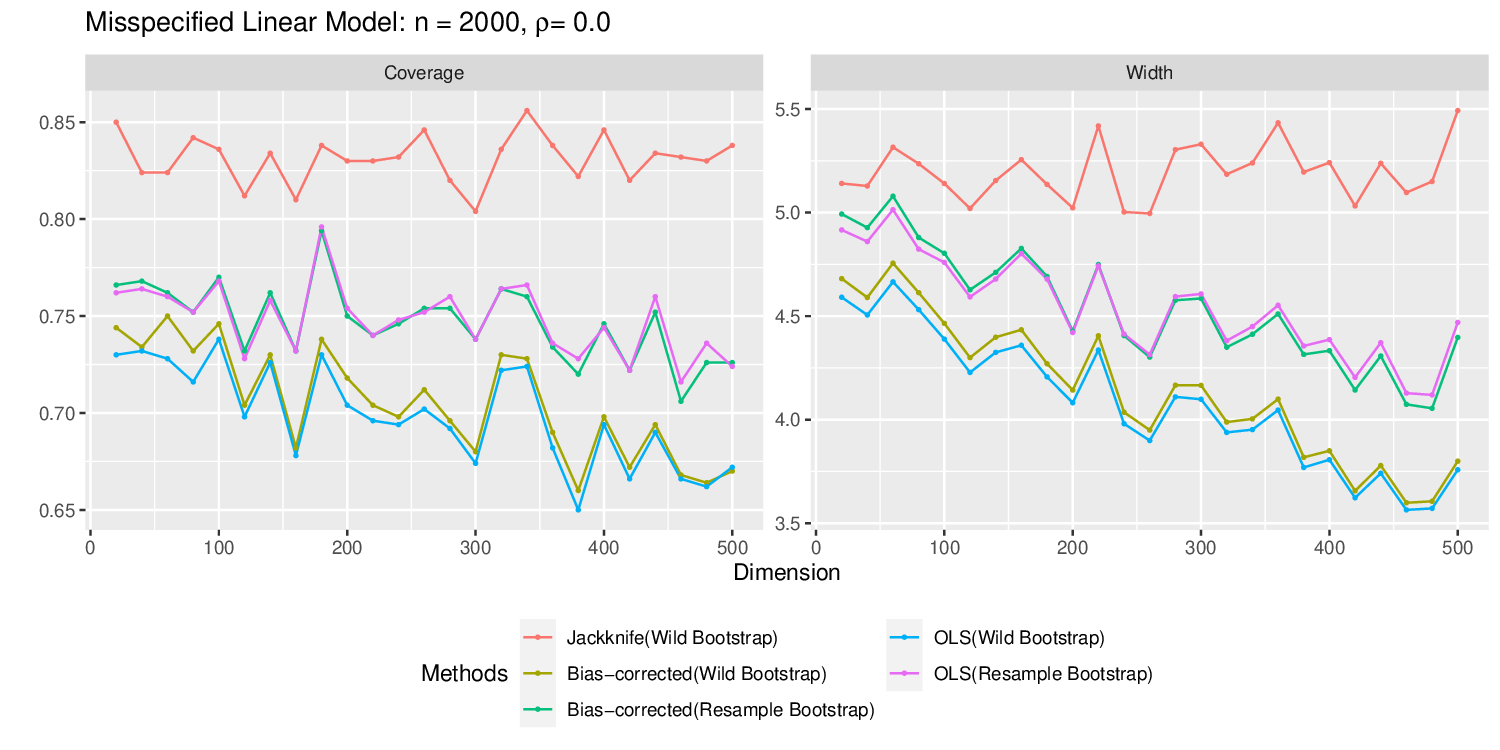}
         \caption{Comparison of Bootstrap-based inferential methods}
     \end{subfigure}     
     \caption{Comparison of coverages and widths of confidence intervals for $\theta^\top\beta$ under the misspecified model with $n=2000$, $\theta=\ev_1$, and $\rho=0.0$.}
     \label{fig:D.2.1.2}
\end{figure}

\begin{figure}
     \centering
     \begin{subfigure}[b]{\textwidth}
         \centering
         \includegraphics[width=\textwidth]{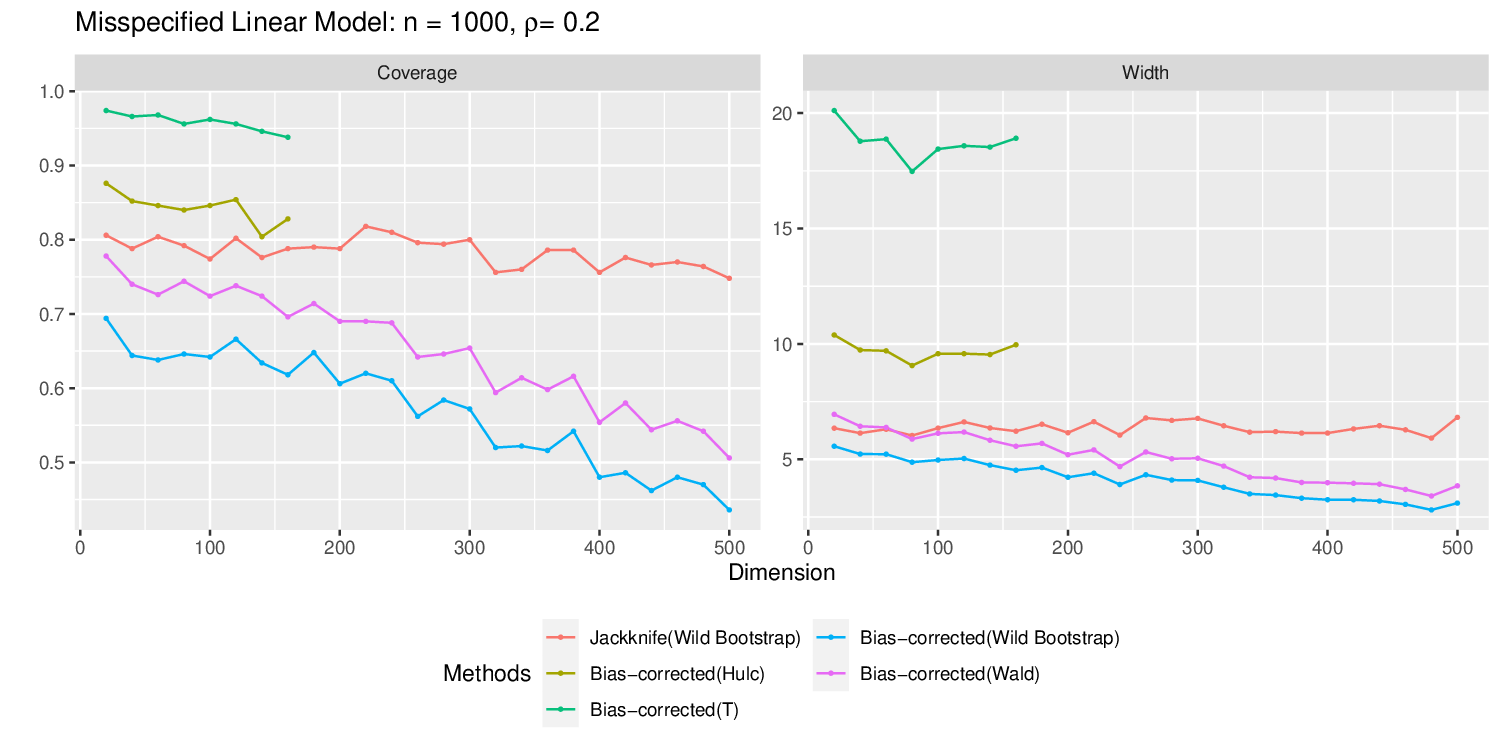}
        \caption{Proposed vs. Jackknife}
     \end{subfigure}\\
     \begin{subfigure}[b]{\textwidth}
         \centering
         \includegraphics[width=\textwidth]{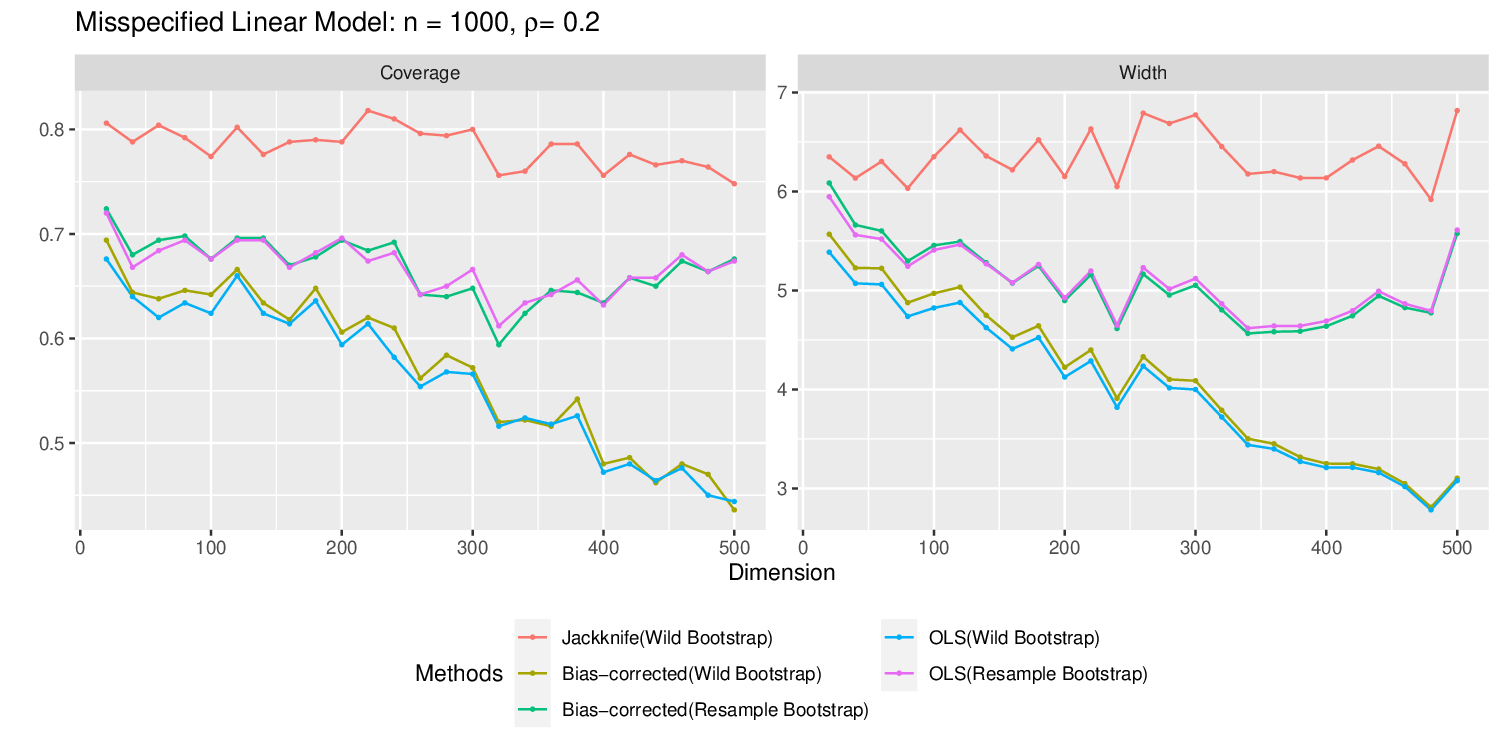}
         \caption{Comparison of Bootstrap-based inferential methods}
     \end{subfigure}     
     \caption{Comparison of coverages and widths of confidence intervals for $\theta^\top\beta$ under the misspecified model with $n=1000$, $\theta=\ev_1$, and $\rho=0.2$.}
     \label{fig:D.2.2.1}
\end{figure}

\begin{figure}
     \centering
     \begin{subfigure}[b]{\textwidth}
         \centering
         \includegraphics[width=\textwidth]{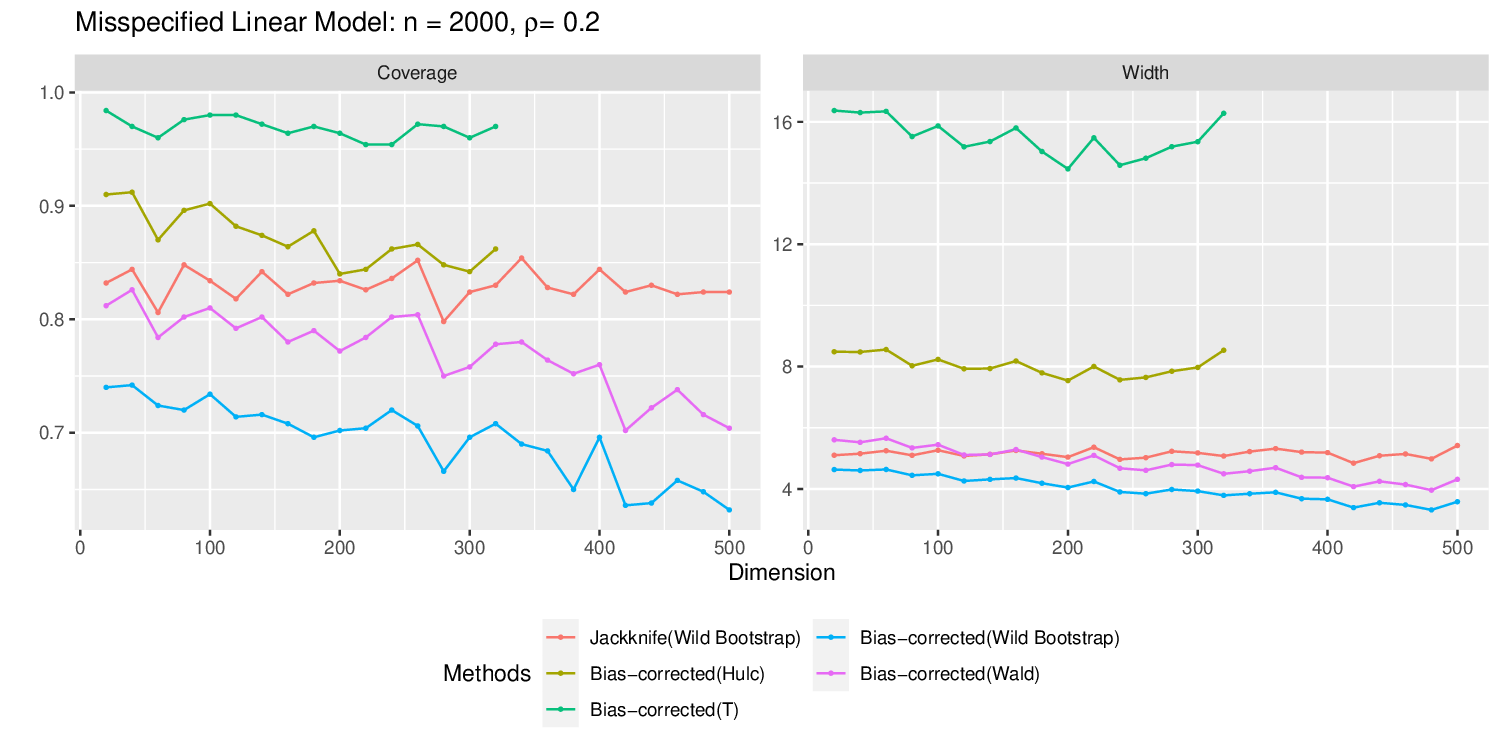}
         \caption{Proposed vs. Jackknife}
     \end{subfigure}\\
     \begin{subfigure}[b]{\textwidth}
         \centering
         \includegraphics[width=\textwidth]{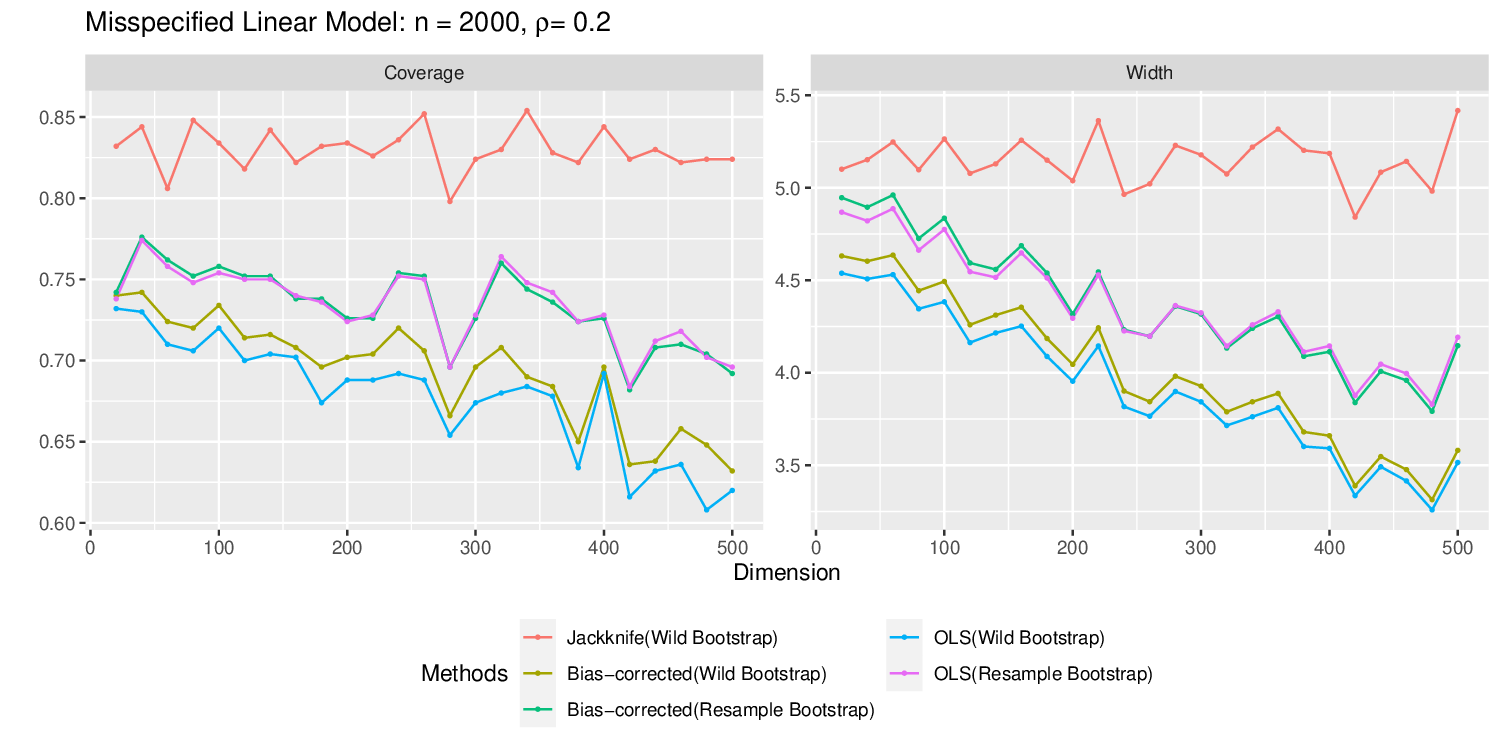}
         \caption{Comparison of Bootstrap-based inferential methods}
     \end{subfigure}     
     \caption{Comparison of coverages and widths of confidence intervals for $\theta^\top\beta$ under the misspecified model with $n=2000$, $\theta=\ev_1$, and $\rho=0.2$.}
     \label{fig:D.2.2.2}
\end{figure}

\begin{figure}
     \centering
     \begin{subfigure}[b]{\textwidth}
         \centering
         \includegraphics[width=\textwidth]{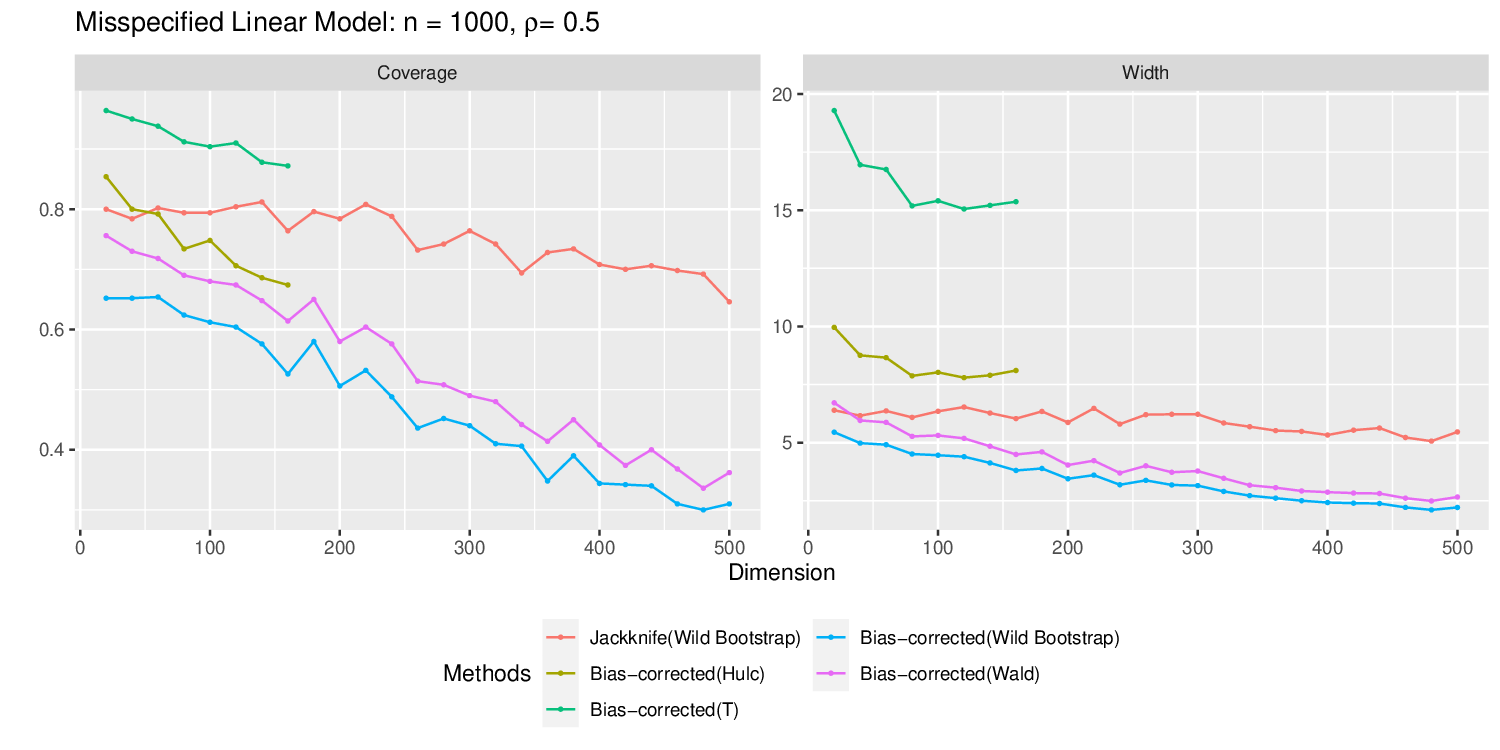}
         \caption{Proposed vs. Jackknife}
     \end{subfigure}\\
     \begin{subfigure}[b]{\textwidth}
         \centering
         \includegraphics[width=\textwidth]{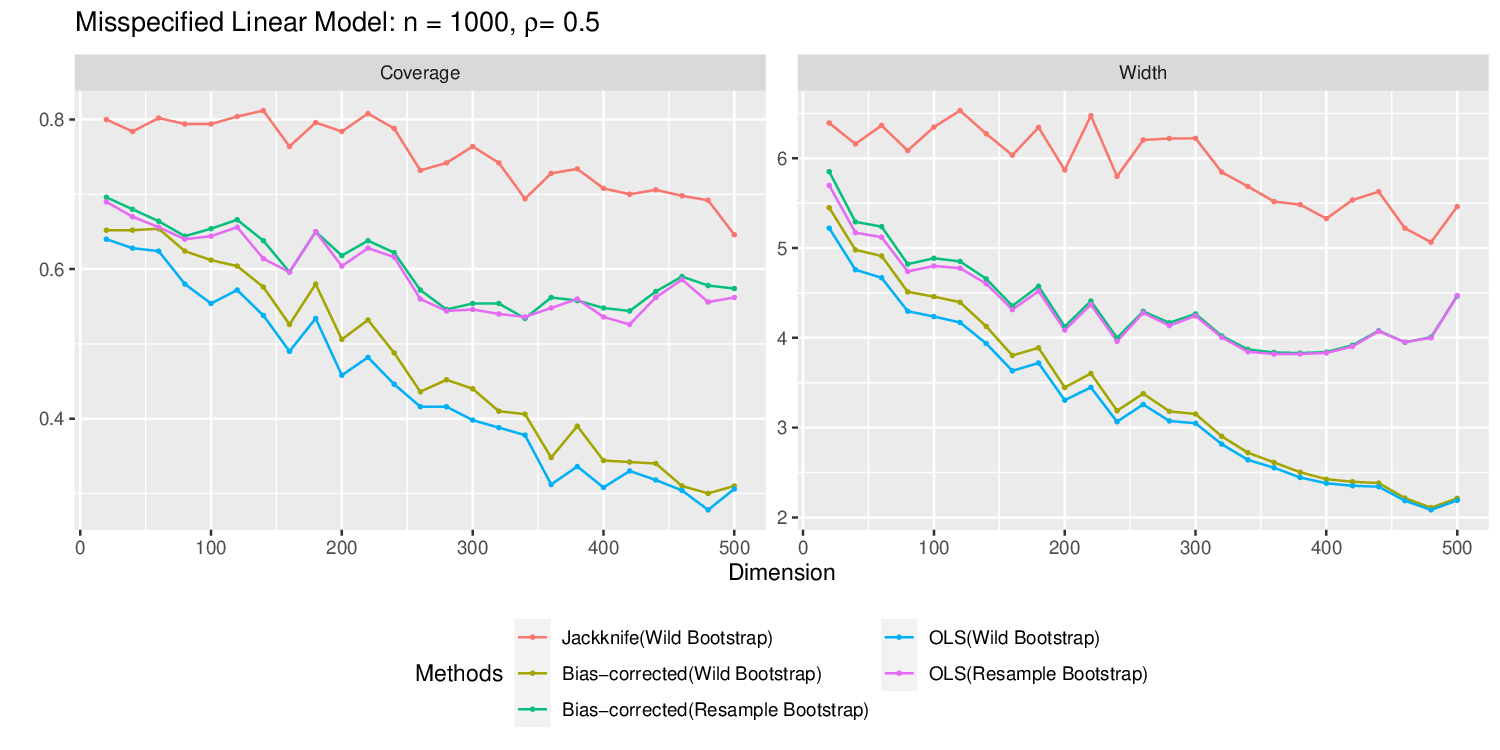}
         \caption{Comparison of Bootstrap-based inferential methods}
     \end{subfigure}     
     \caption{Comparison of coverages and widths of confidence intervals for $\theta^\top\beta$ under the misspecified model with $n=1000$, $\theta=\ev_1$, and $\rho=0.5$.}
     \label{fig:D.2.3.1}
\end{figure}

\begin{figure}
     \centering
     \begin{subfigure}[b]{\textwidth}
         \centering
         \includegraphics[width=\textwidth]{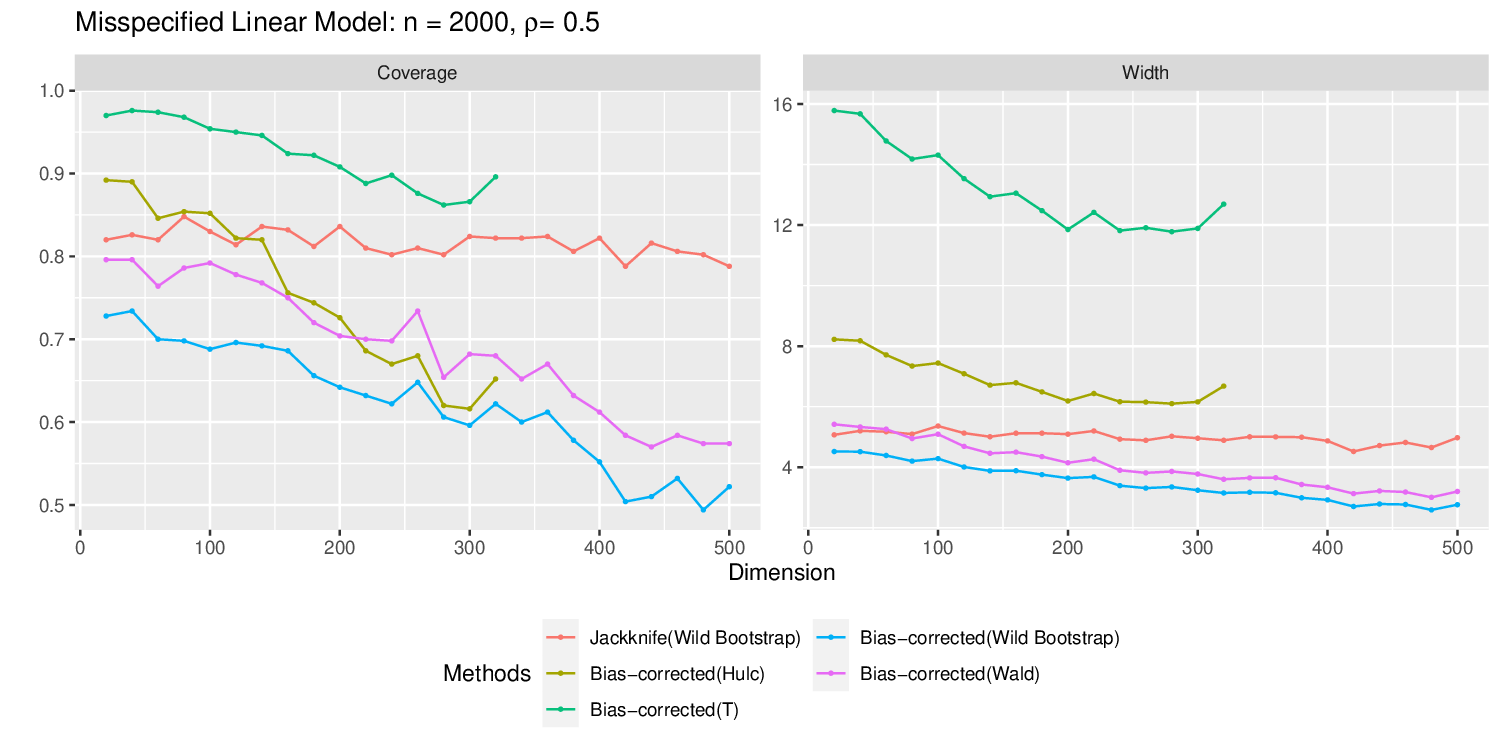}
         \caption{Proposed vs. Jackknife}
     \end{subfigure}\\
     \begin{subfigure}[b]{\textwidth}
         \centering
         \includegraphics[width=\textwidth]{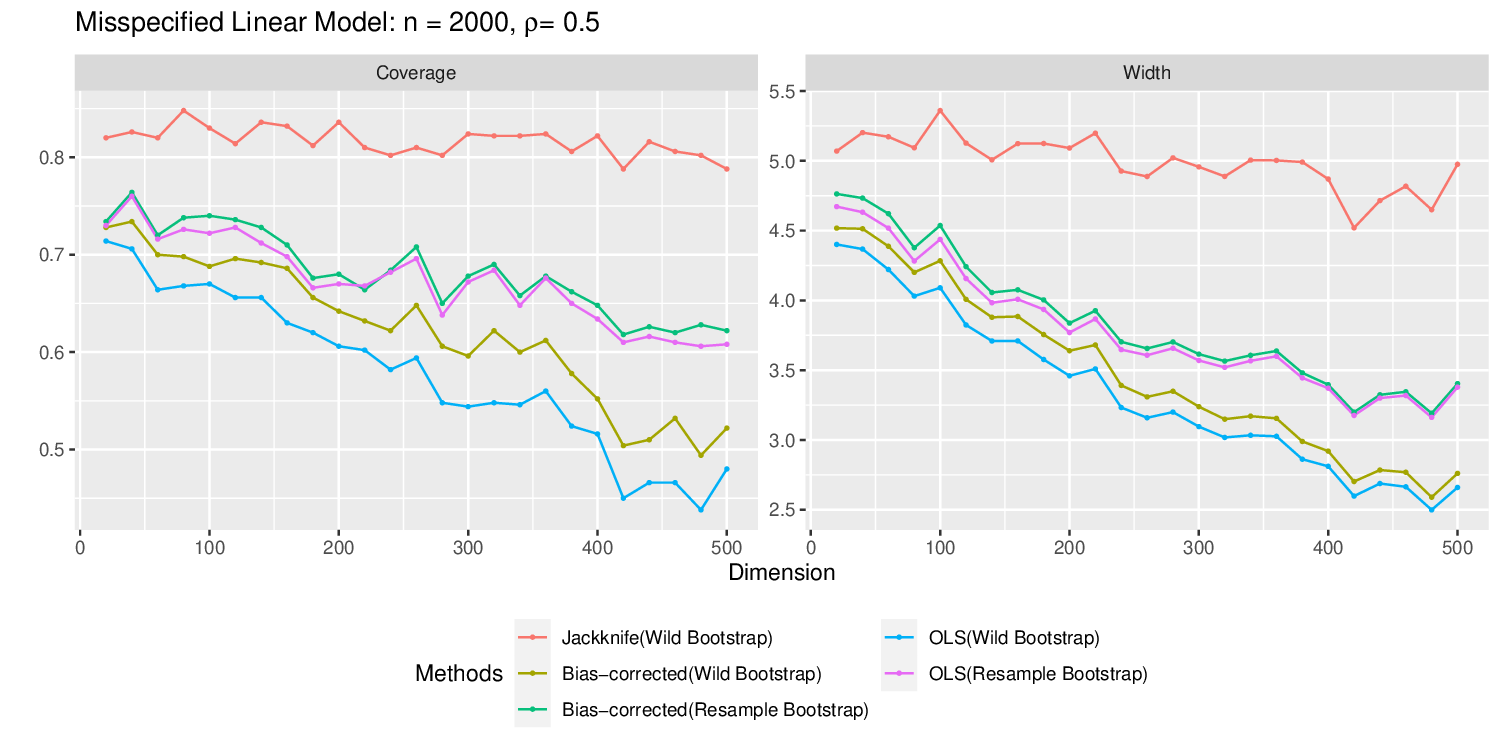}
         \caption{Comparison of Bootstrap-based inferential methods}
     \end{subfigure}     
     \caption{Comparison of coverages and widths of confidence intervals for $\theta^\top\beta$ under the misspecified model with $n=2000$, $\theta=\ev_1$, and $\rho=0.5$.}
     \label{fig:D.2.3.2}
\end{figure}

\begin{figure}
     \centering
     \begin{subfigure}[b]{\textwidth}
         \centering
         \includegraphics[width=\textwidth]{Figures/CI_null_n=1000_rho=0.0_complex_small.eps}
         \caption{Proposed vs. Jackknife}
     \end{subfigure}\\
     \begin{subfigure}[b]{\textwidth}
         \centering
         \includegraphics[width=\textwidth]{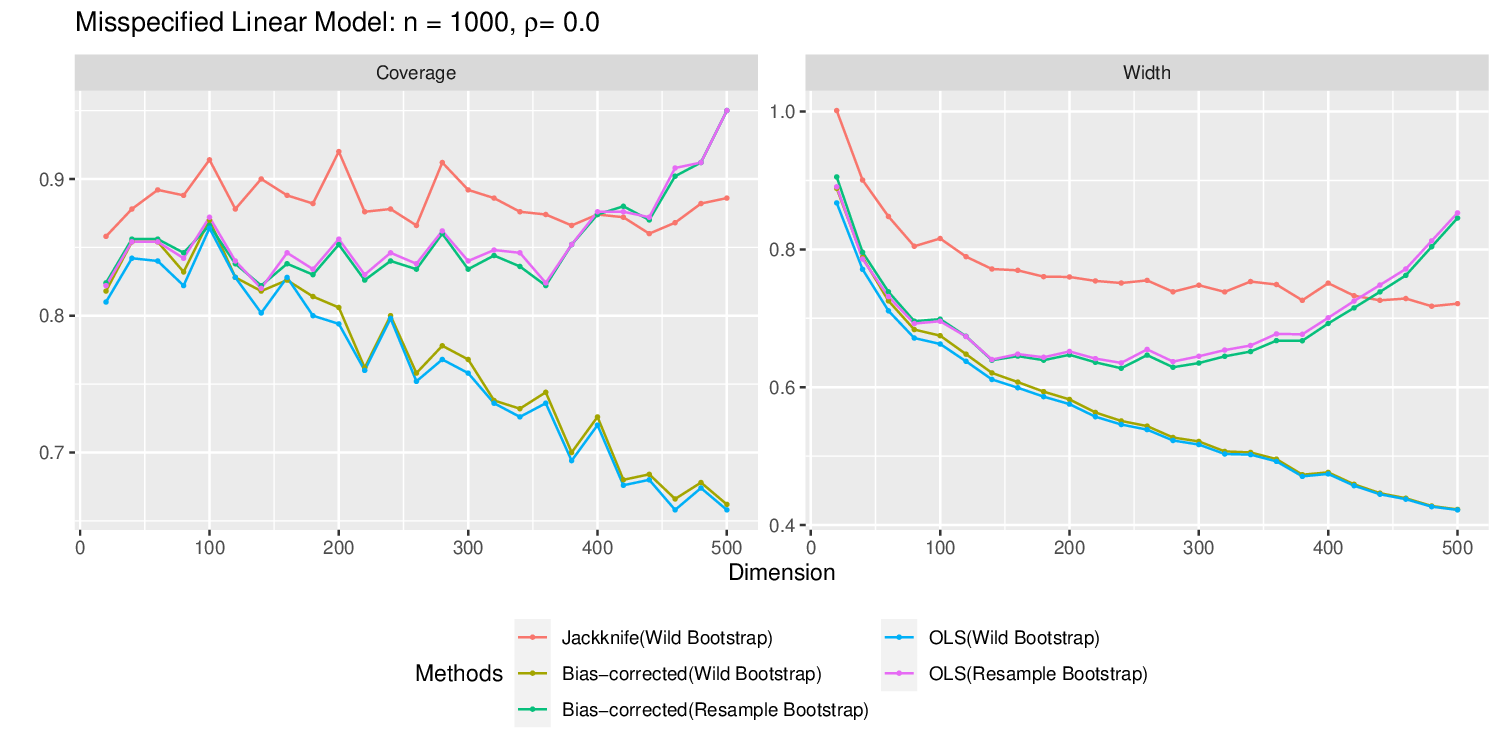}
         \caption{Comparison of Bootstrap-based inferential methods}
     \end{subfigure}     
     \caption{Comparison of coverages and widths of confidence intervals for $\theta^\top\beta$ under the misspecified model with $n=1000$, $\theta=1_d^\top/\sqrt{d}$, and $\rho=0.0$.}
     \label{fig:D.2.4.1}
\end{figure}

\begin{figure}
     \centering
     \begin{subfigure}[b]{\textwidth}
         \centering
         \includegraphics[width=\textwidth]{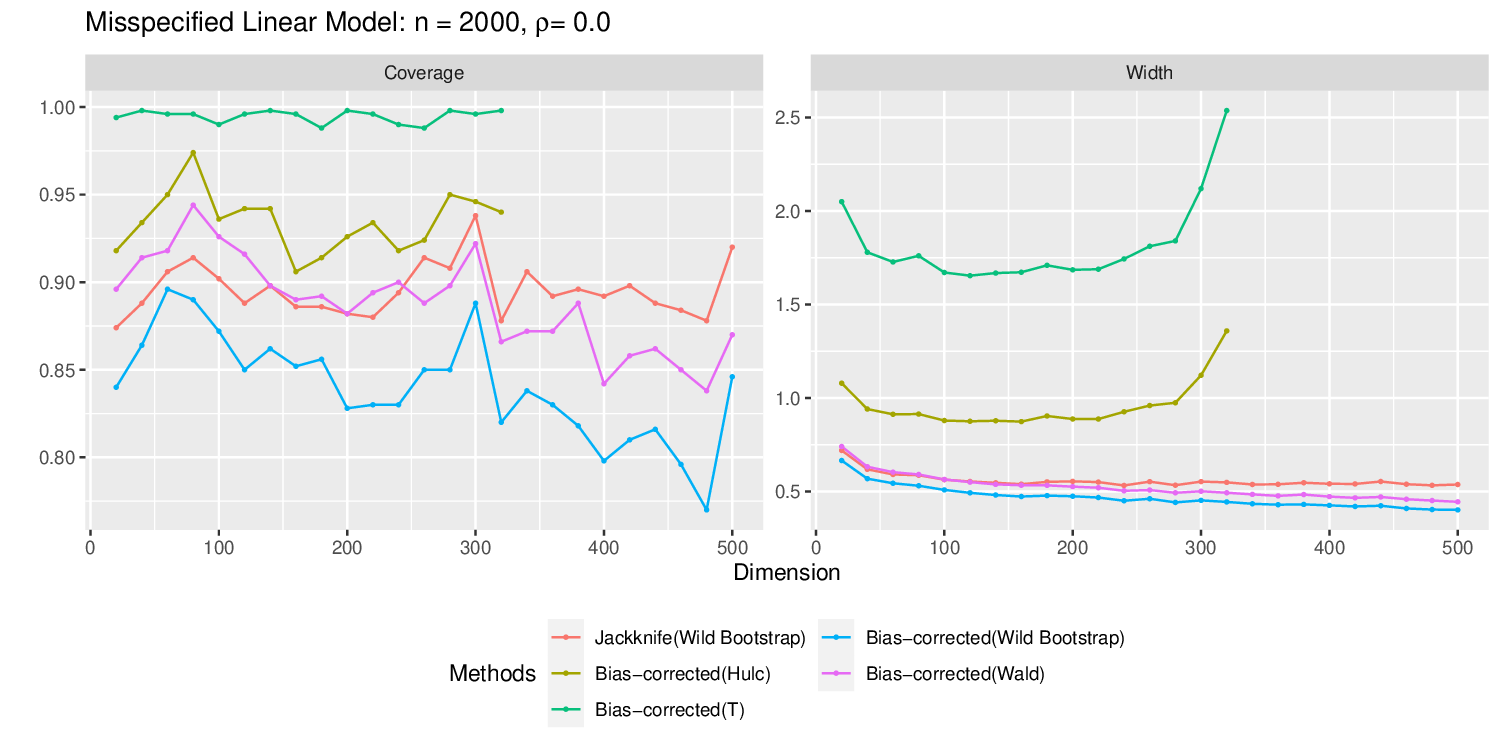}
         \caption{Proposed vs. Jackknife}
     \end{subfigure}\\
     \begin{subfigure}[b]{\textwidth}
         \centering
         \includegraphics[width=\textwidth]{Figures/CI_null_n=2000_rho=0.0_complex_small_bs.eps}
         \caption{Comparison of Bootstrap-based inferential methods}
     \end{subfigure}     
     \caption{Comparison of coverages and widths of confidence intervals for $\theta^\top\beta$ under the misspecified model with $n=2000$, $\theta=1_d^\top/\sqrt{d}$, and $\rho=0.0$.}
     \label{fig:D.2.4.2}
\end{figure}
\begin{figure}
     \centering
     \begin{subfigure}[b]{\textwidth}
         \centering
         \includegraphics[width=\textwidth]{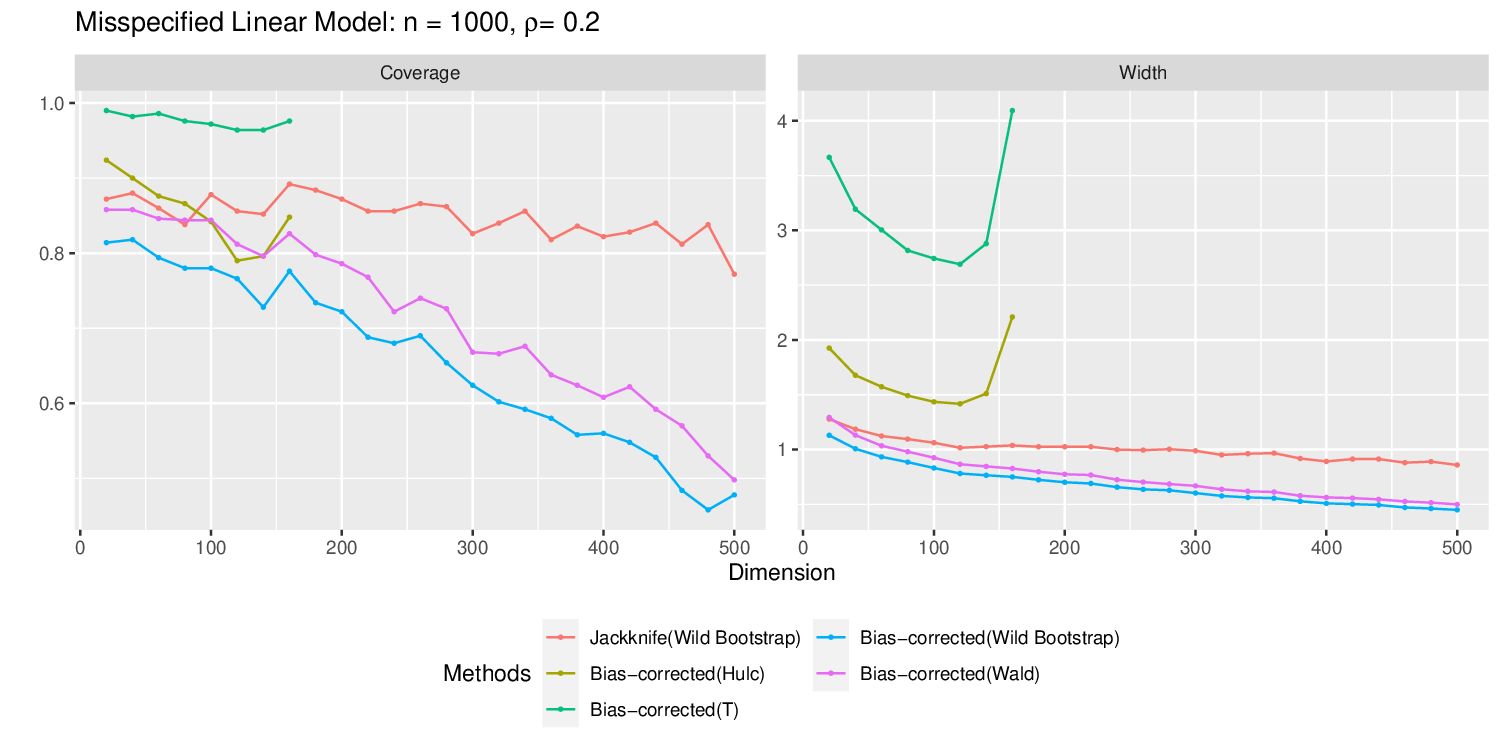}
         \caption{Proposed vs. Jackknife}
     \end{subfigure}\\
     \begin{subfigure}[b]{\textwidth}
         \centering
         \includegraphics[width=\textwidth]{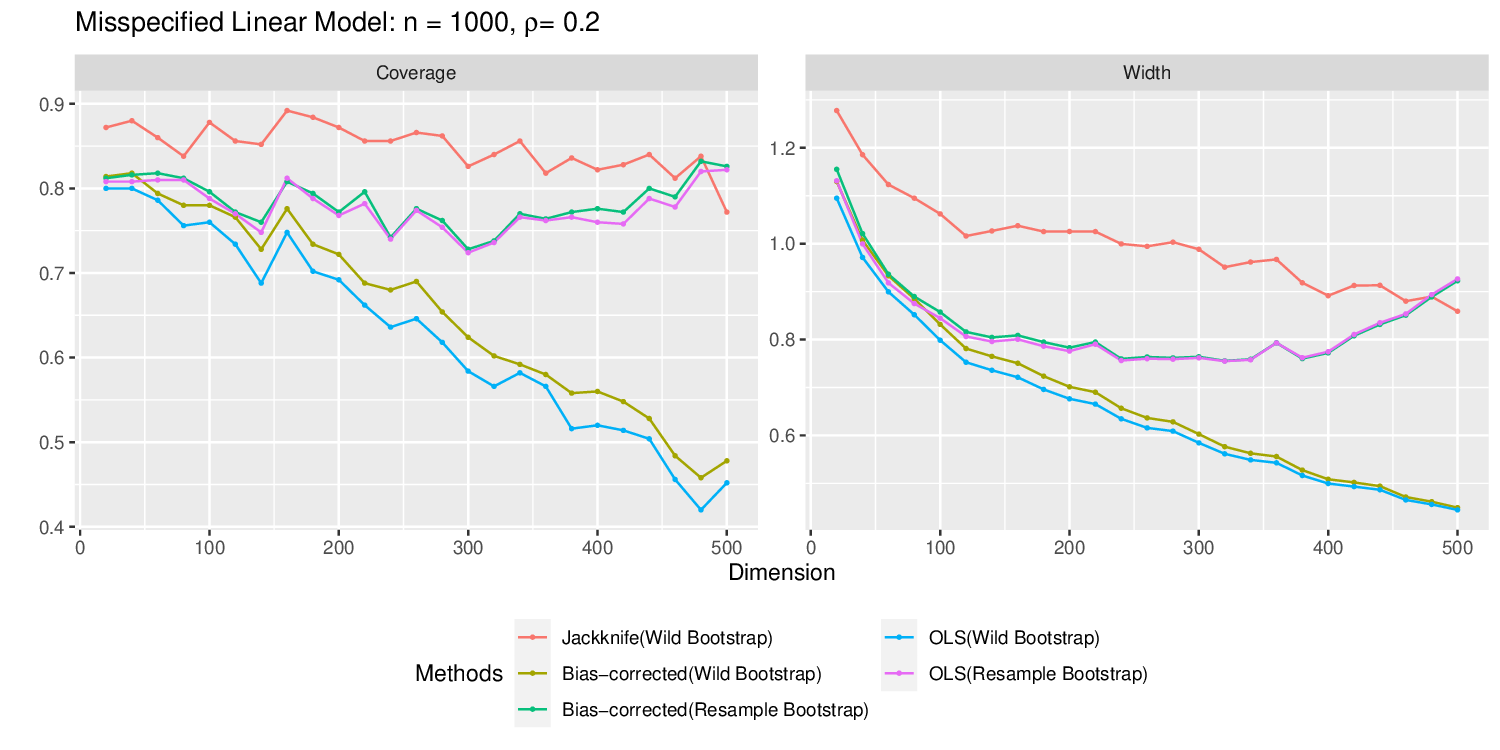}
         \caption{Comparison of Bootstrap-based inferential methods}
     \end{subfigure}     
     \caption{Comparison of coverages and widths of confidence intervals for $\theta^\top\beta$ under the misspecified model with $n=1000$, $\theta=1_d^\top/\sqrt{d}$, and $\rho=0.2$.}
     \label{fig:D.2.5.1}
\end{figure}

\begin{figure}
     \centering
     \begin{subfigure}[b]{\textwidth}
         \centering
         \includegraphics[width=\textwidth]{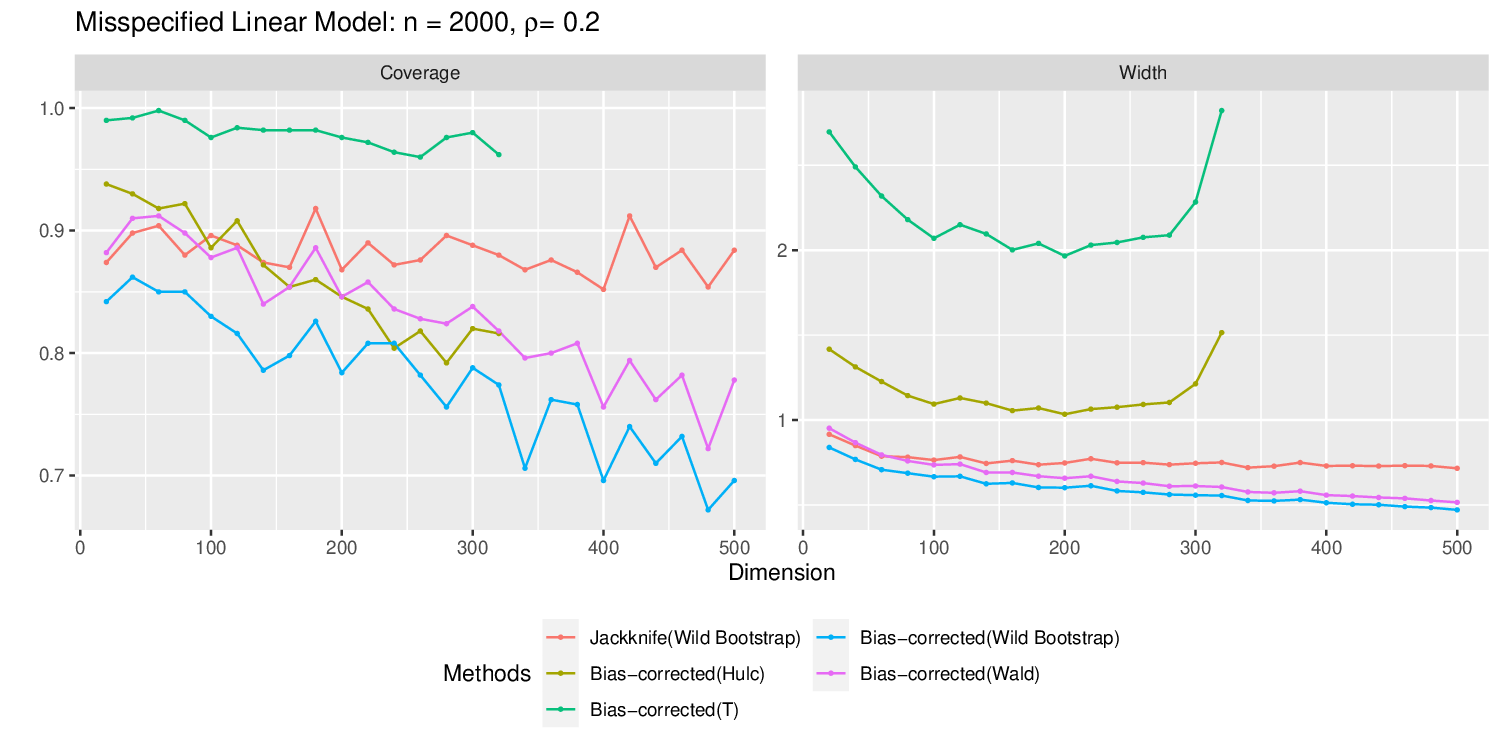}
         \caption{Proposed vs. Jackknife}
     \end{subfigure}\\
     \begin{subfigure}[b]{\textwidth}
         \centering
         \includegraphics[width=\textwidth]{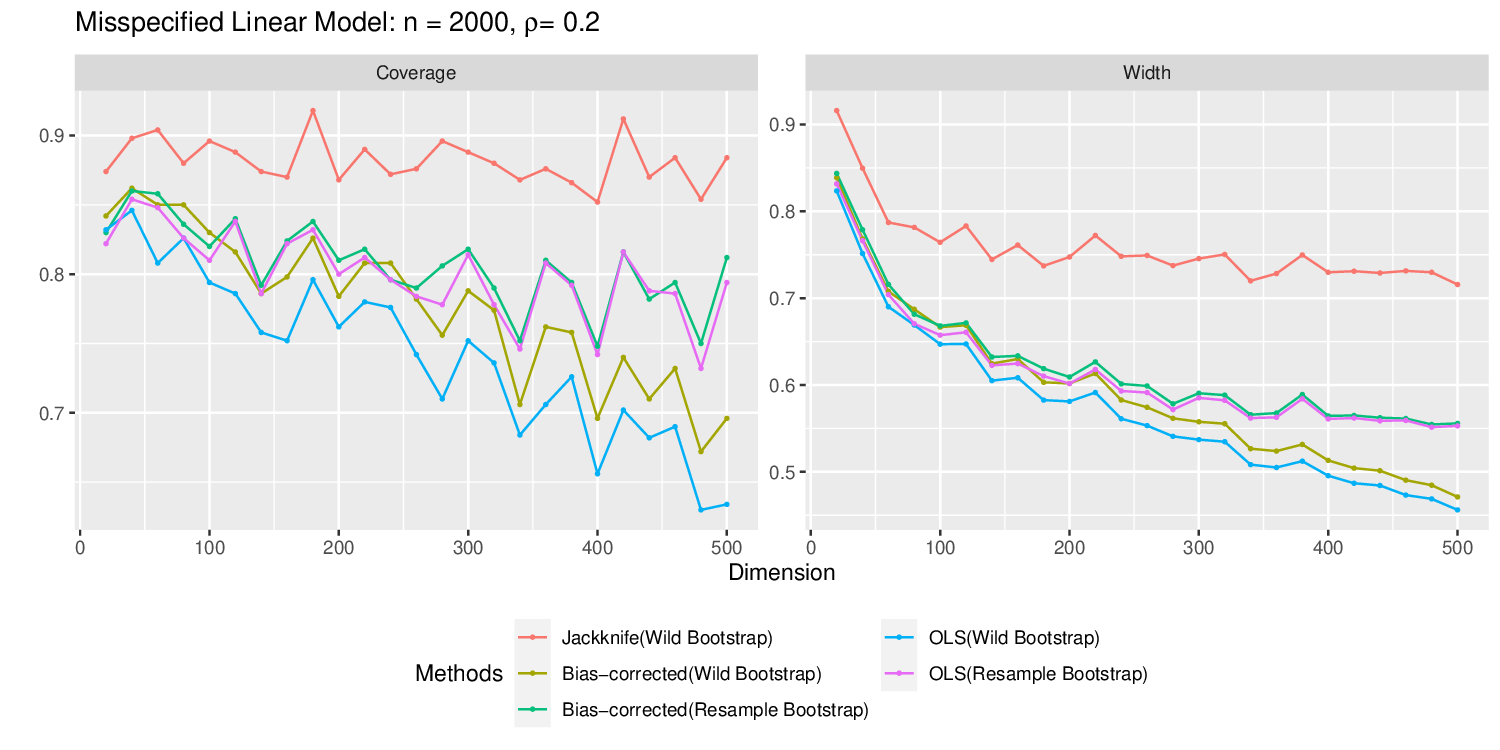}
         \caption{Comparison of Bootstrap-based inferential methods}
     \end{subfigure}     
     \caption{Comparison of coverages and widths of confidence intervals for $\theta^\top\beta$ under the misspecified model with $n=2000$, $\theta=1_d^\top/\sqrt{d}$, and $\rho=0.2$.}
     \label{fig:D.2.5.2}
\end{figure}
\begin{figure}
     \centering
     \begin{subfigure}[b]{\textwidth}
         \centering
         \includegraphics[width=\textwidth]{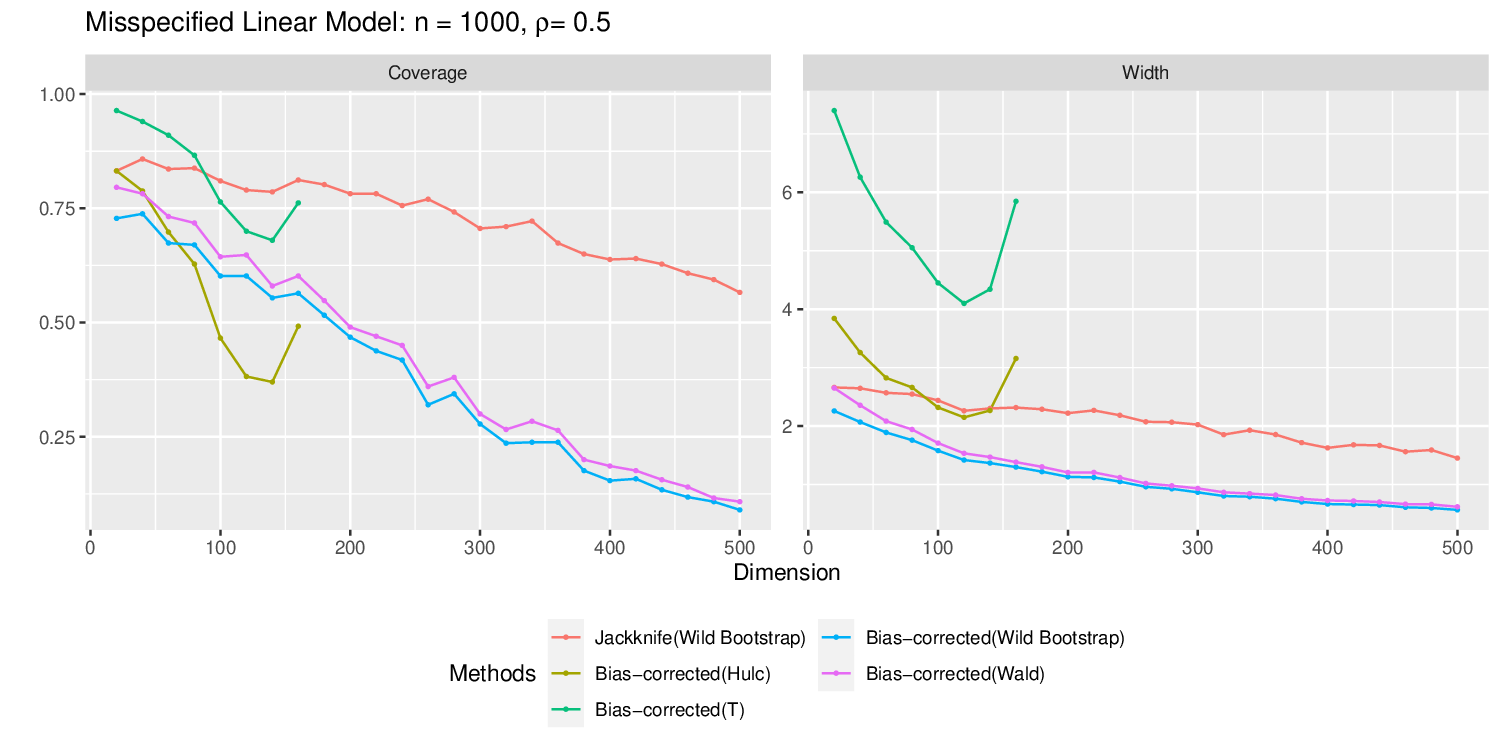}
         \caption{Proposed vs. Jackknife}
     \end{subfigure}\\
     \begin{subfigure}[b]{\textwidth}
         \centering
         \includegraphics[width=\textwidth]{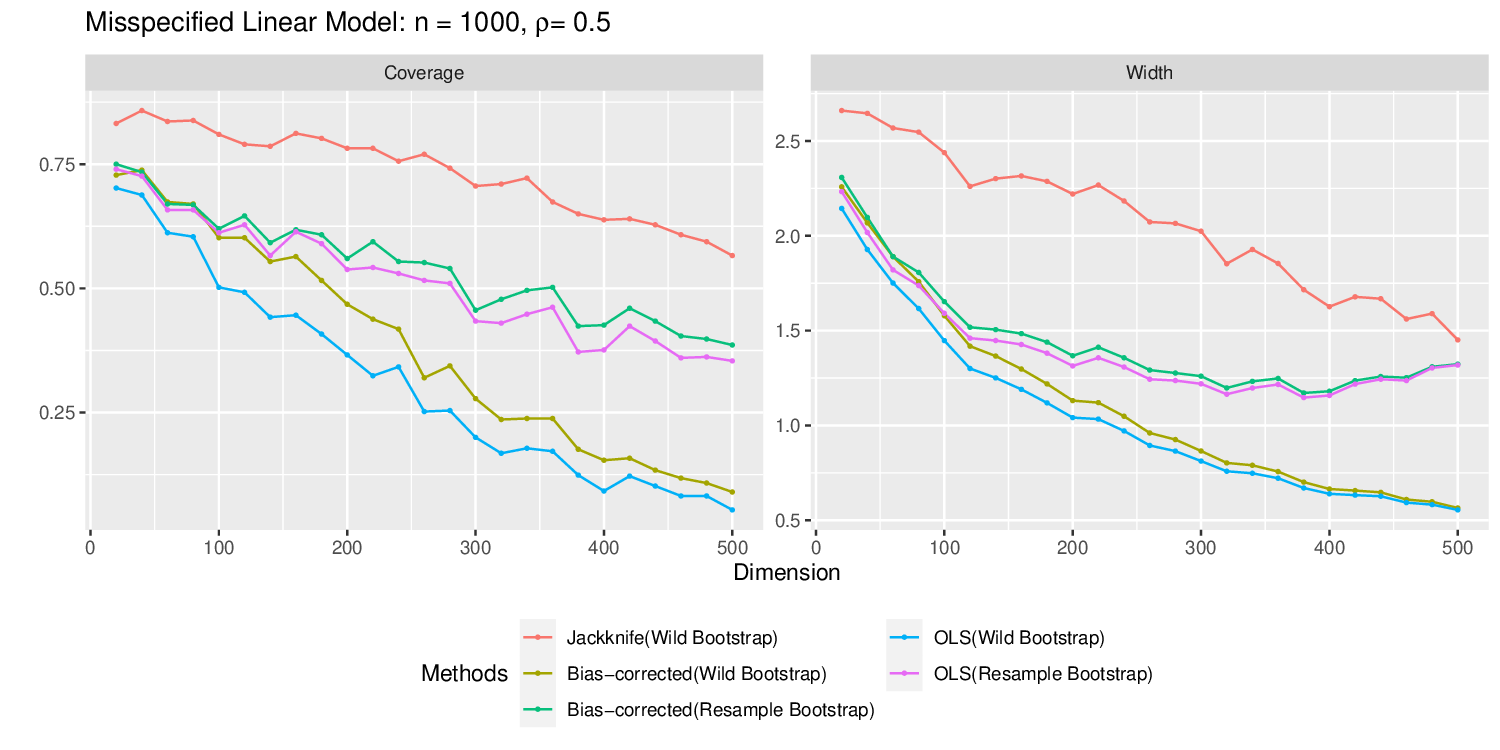}
         \caption{Comparison of Bootstrap-based inferential methods}
     \end{subfigure}     
     \caption{Comparison of coverages and widths of confidence intervals for $\theta^\top\beta$ under the misspecified model with $n=1000$, $\theta=1_d^\top/\sqrt{d}$, and $\rho=0.5$.}
     \label{fig:D.2.6.1}
\end{figure}

\begin{figure}
     \centering
     \begin{subfigure}[b]{\textwidth}
         \centering
         \includegraphics[width=\textwidth]{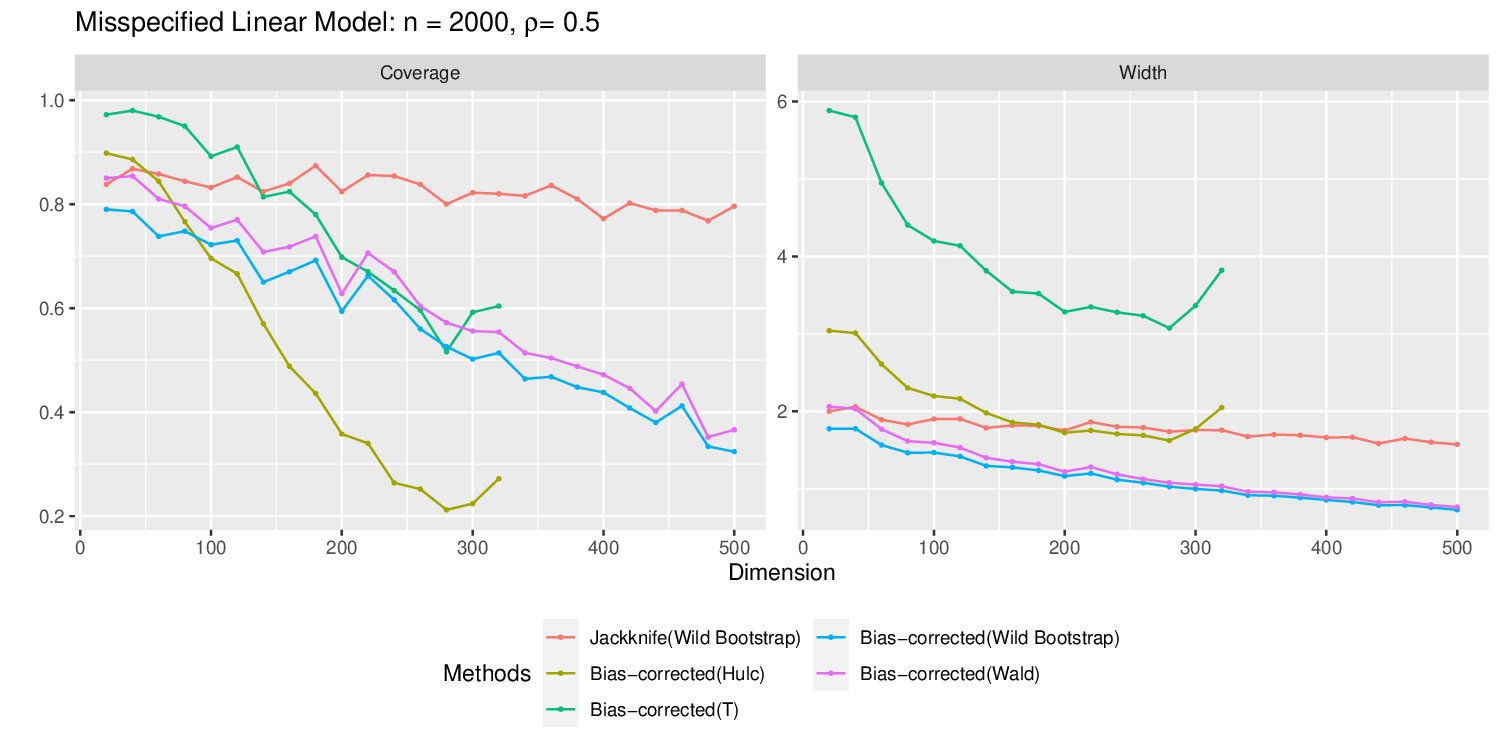}
         \caption{Proposed vs. Jackknife}
     \end{subfigure}\\
     \begin{subfigure}[b]{\textwidth}
         \centering
         \includegraphics[width=\textwidth]{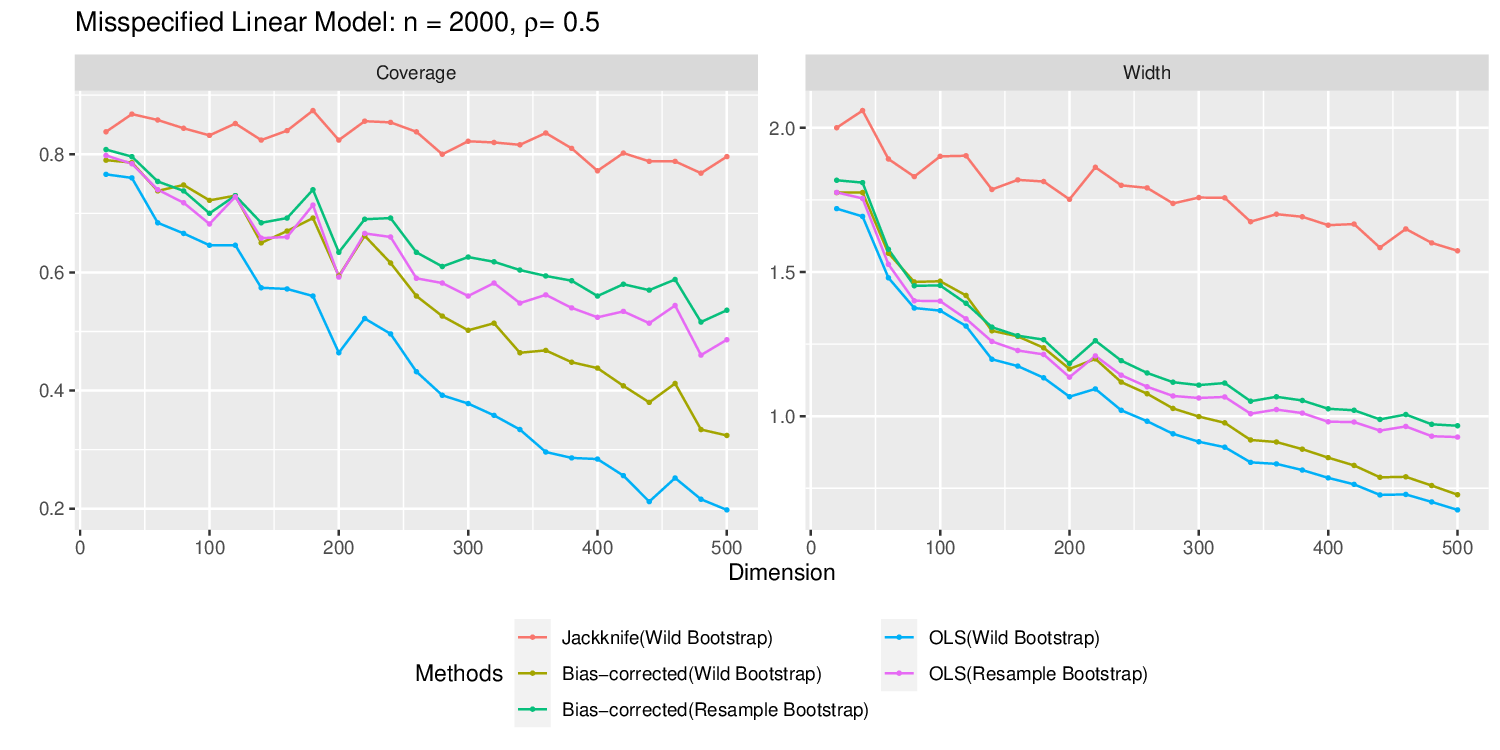}
         \caption{Comparison of Bootstrap-based inferential methods}
     \end{subfigure}     
     \caption{Comparison of coverages and widths of confidence intervals for $\theta^\top\beta$ under the misspecified model with $n=2000$, $\theta=1_d^\top/\sqrt{d}$, and $\rho=0.5$.}
     \label{fig:D.2.6.2}
\end{figure}

\begin{figure}
     \centering
     \begin{subfigure}[b]{\textwidth}
         \centering
         \includegraphics[width=.7\textwidth]{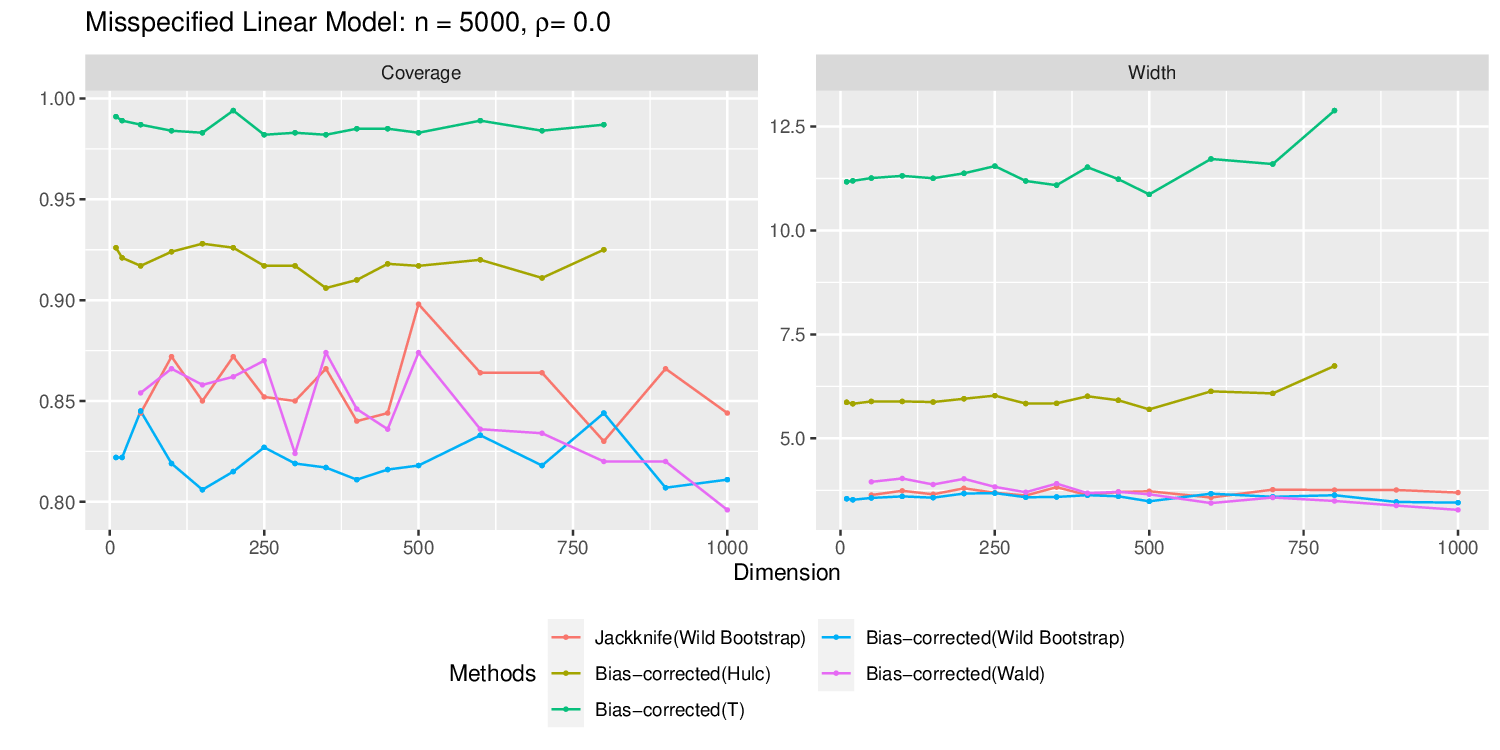}
         \caption{Proposed vs. Jackknife}
     \end{subfigure}\\\begin{subfigure}[b]{\textwidth}
         \centering
         \includegraphics[width=.7\textwidth]{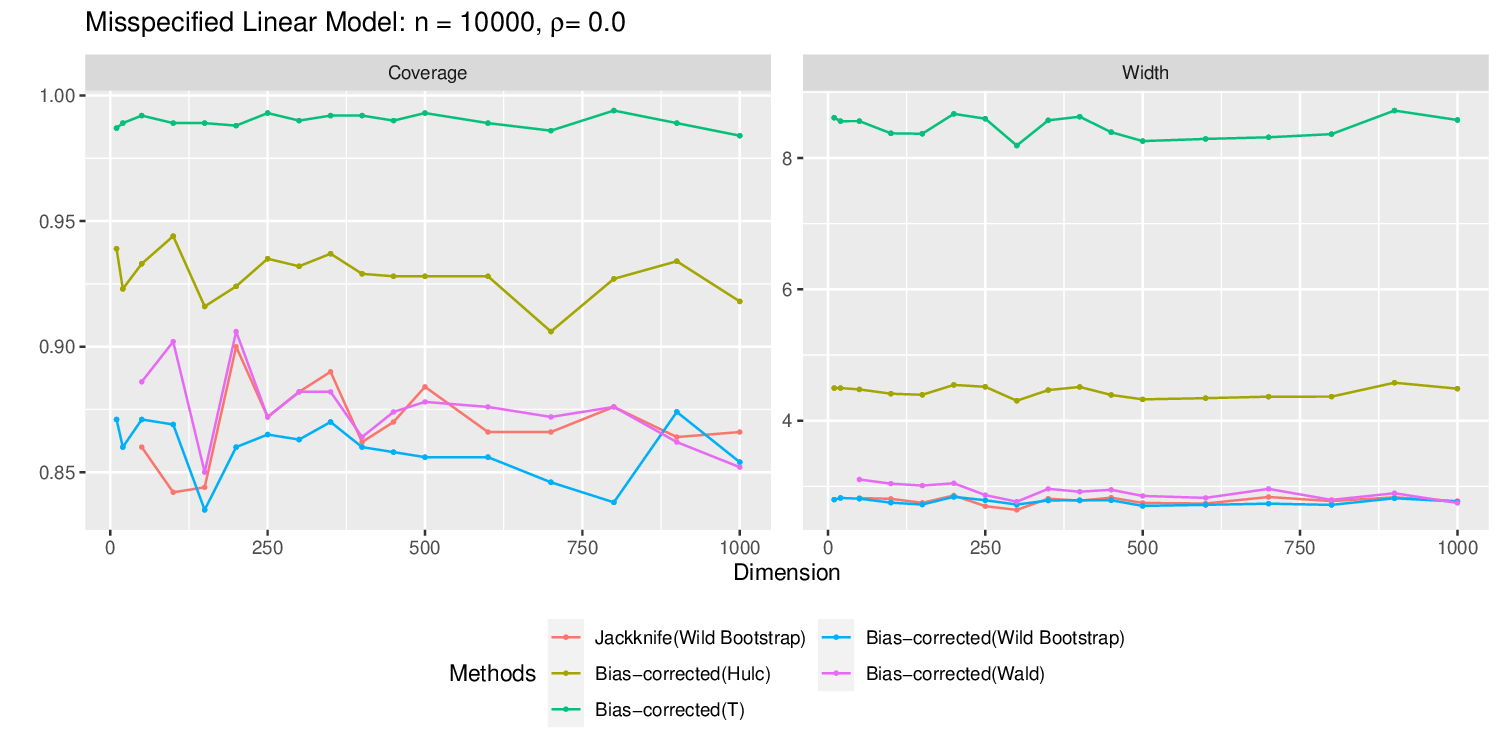}
         \caption{Proposed vs. Jackknife}
     \end{subfigure}\\
     \begin{subfigure}[b]{\textwidth}
         \centering
         \includegraphics[width=.7\textwidth]{Figures/CI_null_n=20000_rho=0.0_complex_large.eps}
         \caption{Proposed vs. Jackknife}
     \end{subfigure}\\
     
     \caption{Comparison of coverages and widths of confidence intervals for $\theta^\top\beta$ under the misspecified model with $n\in\set{5000, 10000, 20000}$, $\theta=\ev_1$, and $\rho=0.0$.}
     \label{fig:D.2.7}
\end{figure}
\begin{figure}
     \centering
     \begin{subfigure}[b]{\textwidth}
         \centering
         \includegraphics[width=.7\textwidth]{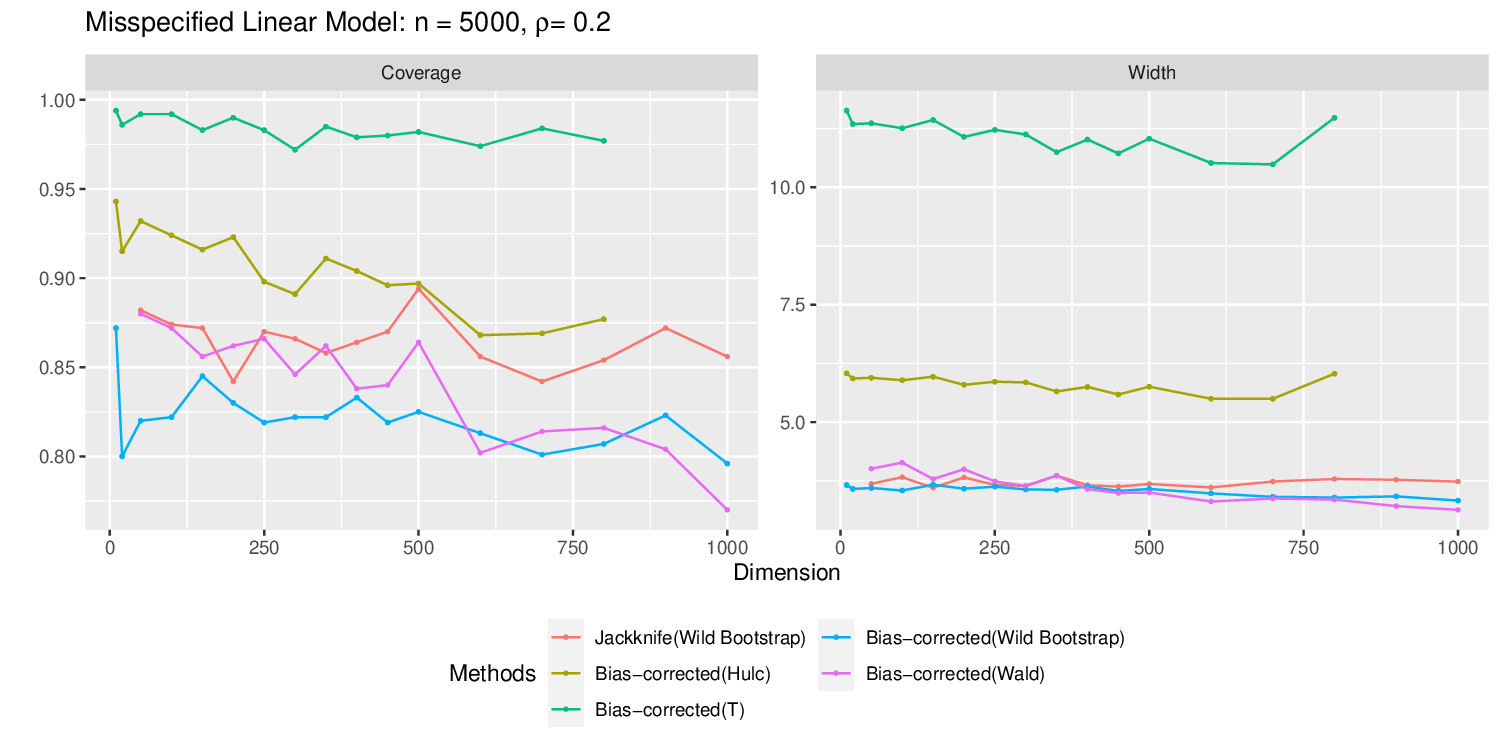}
         \caption{Proposed vs. Jackknife}
     \end{subfigure}\\\begin{subfigure}[b]{\textwidth}
         \centering
         \includegraphics[width=.7\textwidth]{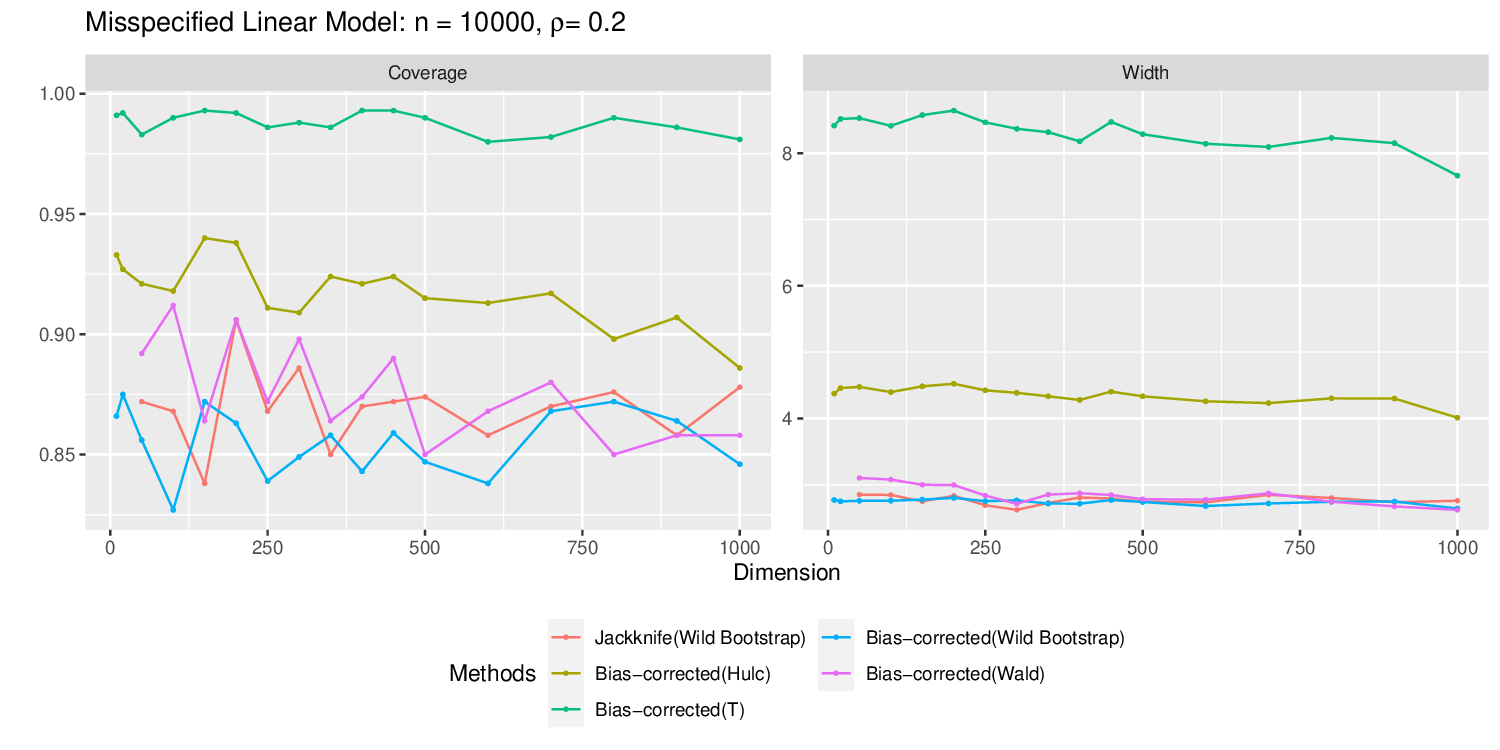}
         \caption{Proposed vs. Jackknife}
     \end{subfigure}\\
     \begin{subfigure}[b]{\textwidth}
         \centering
         \includegraphics[width=.7\textwidth]{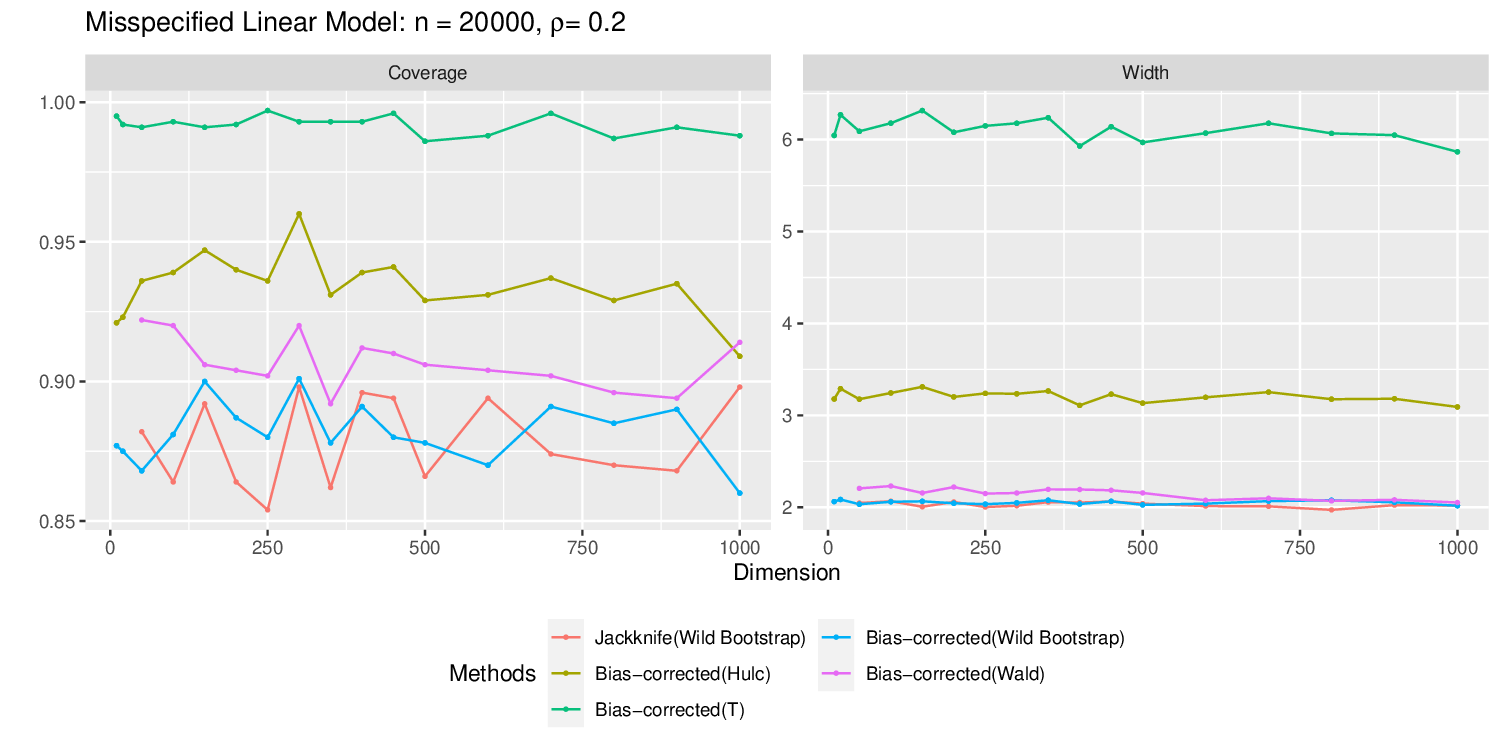}
         \caption{Proposed vs. Jackknife}
     \end{subfigure}\\
     
     \caption{Comparison of coverages and widths of confidence intervals for $\theta^\top\beta$ under the misspecified model with $n\in\set{5000, 10000, 20000}$, $\theta=\ev_1$, and $\rho=0.2$.}
     \label{fig:D.2.8}
\end{figure}

\begin{figure}
     \centering
     \begin{subfigure}[b]{\textwidth}
         \centering
         \includegraphics[width=.7\textwidth]{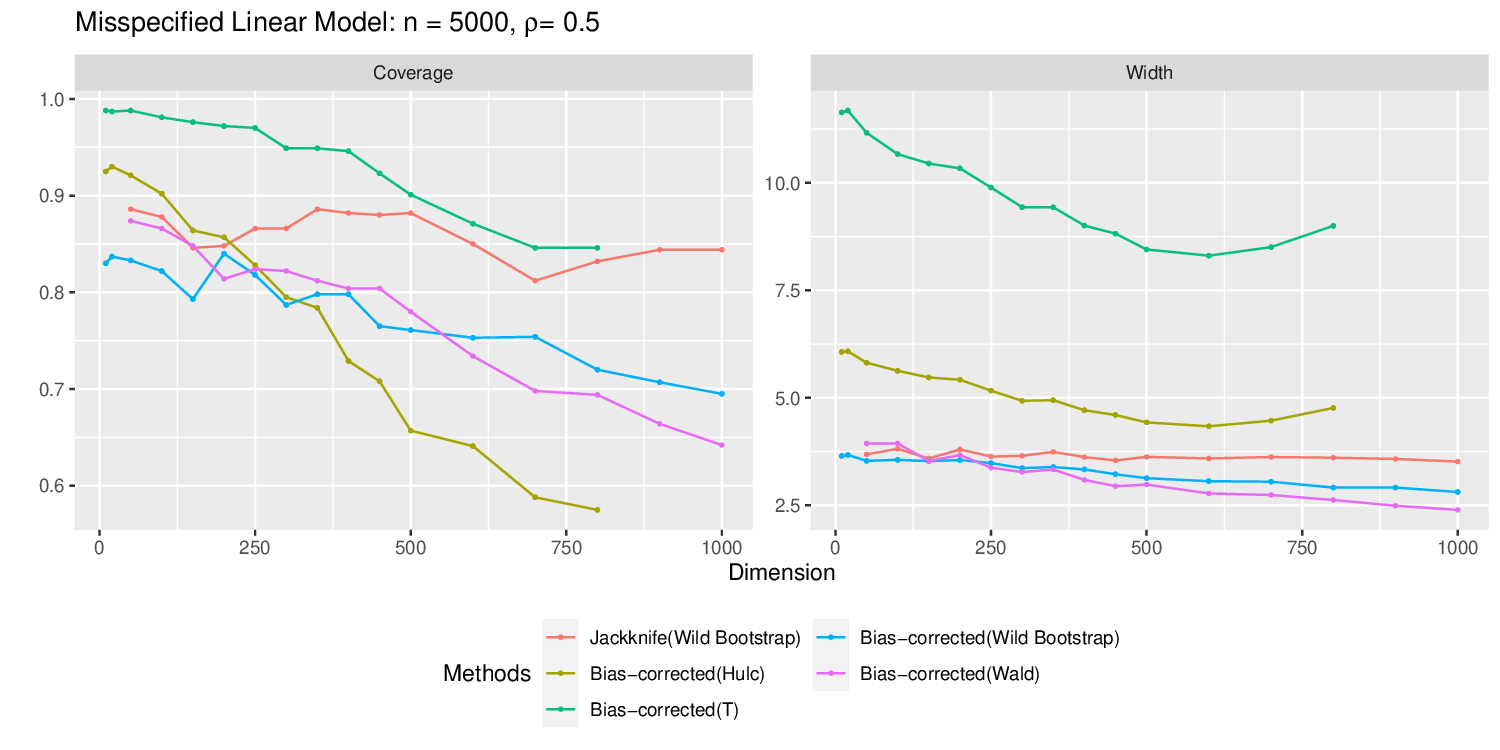}
         \caption{Proposed vs. Jackknife}
     \end{subfigure}\\\begin{subfigure}[b]{\textwidth}
         \centering
         \includegraphics[width=.7\textwidth]{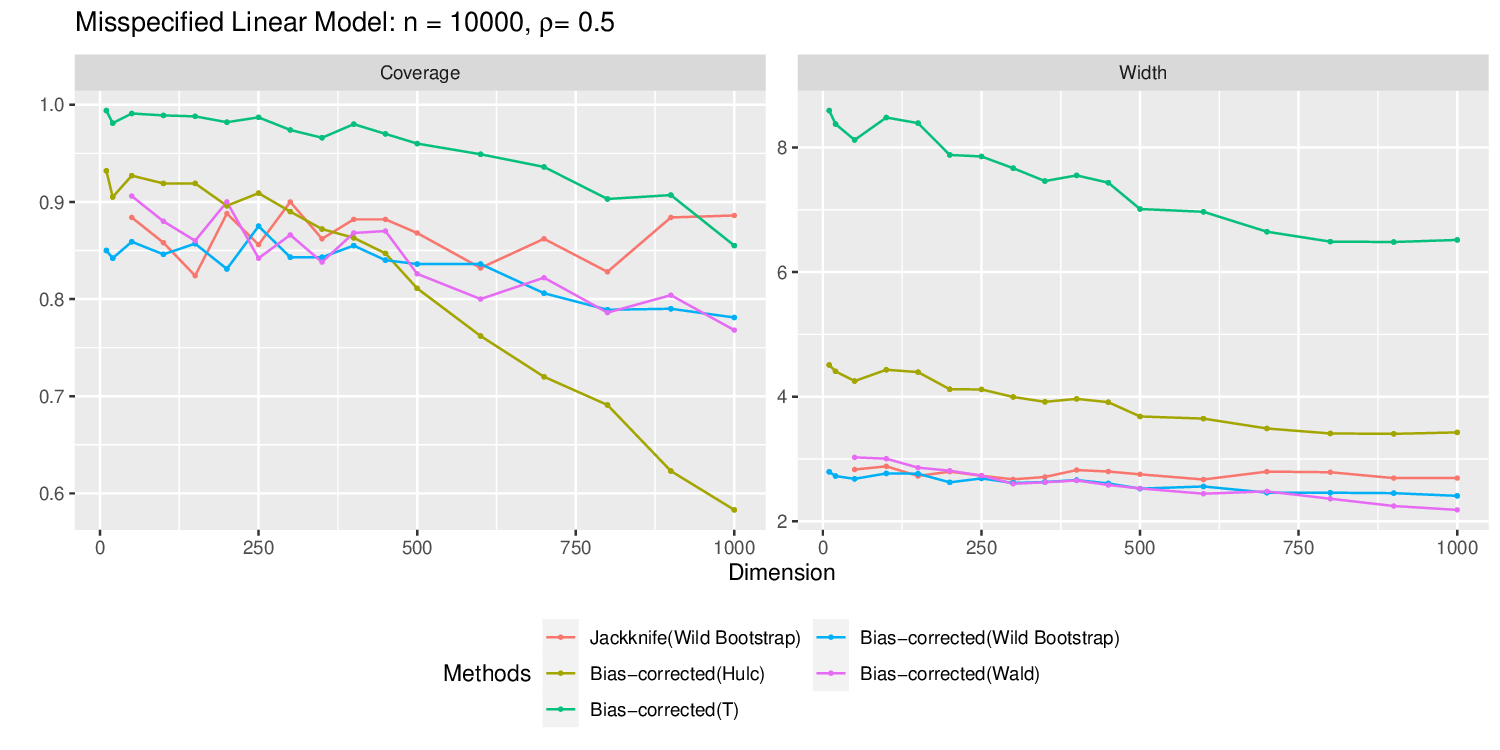}
         \caption{Proposed vs. Jackknife}
     \end{subfigure}\\
     \begin{subfigure}[b]{\textwidth}
         \centering
         \includegraphics[width=.7\textwidth]{Figures/CI_null_n=20000_rho=0.5_complex_large.eps}
         \caption{Proposed vs. Jackknife}
     \end{subfigure}\\
     
     \caption{Comparison of coverages and widths of confidence intervals for $\theta^\top\beta$ under the misspecified model with $n\in\set{5000, 10000, 20000}$, $\theta=\ev_1$, and $\rho=0.5$.}
     \label{fig:D.2.9}
\end{figure}

\end{document}